\documentclass[preprint,review,11pt,authoryear]{elsarticle}

\usepackage{amsmath,amssymb}
 \usepackage{amsthm}
 \usepackage{bbm}
 \usepackage{graphicx}
\usepackage{color}
\definecolor{cornell-red}{RGB}{179,27,27}
\usepackage{subcaption}
\usepackage{mathtools}
\usepackage{dsfont}
\usepackage{color}
\usepackage{natbib}
    \usepackage{multirow}
\usepackage{pgfplots}
\usepackage{caption,subcaption}
\usepackage[left=1in,right=1in,top=1in,bottom=1in]{geometry}
\usepackage{indentfirst}

\usepackage{xcolor}
\usepackage[colorlinks = true,linkcolor = blue,urlcolor  = blue,citecolor = blue,anchorcolor = blue]{hyperref}

\usepackage{booktabs}

\usepackage{natbib}

\graphicspath{ {Figures/} {Figures New/} }

\newtheorem{thm}{Theorem}
\newtheorem{lem}[thm]{Lemma}
\newtheorem{prop}[thm]{Proposition}
\newtheorem{cor}[thm]{Corollary}

\theoremstyle{definition}

\theoremstyle{remark}
\newtheorem{rem}{\textbf{Remark}}

\newcommand{\ra}[1]{\renewcommand{\arraystretch}{#1}}  

\newcommand{\E}{\mathbb{E}}
\newcommand{\Prob}{\mathbb{P}}
\newcommand{\Q}{\mathbb{Q}}
\newcommand{\R}{\mathbb{R}}
\newcommand{\Z}{\mathbb{Z}}
\newcommand{\calA}{\mathcal{A}}
\newcommand{\calF}{\mathcal{F}}

\newcommand{\calP}{\mathcal{P}}
\newcommand{\calS}{\mathcal{S}}
\newcommand{\calX}{\mathcal{X}}
\newcommand{\calY}{\mathcal{Y}}

\newcommand{\one}{\mathds{1}}
\newcommand{\norm}[1]{\left\lVert#1\right\rVert}
\newcommand{\norms}[1]{\lVert#1\rVert}

\newcommand{\xb}{\pmb{x}}
\newcommand{\ab}{\pmb{a}}
\newcommand{\bb}{\pmb{b}}
\newcommand{\yb}{\pmb{y}}
\newcommand{\db}{\pmb{d}}
\newcommand{\tb}{\pmb{t}}
\newcommand{\mub}{\pmb{\mu}}
\newcommand{\betab}{\pmb{\beta}}
\newcommand{\ub}{\pmb{u}}
\newcommand{\wb}{\pmb{w}}
\newcommand{\nub}{\pmb{\nu}}
\newcommand{\rhob}{\pmb{\rho}}
\newcommand{\alphab}{\pmb{\alpha}}
\newcommand{\gammab}{\pmb{\gamma}}
\newcommand{\deltab}{\pmb{\delta}}

\newcommand{\dlb}{\underline{d}}
\newcommand{\dub}{\overline{d}}
\newcommand{\tlb}{\underline{t}}
\newcommand{\tub}{\overline{t}}
\newcommand{\ulb}{\underline{u}}
\newcommand{\uub}{\overline{u}}
\newcommand{\nulb}{\underline{\nu}}
\newcommand{\nuub}{\overline{\nu}}

\newcommand{\alphalb}{\underline{\alpha}}
\newcommand{\alphaub}{\overline{\alpha}}
\newcommand{\deltalb}{\underline{\delta}}
\newcommand{\deltaub}{\overline{\delta}}
\newcommand{\gammalb}{\underline{\gamma}}
\newcommand{\gammaub}{\overline{\gamma}}
\newcommand{\rholb}{\underline{\rho}}
\newcommand{\rhoub}{\overline{\rho}}
\newcommand{\mud}{\mu^{\mbox{\tiny d}}}
\newcommand{\mut}{\mu^{\mbox{\tiny t}}}

\newcommand{\cw}{c^{\mbox{\tiny w}}}
\newcommand{\cu}{c^{\mbox{\tiny u}}}
\newcommand{\co}{c^{\mbox{\tiny o}}}
\newcommand{\mudb}{\mub^{\mbox{\tiny d}}}
\newcommand{\mutb}{\mub^{\mbox{\tiny t}}}

\newcommand{\xib}{\pmb{\xi}}
\newcommand{\dhat}{\hat{d}}
\newcommand{\that}{\hat{t}}
\newcommand{\xibhat}{\hat{\xib}}

\usepackage{titlesec}

    \titlespacing{\section}{0pt}{0ex}{0ex}
    \titlespacing{\subsection}{0pt}{0ex}{0ex}
    \titlespacing{\subsubsection}{0pt}{0ex}{0ex}
    
    \journal{arXiv}
    \date{ }
\begin{document}

\begin{frontmatter}

\title{Stochastic Optimization Models for a Home Service Routing and Appointment Scheduling Problem with Random Travel and Service Times}

\author[mymainaddress1]{Man Yiu (Tim) Tsang}
\cortext[cor1]{Corresponding author. }
 \ead{mat420@lehigh.edu}
\author[mymainaddress1]{Karmel S. Shehadeh\corref{cor1}}
 \ead{kas720@lehigh.edu}

 \address[mymainaddress1]{Department of Industrial and Systems Engineering, Lehigh University, Bethlehem, PA,  USA}

\begin{abstract}

\noindent  We study a routing and appointment scheduling problem with uncertain service and travel times arising from home service practice. Specifically, given a set of customers within a service region that an operator needs to serve, we seek to find the operator's route and time schedule. The quality of routing and scheduling decisions is a function of the total operational cost, consisting of customers' waiting time, and the operator's travel time, idle time and overtime. We propose and rigorously analyze a stochastic programming model and two distributionally robust optimization (DRO) models to solve the problem, assuming a known and unknown service and travel time distribution, respectively. We consider two popular types of ambiguity sets for the DRO models, namely, the mean-support and 1-Wasserstein ambiguity set. We derive equivalent mixed-integer linear programming (MILP) reformulations of both DRO models that can be implemented and efficiently solved using off-the-shelf optimization software, thereby enabling practitioners to use these models. In an extensive numerical experiment, we investigate the proposed models’ computational and operational performance, demonstrating where significant performance improvements can be gained with each proposed model and derive insights into the problem.



\end{abstract}
\end{frontmatter}

\section{Introduction}

The home service industry has been rapidly growing worldwide due to emergent changes in family structures, work obligations and extended work hours, aging population, and outspread of chronic diseases, among others.  The United States alone annually spends about $\$82$ billion on home healthcare services and $\$300$ billion on home repairs and maintenance \citep{Lunden, zhan2020home}. In 2018, the global home service market was valued at around $\$282$ billion and is expected to grow by $\$18.9$\% annually from 2019-2026, reaching $\$1133.4$ billion by 2026 (Verified Market Research, 2019). Therefore, the development of efficient home service scheduling and planning tools is essential to support decision-making in all areas of the home service industry. 


Home service planning process often includes the following steps. First, customers request their home services on their preferred dates through a booking system. Second, given each date’s requested services, a home service provider assigns the available providers to customers. Third, the home service provider sets the appointment time for each customer and determine each operator’s route. Fourth, the service provider communicates the appointment times to the customers. On the day of service, the operator executes the schedule and visits customers one by one. In this paper, we consider the perspective of a home service provider who needs to make the following planning decisions for each operator and assigned set of customers: (1) how to route the operator, and (2) how to assign appointment times to customers. Henceforth, for simplicity, we use the term ``\textit{routing}'' to refer to routes that specify the operator movement between customers or the sequence of visits on the operator’s schedule, and ``\textit{appointment time}'' to refer to the scheduled start time of the service of a customer. The set of appointment times of customers represent a schedule.


We formally refer to this problem as the home service operator routing and appointment scheduling (HRAS) problem. Specifically, given a set of customers within a service region that an operator needs to serve, the HRAS problem seeks to determine the operator's route and customers' appointment times. We consider both random customers' service time and travel time between customers. If the operator arrives at the customer's location before the scheduled service start time, then the operator needs to wait until the scheduled start time (i.e., remains idle). On the other hand, the customer needs to wait if the operator arrives after the scheduled start time. The operator has fixed service hours beyond which s/he experiences overtime. The quality of routing and scheduling decisions is a function of the total operational cost, consisting of customers' waiting time, and the operator's travel time, idle time, and overtime.

HRAS is a challenging optimization problem for several reasons. First, it is a complex combinatorial optimization problem that requires deciding the operator's route and assigning appointment times to each customer in the route simultaneously. Second, suppose we fix the appointment times. In this case, HRAS becomes similar to the traveling salesman problem (TSP) and the single-vehicle routing problem with time constraints, which are known to be challenging optimization problems (see \citealp{cook2011pursuit}). Third, for a fixed route, HRAS reduces to the well-known challenging  single-server stochastic appointment scheduling (AS) problem with random service time, which seeks a sequence of appointment times that minimizes customers' waiting times and the provider's idling time and overtime \citep{ahmadi2017outpatient, robinson2003scheduling}.  The rationale for considering travel time or travel distance as an additional optimization criterion in HRAS is to minimize service delivery or operational cost by controlling operator's amount of compensation for the distance or time traveled \citep{grieco2020operational,kandakoglu2020decision}.

One can use stochastic programming (SP) to model uncertain travel and service time. However, the applicability of SP is limited to the case in which we know the distribution of uncertainty or we have a large data sample. In most real-world applications, it is challenging to estimate the actual distribution of uncertain parameters accurately, especially with limited data during the planning process. Moreover, future uncertainty realizations may differ from all training samples. If we optimize according to a biased data sample from a misspecified distribution, the resulting optimal routing and scheduling decisions may have a disappointing out-of-sample performance (e.g., excessive overtime) under the true distribution.

Alternatively, one can construct an \textit{ambiguity set} (i.e., a family) of all possible distributions that share some partial information about the random parameter. Then, using the ambiguity set, one can formulate a distributionally robust optimization (DRO) problem to minimize the worst-case (i.e., maximum) expectation of second-stage random cost over all distributions residing within the ambiguity set.  Note that in DRO, the optimization is based on distributions within an ambiguity set, i.e., the distribution is a decision variable in DRO. DRO has recently received substantial attention in various application domains due to the following benefits. First, DRO models acknowledge and hedge against the presence of distributional uncertainty \citep{kuhn2019wasserstein}. Therefore, DRO solutions often faithfully anticipate the possibility of disappointing out-of-sample consequences and depending on the ambiguity set, they often guarantee an out-of-sample cost that falls below the worst-case optimal cost \citep{kuhn2019wasserstein,rahimi2020comprehensive}. 

Second, by allowing random parameters to follow a distribution defined in the ambiguity set, DRO alleviates the unrealistic assumption of the decision-maker's complete distributional knowledge. Third, intuitive and easy-to-approximate statistics could be used to construct ambiguity sets  that a decision-maker may approximate and change in the model. For example, decision-makers could estimate the average service and travel time based on their experience, a prediction model, or a  small  data set. Then, one could construct a mean-support ambiguity set on service and travel times where the support could represent the error margin and variability in the estimates that we seek protection against. Fourth, some DRO models of real-world problems are more tractable than their SP counterparts (see, e.g, \citealp{jian2017integer, saif2020data,  shehadeh2020distributionally,  shehadeh2020distributionallyTucker, wang2019distributionally, wang2020distributionally,  zhang2017distributionally}). 

The ambiguity set is a key ingredient of DRO models that should capture the true distribution with a high degree of certainty and be computationally manageable (i.e., allow for a tractable DRO model). There are several approaches to construct the ambiguity set. Most applied DRO literature employs moment-based ambiguity,  which incorporates all distributions sharing certain moments (e.g., mean). This is because many tools have been developed to derive tractable and solvable reformulations of moment-based DRO models. In most cases, moment-based DRO models do not scale with the number of scenarios and this leads to a more efficient computational approach. However, moment-based DRO models often do not have asymptotic properties  because the ambiguity sets only incorporates partial distributional information (i.e., moment information).

Recent DRO studies have shifted toward data-driven, distance-based DRO approaches such as $\phi$--divergence \citep{jiang2016data} and Wasserstein distance \citep{esfahani2018data, gao2016distributionally}, that construct ambiguity sets in the vicinity of a reference (e.g., empirical) distribution. One advantage of the Wasserstein ambiguity set is that it could incorporate small-sized data in the ambiguity set and  enjoys asymptotic properties \citep{kuhn2019wasserstein, esfahani2018data}. Recent results indicate that Wasserstein ambiguity set with carefully chosen radius may contain the unknown true distribution with a high probability and is richer than other divergence-based ambiguity sets.  Despite the potential advantages, to the best of our knowledge, there are no moment-based, Wasserstein-based, or other DRO approaches for the HRAS problem we study in this paper. This inspires us to consider this paper's main question: \textit{what are the computational and operational values of employing DRO models to address both service and travel time uncertainty, compared with the classical SP approach for HRAS}?. To answer this question, we design and rigorously analyze an SP and two DRO models based on the mean-support ambiguity and Wasserstein ambiguity sets and demonstrate where significant performance improvements can be gained with each proposed model and derive insights into the problem.



\subsection{Contributions}
In this paper, we address the uncertainty and distributional ambiguity of random travel and service times in HRAS. We summarize our main contributions as follows.

(1)  We propose the first two \textit{\underline{d}istributionally robust \underline{h}ome service \underline{r}outing \underline{a}nd \underline{s}cheduling} (DHRAS) models that seek optimal routing and scheduling decisions to minimize the worst-case sum of the operater's travel time, idle time, and overtime costs, and customers' waiting time costs. We consider two popular types of ambiguity sets for the DRO models, namely, the mean-support (M-DHRAS model) and 1-Wasserstein (W-DHRAS model) ambiguity set.  We also propose SP model for HRAS.


(2) We derive equivalent mixed-integer linear programming (MILP) reformulations of the min-max M-DHRAS and W-DHRAS models. The reformulations can be implemented and solved by off-the-shelf optimization software, thereby enabling practitioners to use these models.

(3) To investigate the value of the DRO approach for HRAS, we conduct an extensive numerical experiment comparing the proposed models operational and computational performance. Our results demonstrate that: (a) W-DHRAS yields robust decisions with both small and large data size, and enjoys both asymptotic consistency and finite-data guarantees; (b) M-DHRAS produces the most conservative schedules and it performs well when the distribution of travel time changes dramatically (e.g., actual travel times are longer); (c) W-DHRAS solutions have better operational performance than the SP solutions both under perfect and misspecified distributions, even when only a small data set is available; (d) DRO models produce more reliable solutions than the SP model; (e) the proposed SP and DRO models are computationally efficient  under realistic settings.

To the best of our knowledge, according to the recent review of operational research in home health care applications \citep{grieco2020operational} and our literature review in Section~\ref{sec2:Literature}, our paper is the first to attempt to investigate and compare theoretically and numerically DRO models for this specific HRAS problem. The remainder of this paper is organized as follows. In Section~\ref{sec2:Literature}, we review the relevant literature. In Section~\ref{sec3:Formulations}, we formally define the DHRAS models, present and analyze the M-DHRAS and W-DHRAS models. In Section~\ref{sec4:Exp}, we present computational results and managerial implications. Finally, we draw conclusions and discuss future directions in Section~\ref{sec5:conclusion}.


\section{Relevant Literature}\label{sec2:Literature}
In this section, we focus primarily on the literature mostly relevant to our problem, i.e., papers that apply stochastic optimization to address HRAS problems similar to ours.  For comprehensive recent surveys of operations research methods applied to decisions in home health care, we refer readers to \cite{fikar2017home,grieco2020operational, gutierrez2013home}. If we fix the appointment time, our problem reduces to the single-vehicle routing problem. That is, we need to design the optimal route of a single vehicle to visit all the customers by minimizing the costs of travel. Recent survey of the literature on these problems includes \cite{costa2019exact, oyola2018stochastic}. On the other hand, if we fix the route, our problem is similar to the single-server appointment scheduling problem. That is, we need to decide the optimal appointment time for the customers to minimize the total cost of customer waiting and server idling and overtime costs (see, e.g., \citealp{berg2014optimal, shehadeh2019analysis}). Appointment times can be fixed time points (see, e.g., \citealp{braekers2016bi, demirbilek2019dynamically, milburn2012operations}) or time windows (see, e.g., \citealp{lee2013continuous, mankowska2014home, yuan2015branch} and references therein). Herein, we focus on the former case under random service and travel times.  We refer to \cite{pinedo2016scheduling} for a comprehensive survey of scheduling theory, applications and methods.

SP approaches for HRAS include \cite{liu2019branch,yuan2015branch,zhan2018vehicle}. The recent work of  \cite{zhan2020home} is closely related to ours. Specifically, \cite{zhan2020home} consider a single operator HRAS problem in which they assume a \textit{deterministic travel time }between customers and that \textit{each customer's service time follows a fully known probability distribution}.  Accordingly,  \cite{zhan2020home} propose a two-stage stochastic mixed-integer linear program,  which seeks first-stage routing and scheduling decisions that minimize the operator traveling cost and an expectation of the second-stage cost of the operator idling and customer waiting. However, they do not incorporate the operator overtime cost in the objective, which is one important source of operational expenses. Note that ignoring random travel time may lead to sub-optimal solutions with, for example, excessive customer waiting time or provider overtime.
 
The SP approach assumes that the decision-maker is risk-neutral and knows the random parameters' distributions with certainty or can fully estimate them. In practice, it is implausible that decision-makers have sufficient high-quality data to infer the true distribution of random parameters, especially in health care applications. Given distributional ambiguity, if we calibrate a model to a misspecified (i.e., biased) distribution, the resulting optimal SP decisions may have a disappointing out-of-sample performance under the true distribution.  Various robust approaches have been proposed to model the risk-averse nature of decision-makers and uncertain parameters based on partial information of their distributions. The classical robust optimization (RO) approach assumes that uncertain parameters resided on an uncertainty set of possible outcomes with some structure (e.g., polyhedron, see, \citealp{ben2015deriving,bertsimas2004price,soyster1973convex}). Accordingly, optimization in RO is based on the worst-case scenario within the uncertainty set.

Recently, \cite{shi2019robust} propose an RO approach for home health care service.  They consider the problem with multiple caregivers and the goal is to find the optimal route for each caregiver and appointment time for each patient. They propose a robust model on both unknown service and travel time and discuss the heuristic solution approaches, with the objective function being the travel cost and delay cost with respect to the schedules. They find that, as compared with the SP model in \cite{shi2018modeling}, solutions from RO model has a better out-of-sample performance. For instance, the probability of visiting the customers timely is higher.  

By focusing the optimization on the worst-case scenario, classical RO approaches often yield over-conservative solutions and suboptimal decisions for other more-likely scenarios \citep{chen2020robust, delage2018value}. Distributionally robust optimization (DRO) is another approach for modeling uncertainty that bridges SP and RO by overcoming their shortcoming and has recently gained significant attention. In DRO, we model the distribution of uncertainty as a decision variable that belongs to an ambiguity set (i.e., a family of all possible distributions of uncertainty). We then optimize based on the distribution within this set \citep{rahimian2019distributionally}. 

 There are two common types of ambiguity sets: moment-based and distance-based ambiguity sets (see, e.g., \citealp{rahimian2019distributionally}). In this paper,  we address the distributional ambiguity of service time and travel time in the single-operator HRAS using these two types of ambiguity set, namely the M-DHRAS and W-DHRAS models, in the absence of abundant historical data. We reformulate them into MILPs, which could be solved by off-the-shelf optimization software. To the best of our knowledge and according to the recent surveys \cite{fikar2017home,grieco2020operational}, there is no DRO approach for the specific HRAS problem that we address.  Moment-based DRO models for single-server appointment scheduling includes \citep{jiang2017integer, kong2013scheduling, mak2015appointment}.   Recently, \cite{jiang2019data} propose the distance-based DRO model based on the Wasserstein metric for single-server appointment scheduling for a fixed sequence of customers with random service time and no-show. They derive tractable reformulations of their model under the 1--Wasserstein and 2--Wasserstein ball. 
 
We compare our work with \cite{jiang2019data} and \cite{zhan2020home}, which are recent single-server studies relevant to our work (see Table \ref{table:literature_comparison} in \ref{Appx:LR_Compare} for a quick comparison). First, the three papers address stochastic service time. Second, \cite{jiang2019data} study a classical single-server appointment scheduling problem in which they need to determine the appointment time for a fixed sequence of customers. The server does not need to travel to different customers. Hence, no routing decisions or travel time consideration is necessary. In other words, \cite{jiang2019data}'s model cannot be used for routing and appointment scheduling. Third, \cite{zhan2020home} assume that the travel time is deterministic, which as mentioned earlier, may lead to sub-optimal and unrealistic solutions. In addition,  \cite{zhan2020home} do not address the potential distributional ambiguity of service time while \cite{jiang2019data} address this.  In contrast to \cite{zhan2020home} we consider both stochastic service and travel times. Moreover, we address both uncertainty and distributional ambiguity of these parameters via SP and DRO. Note that if we fix the travel time as zero, our model reduces to a scheduling model. Thus, in this case, our models generalize that of \cite{jiang2019data} by incorporating sequencing decisions and hence, can be used for sequencing and scheduling. Various appointment scheduling studies \citep{ahmadi2017outpatient, berg2014optimal, cayirli2006designing, cayirli2008assessment, creemers2012optimal,  shehadeh2019analysis} demonstrate the benefit of sequencing customers appointments based on their characteristics for improving performance and reducing costs as compared with fixed sequence approach that \cite{jiang2019data} adopt.

\section{Formulation and Analysis}\label{sec3:Formulations}
In this section, we formulate the DHRAS problem using moment  (M-DHRAS) and Wasserstein  (W-DHRAS) ambiguity sets. We define DHRAS and provide a two-stage stochastic programming formulation in Section~\ref{sec3:Defintions}. We analyze and derive equivalent mixed-integer linear programming reformulations of M-DHRAS in Section~\ref{MDHRAS} and W-DHRAS in Section~\ref{WDHRAS}. 

\vspace{1mm}
\noindent \textbf{Notation:} For $a,b \in \Z$, we define $[a]:=\{1,2,\dots, a \}$ and $[a,b]_\Z:=\{ c \in \Z: a \leq c \leq b \}$, i.e., $[a,b]_\Z$ represent the set of running integer indices $\{a, a+1, a+2, \dots, b \}$. For a real number $a$, we define $(a)^+=\max\{a,0\}$. We use boldface notation to denote vectors, e.g., $\db:=[d_1, d_2, \ldots, d_N]^\top$.  

\begin{table}[t]  
\small
\center
\renewcommand{\arraystretch}{0.5}
\caption{Notation.} 
\begin{tabular}{ll}
\hline 
 \multicolumn{2}{l}{\textbf{Indices}} \\
$i$ & index of customer, $i=1,\dots,N$ \\
$j$ & index of service position, $j=1,\dots,N$ \\
\multicolumn{2}{l}{\textbf{Parameters and sets}} \\
$N$ & number of customers \\
$L$ & standard work time \\
$\cu_j$ & idling cost associated to the early arrival at the $j$th customer's location\\
$\cw_j$ & waiting cost of the $j$th customer \\
$\co$ & overtime cost\\
$\lambda$ & travel time cost \\
$d_i$ & service time of customer $i$ \\
$t_{i,i'}$ & travel time from customers $i$ to $i'$ (with $0$ denoting the depot)\\
$\dlb_i$/$\dub_i$ & lower/upper bound of the service time of customer $i$ \\
$\tlb_{i,i'}$/$\tub_{i,i'}$ & lower/upper bound of the travel time from customers $i$ to $i'$ \\ 
\multicolumn{2}{l}{\textbf{First-stage decision variables}} \\
$x_{i,j}$ & binary variable equal to $1$ if customer $i$ is the $j$th customer served and $0$ otherwise \\
$a_j$ & appointment time of the $j$th customer\\
\multicolumn{2}{l}{\textbf{Second-stage decision variables}} \\
$u_j$ & idling time due to an early arrival at the $j$th customer's location \\
$w_j$ & waiting time of the $j$th customer ($w_{N+1}$ as the overtime)\\
\hline
\end{tabular}\label{table:notation}
\end{table}

\subsection{\textbf{Definitions and assumptions}}\label{sec3:Defintions}

We consider a set of $N$ customers that need to be served within a given day by a single home service provider (operator). Traveling time, $t_{i,i'}$, between each pair of customers $i$ and $i'$ ($i,i'=0, \dots, N$) is random, with $i=0$ representing the service provider's office (depot).  Service time $d_i$ of each customer $i \in [N]$ is also random. Note that we may only have partial information on the distributions or possibly small data available on these random parameters. We seek two sets of decisions:  (1) service provider's visiting sequence (routes), (2) customers appointment times (schedule).  The objective is to minimize the sum of customers' waiting time, and the provider's traveling time, idling time, and overtime. 

For all $i \in [N]$ and $j \in [N]$, we let the binary decision variable $x_{i,j}$ equal 1 if customer $i$ is the $j$th customer in the operator's route/schedule, and zero otherwise.  For all $j \in [N]$, we let the continuous variable $a_j$ represent the appointment starting time of the $j$th customer.  The feasible region of variable $\xb$ is defined in \eqref{eq:RegionX} such that each customer is assigned to one position in the operator's route/schedule, and each position is assigned to one customer. The feasible region of  $\ab$ is defined in \eqref{eq:RegionA} such that all appointments are scheduled within the provider's service hours [$0, L$].
\allowdisplaybreaks
\begin{align}
\calX&=\left\{ \xb:  \begin{array}{l} \sum_{i=1}^N x_{i,j} =1, \  \forall j \in [N] \\ 
 \sum_{j=1}^N x_{i,j} =1,  \ \forall  i \in [N] \\ 
x_{i,j} \in \{0, 1 \}, \  \forall i \in [N], \ \forall j \in [N] \end{array} \right\}  \label{eq:RegionX}
\end{align} 
\begin{align}
 \calA&= \left\{ \ab: \begin{array}{l}  
 \  0 \leq a_j \leq L, \  \ \forall j \in [N] \\
  \ a_{j} \geq a_{j-1}, \ \forall j \in [2, N]_\Z \end{array} \right\} \label{eq:RegionA}
\end{align} 

Due to random travel and service times, one or multiple of the following scenarios may happen: (1) operator arrives at the customer's location before the scheduled service start time ($a_j$), and thus s/he remains idle until the scheduled start time; (2)  the operator arrives after the scheduled start time of the customer, and thus the customer incurs waiting cost; (3) provider works overtime beyond his/her scheduled $L$ to finish all appointments. Let the continuous decision variable $w_j$ represent the waiting time of the $j$th customer, for all $j \in [N]$ and $w_{N+1}$ represent the operator's overtime. For all $j \in [N]$, let the continuous decision variable $u_j$ represent the provider's idle time before the start time of the $j$th appointment. For all $i \in [N]$, let the random parameter $d_i$ represent the service duration of customer $i$. For all $i \in [N]$ and $i' \in [N]$, let the random parameter $t_{i,i'}$ represent the travel time between $i$ and $i'$. Given a fixed $\xb \in \calX$,  $\ab \in \calA$ and a joint realization $\xib:=[\tb, \db ]^\top$, we can compute the operational costs (as a function of travel time, idle time, and overtime, and customers waiting time), using the following linear program:
\allowdisplaybreaks
\begin{subequations}\label{LP:Sec}
\begin{align}
  f(\xb, \ab, \xib ):=  & \min_{\ub,\,\wb} \ \Bigg \{  \sum \limits_{j=1}^N (\cw_j w_j+\cu_j u_j \big) +\co w_{N+1} + \lambda A \Bigg \}  \label{2nd:Obj}\\
& \ \text{s.t.}  \ \ w_1-u_1= \sum \limits_{i=1}^N t_{0,i}x_{i,1}-a_1, \label{2nd:C1}\\
&  \ \ w_j-w_{j-1}-u_j=a_{j-1}-a_j+\sum_{i=1}^N d_ix_{i,j-1}+ \sum_{i=1}^N \sum_{i'\neq i} t_{i,i'} x_{i,j-1}x_{i,j}, \ \forall j\in [2, N]_\Z, \label{2nd:C2}\\
& \ \   w_{N+1}-w_N-u_{N+1}= a_N-a_{N+1}+\sum_{i=1}^N d_ix_{i,N}, \label{2nd:C3}\\
& \ \ A= \sum \limits_{j=2}^N   \sum_{i=1}^N \sum_{i'\neq i} t_{i,i'} x_{i,j-1}x_{i',j}+\sum_{i=1}^N(t_{0,i}x_{i,1}+t_{i,0}x_{i,N}), \label{2nd:C5}\\
& \ \ (w_j,u_j)\geq 0, \ \ \forall j \in [N+1], \label{2nd:C4}
\end{align} \label{2ndStage}%
\end{subequations}
where $\cw_j$, $\cu_j$, $\co$, and $\lambda$ are respectively the non-negative unit penalty costs of waiting, idling,  overtime, and travel time for all $j\in [N]$.  Also, we let $a_{N+1}=L$ and $x_{i,N+1}=0$ for all $i\in[N]$. The objective function \eqref{2nd:Obj} minimizes a linear cost function of waiting, idling, overtime, and travel time. Constraint \eqref{2nd:C1}  yields either provider's idle time before the scheduled time of the first appointment or the waiting time of the first customer. Constraint \eqref{2nd:C2} yields either the waiting time of the $j$th customer or the provider's idle time if s/he arrives at the $j$th customer, respectively, after or before the scheduled start time $a_j$ of the $j$th customer.  Constraints \eqref{2nd:C3} yields either the overtime or the schedule earliness. Constraint \eqref{2nd:C5} computes the operator's total travel time.

The SP model in \eqref{SP} seeks to find $(\xb,\ab)\in\calX\times\calA$ that minimizes the expectation of the cost $f(\xb,\ab, \xib$) subject to random $\xib$ with a known joint probability distribution $\Prob_{\xib}$.
\begin{align}
& Z^*:= \min \limits_{\xb \in \calX, \ab \in \calA} \E_{\Prob_{\xib}} [f(\xb,\ab, \xib) ]  \label{SP}
\end{align}

\subsection{\textbf{DHRAS over moment ambiguity (M-DHRAS)}}\label{MDHRAS}

In this section, we consider the case when $\Prob_{\xib}$ is not fully known or hard to estimate. We assume that we know or can approximate the mean values $\E_{\Prob_{\xib}}(\xib):=\mub=[\mudb, \mutb]^\top$, lower bounds $[\underline{\db},\underline{\tb}]^\top $ and upper bounds $[\overline{\db},\overline{\tb}]^\top$ of $[\db,\tb]^\top$. Mathematically, we consider the support $\calS= \calS^{\mbox{\tiny d}} \times \calS^{\mbox{\tiny t}}$ of $\xib$, where $\calS^{\mbox{\tiny d}}$ and $\calS^{\mbox{\tiny t}}$ are respectively the supports of random parameters $\db$ and $\tb$ defined in \eqref{Support}. 
\begin{subequations} \label{Support}
\begin{align}
& \calS^{\mbox{\tiny d}}:=\left\{\, \db  \geq 0: \begin{array}{l} \dlb_{i} \leq d_{i} \leq \dub_{i},  \ \forall i \in [N]\end{array} \right\}\,,\\
& \calS^{\mbox{\tiny t}}:=\left\{\, \tb  \geq 0: \begin{array}{l} \tlb_{i, i'} \leq t_{i, i'} \leq \tub_{i,i'},  \ \forall i \in [0,N],\, i'\in[0,N],\, i' \neq i\end{array} \right\}\,.
\end{align}
\end{subequations}
The assumption of known mean and support is motivated by the fact that experts may be able to approximate the average service or travel time based on their experience (or from prediction models). The range can represent the error margin of the mean values estimates and variability of the parameters which we seek protection against. Experts may also be able to provide an upper and lower bound on service time. Therefore, we can use such information to define a mean-support ambiguity set of uncertain parameters distributions. Using $\mub$ and $\calS$, we consider the following mean-support ambiguity set  $\calF(\calS, \mub)$:
\begin{align}\label{eq:Mambiguity}
\calF(\calS, \mub) := \left\{ \Prob \in \calP(\calS): \begin{array}{l} \int_\calS d\Prob = 1\\ \mathbb{E_P}(\xib) = \mub \end{array} \right\},
\end{align}
where $\calP(\calS)$ represents the set of all probability distributions supported on $\calS$. Using ambiguity set $\calF(\calS,\mub)$, we formulate M-DHRAS as follows:
\begin{align}
(\text{M-DHRAS}) &\ \ \ \min \limits_{\xb \in \calX, \ab \in \calA} \bigg \{ \sup_{\Prob \in \calF(\calS,\mub) }\E_{\Prob} [f(\xb,\ab, \xib) ] \bigg\}. \label{M-DHRAS}
\end{align}
Formulation \eqref{M-DHRAS} seeks first stage decisions ($\xb, \ab$) that minimizes the worst-case (maximum) expectations of the second-stage operational cost over $\calF(\calS, \mub)$.

\begin{rem} 
Model \eqref{M-DHRAS} could hedge against distributional uncertainty, particularly when limited (or no) information on random parameters is available. The use of mean-support ambiguity set is prevalent due to the following reasons. First, only intuitive statistics, namely the mean and the range, are required. They can be interpreted easily as mean represents the distribution centrality while the range represents the dispersion of the distribution. Decision-makers could approximate these values in various ways (e.g., estimation from the limited available data or from expert knowledge). Second, mathematically, it allows us to derive a tractable MILP reformulation. Various studies have demonstrated that incorporating higher moments or random parameters often undermines the computational tractability of the DRO model due to the non-linearity in higher moments.

\end{rem}

\subsubsection{\textbf{MILP reformulation of M-DHRAS}}\label{sec:ReforM_DHRAS}

Note that model \eqref{M-DHRAS} is not directly solvable in the presented form due to the minimax structure. In this section, we use the recourse problem properties to derive an equivalent MILP reformulation of model \eqref{M-DHRAS}. First, for a fixed $(\xb,\ab)\in\calX\times\calA$, we rewrite the inner maximization problem $\sup_{\Prob \in \calF(\calS,\mub) }\E_{\Prob} [f(\xb,\ab, \xib) ]$ as follows: 
\begin{subequations}\label{MDHRAS:InnerMax}
\begin{align}
& \max   \ \mathbb{E_P}[f(\xb,\ab, \xib)]  \\
& \ \text{s.t.}  \ \ \mathbb{E_P}(\xib) = \mub, \label{MDHRAS::InnerMax-C1}  \\
& \ \ \ \ \  \ \ \mathbb{E_P}[\one_\calS(\xib)] = 1, \label{MDHRAS::InnerMax-C2}
\end{align} 
\end{subequations}
where $\one_\calS(\xib)= 1$ if $\xib \in \calS$ and $\one_\calS(\xib) = 0$ if $\xib \notin \calS$.  As we show in the proof of Proposition~\ref{Prop1M:DualMinMax} in \ref{Appx:ProofOfProb1M}, problem \eqref{MDHRAS:InnerMax} is equivalent to the deterministic problem \eqref{eq:FinalDualInnerMax-1}.

\begin{prop}\label{Prop1M:DualMinMax}
For any $(\xb,\ab)\in\calX\times\calA$, problem \eqref{MDHRAS:InnerMax} is equivalent to
\begin{align} \label{eq:FinalDualInnerMax-1}
&  \min_{\alphab,\,\rhob}  \ \Bigg \{ \sum \limits_{i=1}^N \mu_i^{\mbox{\tiny d}} \rho_i+ \sum_{i=1}^N \sum_{i'=1}^N \mu_{i,i'}^{\mbox{\tiny t}}  \alpha_{i,i'} +\max \limits_{\xib \in \calS  }  \Bigg \{  f(\xb,\ab, \xib)- \sum \limits_{i=1}^N d_i \rho_i- \sum_{i=1}^N \sum_{i'=1}^N t_{i,i'} \alpha_{i,i'}  \Bigg \}  \Bigg\} \\
& \ \ \textup{s.t.} \ \ \alphab\in\R^{(N+1)\times(N+1)},\, \rhob\in\R^N. \nonumber
\end{align}
\end{prop}

 Note that $ f(\xb,\ab, \xib)$  is a minimization problem, and thus in \eqref{eq:FinalDualInnerMax-1} we have an inner max-min problem, which is not suitable to solve using standard solution methods. For a given solution $(\xb,\ab)$ and realized value of $\xib$, $ f(\xb,\ab, \xib)$ is a linear program (LP). The dual of $  f(\xb,\ab, \xib)$ is as follow
%
\begin{subequations}\label{Dual:Sec}
\begin{align}
  f(\xb, \ab,\xib ):=  & \max_{\yb} \Bigg \{ \Bigg(\sum_{i=1}^N t_{0,i}x_{i,1}-a_1\Bigg)y_1+  \sum_{j=2}^{N+1} \Bigg ( a_{j-1}-a_j+\sum_{i=1}^N d_ix_{i,j-1}+ \sum_{i=1}^N \sum_{i'\neq i} t_{i,i'} x_{i,j-1}x_{i,j}\Bigg)y_j  \nonumber\\
  & \qquad \ \ \  \ + \lambda \Bigg[\sum \limits_{j=2}^N   \sum_{i=1}^N \sum_{i'\neq i} t_{i,i'} x_{i,j-1}x_{i',j}+\sum_{i=1}^N\Big(t_{o,i}x_{0,i}+t_{i,0}x_{i,N}\Big) \Bigg] \Bigg \}  \label{Dual2nd:Obj}\\
  & \text{s.t.}  \ \  \yb\in\calY:= \Big\{\yb \mid  0\leq y_{N+1} \leq \co, \ -\cu_j\leq y_j \leq \cw_j+ y_{j+1}, \ \forall j \in [N]\Big\}, \label{Dual2nd:C1}
\end{align} \label{Dual2ndStage}
\end{subequations}
\hspace{-1.2mm}where $y_1, \ldots, y_{N+1}$ are the dual variables associated with constraints \eqref{2nd:C1}--\eqref{2nd:C3}.  Given $(\xb,\ab)\in\calX\times\calA$ and $\xib$, the objective function in \eqref{Dual2nd:Obj} is linear (convex) in $\yb$. Hence, the inner maximization problem in \eqref{Dual:Sec} is a convex maximization problem. It follows from the fundamental convex analysis that it suffices to consider the extreme points of polytope $\calY$. This motivates us to leverage the properties of the extreme points of $\calY$ in deriving an equivalent LP reformulation of \eqref{Dual:Sec}. We formally prove this in Proposition~\ref{Prop2M} (see \ref{Appx:ProofProp2M} for a detailed proof).

\begin{prop}\label{Prop2M}
For any $(\xb,\ab)\in\calX\times\calA$, the inner maximization problem in \eqref{eq:FinalDualInnerMax-1} is equivalent to .
\begin{subequations}\label{DHRASMINLP}
\begin{align} 
 \min_{\betab,\,\gammab,\,\deltab}  &\  \sum \limits_{j=1}^{N+2} \beta_j  +\sum_{i=1}^N\left[(\lambda-\alpha_{i,0})\tlb_{i,0}+\Delta t_{i,0}(\lambda-\alpha_{i,0})^+\right]x_{i,N}  \label{DHRAS_MINLP_Obj}\\
\textup{s.t.} & \ \beta_1 \geq -a_1  \pi_{1,1} +\sum_{i=1}^N \big [\underline{t}_{0,i} (\pi_{1,1}+\lambda-\alpha_{0,i})+\Delta t_{0,i} \gamma_{0,i,1,1}\big] x_{i,1} , \label{DHRASMINLP:Con1}\\
& \ \sum_{j=1}^v \beta_j \geq -a_1  \pi_{1,v} + \sum_{j=2}^{\min(v,N+1)} \big ( a_{j-1}-a_j \big) \pi_{j,v}   + \sum_{i=1}^N \big [\underline{t}_{0,i} (\pi_{1,v}+\lambda-\alpha_{0,i})+\Delta t_{0,i} \gamma_{0,i,1,v}\big] x_{i,1}  \nonumber \\
& \ \qquad \   + \sum_{j=2}^{\min(v,N+1)} \sum_{i=1}^N \sum_{i' \neq i} \big[ \underline{t}_{i,i'} (\pi_{j,v}+\lambda-\alpha_{i,i'})+ \Delta t_{i,i'} \gamma_{i,i',j,v}  \big]x_{i,j-1}x_{i', j} \nonumber\\ 
& \ \qquad \   + \sum_{j=2}^{\min(v,N+1)} \sum_{i=1}^N \big[ \underline{d}_i (\pi_{j,v}-\rho_i)+ \Delta d_i \delta_{i,j,v} \big ] x_{i,j-1}, \ \
 \forall v \in [2, N+2]_{\Z} ,  \label{DHRASMINLP:Con2} \\
& \ \sum_{j=k}^v\beta_j \geq \sum_{j=k}^{\min(v,N+1)} \big ( a_{j-1}-a_j\big)  \pi_{j,v}+\sum_{j=k}^{\min(v,N+1)} \sum_{i=1}^N \big[ \underline{d}_i (\pi_{j,v}-\rho_i)+ \Delta d_i \delta_{i,j,v} \big ] x_{i,j-1} \nonumber\\
& \ \qquad \    \ + \sum_{j=k}^{\min(v,N+1)} \sum_{i=1}^N \sum_{i' \neq i}  \bigg\{ (\pi_{j,v}+\lambda-\alpha_{i,i'})\underline{t}_{i,i'}+\Delta t_{i,i'}\gamma_{i,i',j,v}  \bigg\}x_{i,j-1}x_{i',j},   \nonumber \\
& \ \qquad \    \  \forall k\in[2,N+1]_\Z,\,\forall v\in[k,N+2]_\Z , \label{DHRASMINLP:Con3} \\
& \ \beta_{N+2} \geq 0 , \label{DHRASMINLP:Con4} \\
& \ \gamma_{0,i,1,v}\geq 0,\quad \gamma_{0,i,1,v}\geq \pi_{1,v}+\lambda-\alpha_{0,i},\quad \forall i\in[N],\,\forall v\in[N+2] ,\label{DHRASMINLP:Con5} \\
& \ \gamma_{i,i',j,v}\geq 0,\quad \gamma_{i,i',j,v}\geq \pi_{j,v}+\lambda-\alpha_{i,i'},\quad \forall i\in[N],\,\forall i'\in[N]\setminus\{i\}, \nonumber \\
& \hspace{70mm} \forall j\in[2,N+1]_\Z,\,\forall v\in[j,N+2]_\Z , \label{DHRASMINLP:Con6}  \\
& \ \delta_{i,j,v}\geq 0,\quad \delta_{i,j,v}\geq \pi_{j,v} -\rho_i,\quad\forall i\in[N],\,\, \forall j\in[2,N+1]_\Z,\,\forall v\in[j,N+2]_\Z  ,\label{DHRASMINLP:Con7}
\end{align}
\end{subequations}
where $\pi_{j,v}=-c_v^u+\sum_{l=j}^{v-1}c_l^w$ for $1\leq j\leq v\leq N+2$, Let $\Delta d_i=\dub_i-\dlb_i$, and $\Delta t_{i,i'}=\tub_{i,i'}-\tlb_{i,i'}$.
\end{prop} 

With the use of Proposition \ref{Prop2M}, we can combine the minimization over $\alphab$ and $\rhob$ in problem \eqref{eq:FinalDualInnerMax-1}. Corollary $\ref{Coro3}$ provides the reformulation of problem \eqref{eq:FinalDualInnerMax-1}  (see \ref{apdx: Coro_proof} for a proof).
\begin{cor} \label{Coro3}
Problem \eqref{eq:FinalDualInnerMax-1} is equivalent to
\begin{subequations}\label{InnerMINLPnew}
\begin{align} 
 \min_{\alphab,\,\rhob,\,\betab,\,\gammab,\,\deltab} & \  \sum \limits_{i=1}^N \mu_i^{\mbox{\tiny d}} \rho_i+ \sum_{j=2}^N\sum_{i=1}^N \sum_{i'\ne i} \mut_{i,i'}  \alpha_{i,i'}x_{i,j-1}x_{i',j}  + \lambda \sum \limits_{i=1}^N \mut_{i,0}  x_{i,N} + \sum \limits_{i=1}^N \mut_{0,i} \psi_{0,i}  + \sum \limits_{j=1}^{N+2} \beta_j  \label{InnerMINLPnew:Obj}\\
\textup{s.t. \hspace{3mm}} & \ \alphab\in\R^{(N+1)\times(N+1)},\, \rhob\in\R^N,\, \betab\in\R^{N+2},\, \eqref{DHRASMINLP:Con1} - \eqref{DHRASMINLP:Con7}
\end{align}
\end{subequations}
\end{cor}

Combining  $\sup_{\Prob \in \calF(\calS,\mub) }\E_{\Prob} [f(\xb,\ab, \db) ]$  in the form of \eqref{InnerMINLPnew} with the outer minimization problem in \eqref{M-DHRAS}, we obtain a mixed-integer nonlinear program (MINLP).  Due to its large size and page limitation, we present this MINLP in \ref{Appx:Mac_MDHRAS}, where we also derive the following exact and tight MILP reformulation of this this MINLP (equivalently, M-DHRAS model).
\allowdisplaybreaks
\begin{subequations}\label{M_DHRAS_L}
\begin{align} 
 \min & \  \sum \limits_{i=1}^N \mud_i \rho_i+ \sum_{j=2}^N\sum_{i=1}^N \sum_{i'\ne i} \mut_{i,i'} \eta_{i,i',j-1,j}  + \lambda \sum \limits_{i=1}^N \mut_{i,0} x_{i,N} + \sum \limits_{i=1}^N \mut_{0,i} \alpha_{0,i} x_{i,1}  + \sum \limits_{j=1}^{N+2} \beta_j  \label{M_DHRAS_L:obj}\\
 \text{s.t.}  & \ \xb \in \calX, \ \ab \in \calA,\, \alphab\in\R^{(N+1)\times(N+1)},\, \rhob\in\R^N,\, \betab\in\R^{N+2}, \\ 
& \ \text{constraints }\eqref{M-DHRAS_MILP_Con2}-\eqref{M-DHRAS_MILP_Con5},\,\eqref{DHRASMINLP:Con5}-\eqref{DHRASMINLP:Con7},\,\eqref{MILP-MAC1M}-\eqref{MILP-MAC11M}.
\end{align}
\end{subequations}

\subsection{\textbf{DHRAS over 1-Wasserstein ambiguity (W-DHRAS)}}\label{WDHRAS}

In this section, we consider the case that $\mathbb{P}_{\xib}$ may be observed via a  small finite set $\big\{\hat{\xib}^1, \ldots, \hat{\xib}^R\big\}$ of $R$ i.i.d. samples, which may come from the limited historical realizations or a reference empirical distribution. Note that the empirical distribution defined on this data (or any other reference distribution) could serve as an estimator of the true distribution. To hedge against the estimation error, we robustify the optimization problem against all distributions $\Prob_{\xib}$ close to the empirical distribution $\Prob_{\xib}^R$.  Specifically, we construct a set of all distributions $\mathcal{P}(\mathcal{S})$ supported on the support $\calS$ of the unknown true distribution for which all distributions $\Prob_{\xib}$ have 1-Wasserstein distance that is less than or equal to $\epsilon$ from the reference distribution. We use 1-Wasserstein distance (i.e., we use $\ell_1$--norm in the definition of Wasserstein metric) because it often admits tractable reformulation in most real-world applications (see, e.g., \citealp{Daniel2020, hanasusanto2018conic, jiang2019data, saif2020data}). Next, we present our W-DHRAS model.
 
  Suppose that two probability distributions $\Q_1$ and $\Q_2$ are defined on a common support $\calS \in \mathbb{R}^N$, and let $|| \cdot ||_p$ represent the $p$-norm on $\mathbb{R}^N$ with $p  \geq 1$. Suppose that random vectors $\xib_1$ and $\xib_2$ follow $\Q_1$ and $\Q_2$ respectively. Then, the Wasserstein distance between $\Q_1$ and $\Q_2$, denoted as $W_p(\Q_1, \Q_2)$, represents the minimum transportation cost of moving from $\Q_1$ to $\Q_2$, where the cost of moving from $\xib_1$ to $\xib_2$ is measured by the norm $\norms{\xib_1-\xib_2}_p$. Mathematically, 
\begin{equation}
W_p(\Q_1, \Q_2):= \bigg ( \inf_{\Pi \in \mathcal{P}(\Q_1, \Q_2)} \E_{\Pi} \big[ \norm{\xib_1 -\xib_2}_p^p\big]  \bigg)^\frac{1}{p},
\end{equation}
where $\mathcal{P}(\Q_1, \Q_2) $ is the set of all joint distributions of ($\xib_1$, $\xib_2$) with marginals $\Q_1$ and $\Q_2$. Since we only observe a  set $\{\hat{\xib}^1, \ldots, \hat{\xib}^R\}$ of $R$ i.i.d. samples, we consider the following $p$-Wasserstein ambiguity set 
%
\begin{align}\label{c}
\calF_p (\hat{\Prob}_{\xib}^R, \epsilon)= \left\{ \Q_{\xib} \in \mathcal{P}(\calS): W_p(\Q_{\xib}, \hat{\Prob}_{\xib}^R) \leq \epsilon\right\},
\end{align}
where $ \mathcal{P}(\calS)$ is the set of all probability distributions on $\calS$, $\hat{\Prob}_{\xib}^R=\frac{1}{R} \sum_{r=1}^R \delta_{\hat{\xib}^r}$ is the empirical distribution of $\xib$ based on the $R$ i.i.d samples with $\delta$ being the Dirac measure, and $\epsilon >0$ is the radius of the ambiguity set. The set $\calF_p (\hat{\Prob}_{\xib}^R, \epsilon)$ can be viewed as the $p$-Wasserstein ball of radius $\epsilon$ centered at the empirical distribution $\hat{\Prob}_{\xib}^R$. Therefore, in some sense, one can think Wasserstein ball as the set of all distributions under which our estimation error is below $\epsilon$, where $\epsilon$ is the estimation error we seek protection against. A larger radius $\epsilon$ indicates that we seek more robust solutions. Using the ambiguity set $\calF_p(\hat{\Prob}_{\xib}^R, \epsilon)$, we formulate W-DHRAS as follows:
\begin{align}
(\text{W-DHRAS}) \ \ \ \hat{Z}(R, \epsilon)&= \min \limits_{\xb \in \calX, \ab \in \calA} \Bigg \{ \sup_{\Q_{\xib} \in \calF_p (\hat{\Prob}_{\xib}^R, \epsilon) }\E_{\Prob_{\xib}} \big[f(\xb,\ab, {\xib}) \big] \Bigg \}. \label{W-DHRAS}
\end{align}

In data-driven approaches such as W-DHRAS, we often seek asymptotic consistency. Specifically, we expect that as the sample size $R$ increases to infinity, the optimal value of the problem \eqref{W-DHRAS} $\hat{Z}(R, \epsilon)$ converges to $Z^*$ (the optimal value of the  SP model in \eqref{SP} with perfect knowledge of $\Prob_{\xib}$), and an optimal solution $(\xb,\ab)$ of  (W-DHRAS) converges to an optimal solution of problem \eqref{SP}. Additionally, if $\hat{Z}(R, \epsilon)>Z^*$ almost surely, then W-DHRAS provides a safe upper bound guarantee on the expected total cost with any finite data size $R$. Recall that our support set $\calS$ is non-empty, convex and compact. As such, we can use the existing theory in establishing the asymptotic consistency and finite sample guarantee of our W-DHRAS.  We refer the readers to  \ref{Proof_Lemma1}--\ref{Proof_Thrm2} for adapted proofs of these results.



\subsubsection{\textbf{MILP reformulation of W-DHRAS}}
In this section, we derive an equivalent MILP reformulation of our W-DHRAS model. First, we consider the inner maximization problem of W-DHRAS for a fixed $(\xb,\ab)\in\calX\times\calA$.
\begin{align}
\sup_{\Q_{\xib}\in \calF_1 (\hat{\Prob}_{\xib}^R, \epsilon) }\E_{\Q_{\xib}} [f(\xb,\ab, \xib) ] \label{InnerMaxW}
\end{align}
 In Proposition~\ref{Prop1}, we present an equivalent dual formulation of  \eqref{InnerMaxW}  (see \ref{Proof_Prop1} for a  proof).
\begin{prop}\label{Prop1} 
The optimal value of formulation  \eqref{InnerMaxW} equals that of the following formulation:
\begin{equation}\label{DualOfInner}
\inf_{\rho\geq 0} \Bigg\{ \epsilon \rho + \frac{1}{R} \sum_{r=1}^R \sup_{\xib \in \calS}  \bigg\{ f(\xb, \ab, \xib)-\rho \norms{\xib -\xibhat^r }_1  \bigg\}  \Bigg\}
\end{equation}
\end{prop}


Note that $f(\xb, \ab, \xib)-\rho \norms{\xib -\xibhat^r}_1$ is neither convex nor concave in $\xib$. Thus, formulation \eqref{DualOfInner} is potentially challenging to solve because it requires solving $R$ non-convex  optimization problems. Given that the support of service duration and travel times are rectangular and finite, we next show that we can recast these problems as linear programs for fixed $\rho$ and $(\xb,\ab)\in\calX\times\calA$. In what follows, we provide a high-level road map of the reformulation, and we relegate the notationally heavy proofs and details to Appendices. First, for fixed $(\xb,\ab)\in\calX\times\calA$ and $\rho \geq 0$, we denote $g_r(\rho, \xb, \ab)= \sup_{\xi \in \calS}  \{ f(\xb, \ab, \xib)-\rho \norms{ \xib -\hat{\xib}^r }_1  \}$. Given the dual formulation of $f(\xb, \ab, \xib)$ in \eqref{Dual2ndStage}, and using the dual of $f(\xb,\ab,\xib)$ and using the same proof techniques in Proposition \ref{Prop2M}, we can reformulate $g_r(\rho, \xb, \ab)$ (see \ref{Proof_Prop4M} for a detailed proof).
%
%

\begin{prop}\label{Prop4M}
The maximization problem $g_r(\rho, \xb, \ab)$ is equivalent to  
\allowdisplaybreaks
\begin{subequations}\label{W_inner_r_LP}
\begin{align} 
 \min_{\betab^r,\,\ub^r,\,\nub^r}  &\  \sum \limits_{j=1}^{N+2} \beta^r_j + \sum_{i=1}^N u_{i,0}^r x_{i,N}  \\
\textup{s.t.    } & \ \beta^r_1 \geq -a_1  \pi_{1,1} +\sum_{i=1}^N u_{0,i,1,1}^r x_{i,1} , \label{W_inner_r_LP:Con1} \\
& \ \sum_{j=1}^v \beta^r_j \geq -a_1  \pi_{1,v} + \sum_{j=2}^{\min(v,N+1)} \big ( a_{j-1}-a_j \big) \pi_{j,v} + \sum_{i=1}^N u_{0,i,1,v}^r x_{i,1} \nonumber \\
& \ \qquad \   + \sum_{j=2}^{\min(v,N+1)} \sum_{i=1}^N \sum_{i' \neq i} u_{i,i',j,v}^r x_{i,j-1}x_{i', j}  + \sum_{j=2}^{\min(v,N+1)} \sum_{i=1}^N \nu_{i,j,v}^r x_{i,j-1}, \ \
 \forall v \in [2, N+2]_{\Z},   \label{W_inner_r_LP:Con2} \\
& \ \sum_{j=k}^v\beta^r_j \geq \sum_{j=k}^{\min(v,N+1)} \big ( a_{j-1}-a_j\big)  \pi_{j,v}+\sum_{j=k}^{\min(v,N+1)} \sum_{i=1}^N \nu_{i,j,v}^r x_{i,j-1}  \nonumber \\
& \ \qquad \   + \sum_{j=k}^{\min(v,N+1)} \sum_{i=1}^N \sum_{i' \neq i} u_{i,i',j,v}x_{i,j-1}x_{i',j},\quad  \forall k\in[2,N+1]_\Z,\,\forall v\in[k,N+2]_\Z  ,\label{W_inner_r_LP:Con3} \\
& \ \beta^r_{N+2} \geq 0 , \label{W_inner_r_LP:Con4} \\
& \ u^r_{i,0} \geq \lambda\that^r_{i,0},\quad  u^r_{i,0} \geq \lambda\tub_{i,0}-\rho(\tub_{i,0}-\that^r_{i,0}),\quad\forall i\in[N], \label{W_inner_r_LP:Con5}  \\
& \ u_{0,i,1,v}^r  \geq (\pi_{1,v}+\lambda)\tlb_{0,i}-\rho(\that_{0,i}^r-\tlb_{0,i}),\quad  u_{0,i,1,v}^r  \geq (\pi_{1,v}+\lambda)\that_{0,i}^r,\nonumber\\
& \ u_{0,i,1,v}^r  \geq (\pi_{1,v}+\lambda)\tub_{0,i}-\rho(\tub_{0,i}-\that^r_{0,i}),\quad \forall i\in[N],\, v\in[N+2], \label{W_inner_r_LP:Con7} \\
& \ u_{i,i',j,v}^r   \geq (\pi_{j,v}+\lambda)\tlb_{i,i'}-\rho(\that_{i,i'}^r-\tlb_{i,i'}),\quad u_{i,i',j,v}^r   \geq (\pi_{j,v}+\lambda)\tub_{i,i'}-\rho(\tub_{i,i'}-\that^r_{i,i'}), \nonumber\\
& \ u_{i,i',j,v}^r   \geq (\pi_{j,v}+\lambda)\that_{i,i'}^r,\quad \forall i\in[N],\, i'\in[N]\setminus\{i\},\, j\in[2,N+1]_\Z,\, v\in[j,N+2]_\Z ,\label{W_inner_r_LP:Con9} \\
& \ \nu_{i,j,v}^r    \geq \pi_{j,v}\dlb_{i}-\rho(\dhat_{i}^r-\dlb_{i}),\quad 
 \nu_{i,j,v}^r    \geq \pi_{j,v}\dub_{i}-\rho(\dub_{i}-\dhat^r_{i}), \nonumber \\
& \ \nu_{i,j,v}^r    \geq \pi_{j,v}\dhat_{i}^r,\quad \forall i\in[N],\, j\in[2,N+1]_\Z,\, v\in[j,N+2]_\Z. \label{W_inner_r_LP:Con11} 
\end{align}
\end{subequations}
\end{prop} 

Summing $g_r(\rho,\xb,\ab)$ in the form of Proposition \ref{Prop4M} over $r$ and combining it with the outer minimization in \eqref{DualOfInner} and \eqref{W-DHRAS}, we derive MINLP reformulation of the W-DHRAS. Again due to its large size, we present this MINLP in \ref{Appx:Mac_WDHRAS}, where we also derive the following tight and exact MILP reformulation of the W-DHRAS model.
%
\allowdisplaybreaks
\begin{subequations}\label{W_DHRAS_L}
\begin{align}
 \min &\ \ \epsilon \rho + \frac{1}{R} \sum_{r=1}^R \left(\sum_{i=1}^N \psi^r_{i,0} + \sum_{j=1}^{N+2}\beta_j^r \right)  \\
\text{s.t.} &\ \  \xb\in\calX,\,\ab\in\calA,\, \rho\geq 0,\, \betab^r\in\R^{N+2},\quad\forall r\in[R], \\
& \ \ \text{constraints }\eqref{WMILP-MAC1M} -\eqref{WMILP-MAC9M}, \,\eqref{W_inner_r_LP:Con5}-\eqref{W_inner_r_LP:Con11},\, \eqref{W-DHRAS_MILP_Con2}-\eqref{W-DHRAS_MILP_Con5},\, \forall r\in[R]. \label{W_DHRAS_L:Con2} 
\end{align}
\end{subequations}

\begin{rem}
The use of Wasserstein ambiguity is useful in modeling uncertainty with limited data. In this work, we only consider the $1$-Wasserstein ambiguity set (i.e., $p=1$) since for general $p>2$, it usually results in a non-linear and complex reformulation. In the HRAS problem, routing decision is binary and hence, it could result in mixed-integer non-linear formulation that is more computationally expensive than MILP in general. Also, it is not reasonable to consider $p=\infty$ since the term $\norm{\xib_1-\xib_2}_\infty$ only measure the maximum difference of one particular entry. However, in our problem, every random parameter plays a role in the optimization model and differences in each service and travel time should be considered.
\end{rem}

\section{Numerical Experiments} \label{sec4:Exp}
Our computational study's primary objective is to compare the performance of the proposed DRO models (M-DHRAS and W-DHRAS) and a sample average approximation (SAA) model, and derive insights into HRAS. The SAA model solves model \eqref{SP} with $\mathbb{P}_{\xib}$ replaced by an empirical distribution based on $N$ samples of the random parameters (see \ref{apdx: SAA} for the formulation). For simplicity, we call the SAA model as the SP model. We focus on HRAS instances where the sample size is possibly small or we do not have enough data to model distributions accurately, which is often seen in healthcare applications. However, we also test the computational performance of SP under a large sample size. In Section \ref{subsection:Expt_Setup}, we describe the set of HRAS instances that we constructed and discuss other experimental setup. In Section \ref{subsection:Wass_Eps}, we examine the choice of $\epsilon$ in W-DHRAS model and the corresponding effect on the out-of-sample simulation performance. In Section~\ref{subsection:AppointmentTime}, we analyze the appointment time structure. In Section \ref{subsection:Out-of-sample}, we analyze the optimal solutions of the models and then compare their out-of-sample simulation performance. We also discuss the reliability of the models in Section \ref{subsection:Reliability}. In Section \ref{subsection:CPU_Time}, we compare the computational performance of the three models.

\subsection{\textbf{Description of the experiments}} \label{subsection:Expt_Setup}
We construct HRAS instances based on benchmarks, parameters settings and assumptions made in recent related literature (see, e.g., \citealp{jiang2019data}, \citealp{zhan2020home}, \citealp{zhan2018vehicle}). The average number of customers that an operator may visit per day is often less than six in home health care and banking, and less than 10 in repair service \citep{NAH2010, yuan2015branch,zhan2020home}. Accordingly, we consider problem instances with 6, 8 and 10 customers. For example, it is often not feasible for a caregiver to visit more than 10 patients (often 6)  a day considering care duration at each location and travel time between patients in home health care applications. However, we also test the computational performance of the proposed models under unrealistic larger instances of HRAS. We consider two different cost structures for waiting, idling and overtime in the objective function:  (a) $(\cw_j,\cu_j,\co)=(2,1,20)$ (\citealp{jiang2017integer}, \citealp{jiang2019data}), and (b) $(\cw_j,\cu_j,\co)=(1,5,7.5)$ (\citealp{shehadeh2020distributionally}). For the transportation cost, we consider $\lambda\in\{0.5,1,2\}$ (\citealp{zhan2020home}).  We set  $L$ to 8 hours as in \cite{zhan2018vehicle}.

We use the lognormal distribution (\citealp{jiang2019data}) for the service time $d_i$ truncated on the interval $[10,50]$ with mean $\mu$ and $\sigma=0.5\mu$, where $\mu$ is generated from $U[25,35]$ ($U[a,b]$ refers to uniform distribution over the interval $[a,b]$). Our model works with any choices of the range $[10,50]$. Previous studies such as \cite{zhan2020home} assumes deterministic travel time. In our study, we generate the random travel time $t_{i,i'}$ from $U[15,25]$. That is, we assume that customers are fairly separated and traveling from one place to the other takes $20$ minutes on average (though our models can solve instances with any ranges and distribution of travel time). This is also consistent with prior and recent literature. \cite{nikzad2021matheuristic} particularly point out that customers within a service region form a basic unit or a cluster that share the same distribution of travel time, which is seen in urban areas. We round each generated parameter to the nearest integer.

We use the same upper and lower bounds of service time $[\dlb_i,\dub_i]=[10,50]$ and travel time $[\tlb_{i,i'},\tub_{i,i'}]=[15,25]$ in M-DHRAS and W-DHRAS. In the M-DHRAS model, the mean parameters $\mudb$ and $\mutb$ in the ambiguity set are set as the sample mean of the data. We also introduce new symmetry-breaking constraints (see \ref{apdx: SymBreak}) to break the symmetry in the routing decision. We implemented the three proposed models in AMPL programming language and use CPLEX (version 12.7.0.0) solver with the default setting. The relative MIP gap tolerance is set to $0.02$ while most of the instances have a terminal relative MIP gap tolerance very close to $0$. All the experiments were conducted on a computer with AMD Opteron 2.0 GHz CPU and 16 Gb memory. The time limit for solving each instance is set to 2 hours.


\subsection{\textbf{Effect of $\epsilon$ in W-DHRAS model}} \label{subsection:Wass_Eps}

In the W-DHRAS model, there is one parameter in the uncertainty set that serves as an input: the Wasserstein ball's radius $\epsilon$. In this section, we demonstrate the effect of $\epsilon$ on the out-of-sample performance of the W-DHRAS's optimal solution, $(\hat{\xb} (\epsilon, R),\hat{\ab} (\epsilon, R))$, with respect to the radius $\epsilon$. For illustrative purposes, we focus on one instance of $6$ customers with cost structure $(\cw_j,\cu_j,\co)=(2,1,20)$ and $\lambda=0.5$.

For each $\epsilon \in \{0.01, 0.02, \dots, 0.09, 0.1, \dots, 0.9,1,\dots,10\}$ (i.e., log-scaled interval as in \citealp{esfahani2018data} and \citealp{jiang2019data}), we evaluate the out-of-sample performance  as follows. First, for each customer, we randomly generate $30$ data sets of service duration and travel time scenarios, each consisting of $R\in \{ 5,20,50  \}$ scenarios. We generate these data sets under the same settings described in Section \ref{subsection:Expt_Setup}. Second, we solve the W-DHRAS model in \eqref{W_DHRAS_L} using the generated data sets under each of the candidate Wasserstein radius $\epsilon$. Finally, we fix the first-stage variables to the optimal solution of each instance, and then re-optimize the second-stage of the SP using 10,000  out-of-sample (unseen) data.  This is to compute the corresponding out-of-sample overtime, idle time, travel time, and waiting time and hence, the second-stage cost. 
\begin{figure}
    \centering
    \begin{subfigure}[b]{0.3\textwidth}
       \includegraphics[scale=0.48]{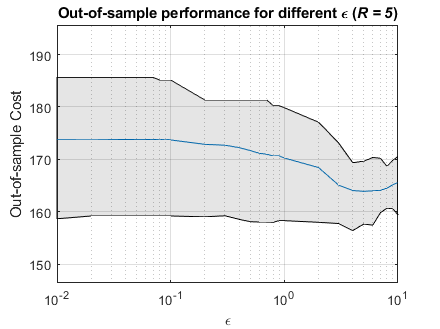}
       \caption{$R=5$}
       \label{fig:Wass_eps_5}
    \end{subfigure}
    \begin{subfigure}[b]{0.3\textwidth}
       \includegraphics[scale=0.48]{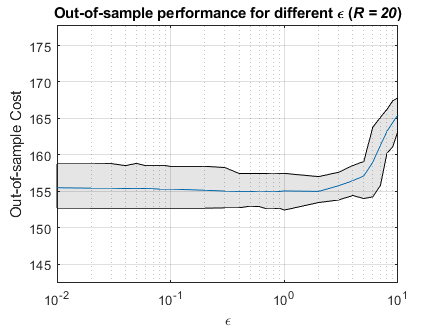}
       \caption{$R=20$}
       \label{fig:Wass_eps_20}
    \end{subfigure}
    \begin{subfigure}[b]{0.3\textwidth}
       \includegraphics[scale=0.48]{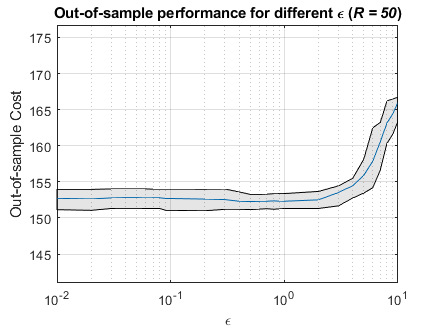}
       \caption{$R=50$}
       \label{fig:Wass_eps_50}
    \end{subfigure}
    \caption{Performance of W-DHRAS with different choices of $\epsilon$}
    \label{fig:Wass_eps}
\end{figure}

Figure \ref{fig:Wass_eps} illustrates the out-of-sample performance with $R\in\{5,20,50\}$ under different choices of $\epsilon$. The blue line represents the mean of the 30 out-of-sample costs, while the shaded region is the area between the 20th and 80th percentiles of the 30 out-of-sample costs. It is quite evident that the out-of-sample performance first decreases with $\epsilon$ and then increases after some value of $\epsilon$. This pattern is often observed in the literature (see, e.g., \citealp{esfahani2018data} and \citealp{jiang2019data}). The values of $\epsilon$ which gives the smallest out-of-sample cost with $R=5$, $20$ and $50$ are respectively $5$ (Figure \ref{fig:Wass_eps_5}), $0.6$ (Figure \ref{fig:Wass_eps_20}) and $0.5$ (Figure \ref{fig:Wass_eps_50}) respectively. This decrease in $\epsilon$ with the increase in $R$ is not surprising. Intuitively, a small sample provides little information on the true distribution, and thus a larger $\epsilon$ produces robust solutions that better hedge against ambiguity. In contrast, with a larger sample, we have more information from the data, and so we can make less conservative decisions using a smaller $\epsilon$ value. As such, one should choose a larger (smaller) $\epsilon$ with a small (large) sample. Indeed, we can see that the out-of-sample cost for $R=20$ is smaller than that of $R=5$ in most of the cases, which indicates that a larger sample size could give a better out-of-sample performance. Moreover, this shows that when sufficient data is available, one could employ the SP model directly (with $\epsilon=0$). Thus, as mentioned earlier, W-DHRAS is particularly useful when there is a small data on random parameters.

In practice, decision-makers do not often have optimization expertise (or time) to run the above procedure (or other iterative or cross validation procedures). Additionally, we do not have full distributional information most of the time but only a small set of data samples. In the following experiments, we pick three different values of $\epsilon$, namely $0.5$, $5$, and $50$, which captures different extents of robustness of the W-DHRAS model. For brevity, we label them as W-DHRAS(0.5), W-DHRAS(5) and W-DHRAS(50) respectively.

\subsection{\textbf{Appointment time structure}}\label{subsection:AppointmentTime}
In this section, we analyze the optimal appointment structure produced by the SP, M-DHRAS, and W-DHRAS models.  For illustrative purposes and brevity, we present results under $(\cw_j,\cu_j,\co)=(2,1,20)$ and $(\cw_j,\cu_j,\co)=(1,5,7.5)$ with $\lambda=2$. We observe similar results with other choices of $\lambda$ (see \ref{apdx:AT_Structure}). Figures~\ref{fig:AppointmentTime1_lam2} and \ref{fig:AppointmentTime2_lam2} present the optimal schedules of the operator produced by the SP and DRO models under  $(\cw_j,\cu_j,\co)=(2,1,20)$ and $(\cw_j,\cu_j,\co)=(1,5,7.5)$, respectively. The point $(x,y)=(i$,  inter-arrival time$)$ of every schedule in each subfigure corresponds to the mean optimal inter-arrival time (i.e., differences between the scheduled service start times of two consecutive customers, $I_j=a_j-a_{j-1}$ for $i\in[N]$ with $a_0=0$). 
\begin{figure}[t]
    \hspace{-10mm}
    \includegraphics[scale=0.60]{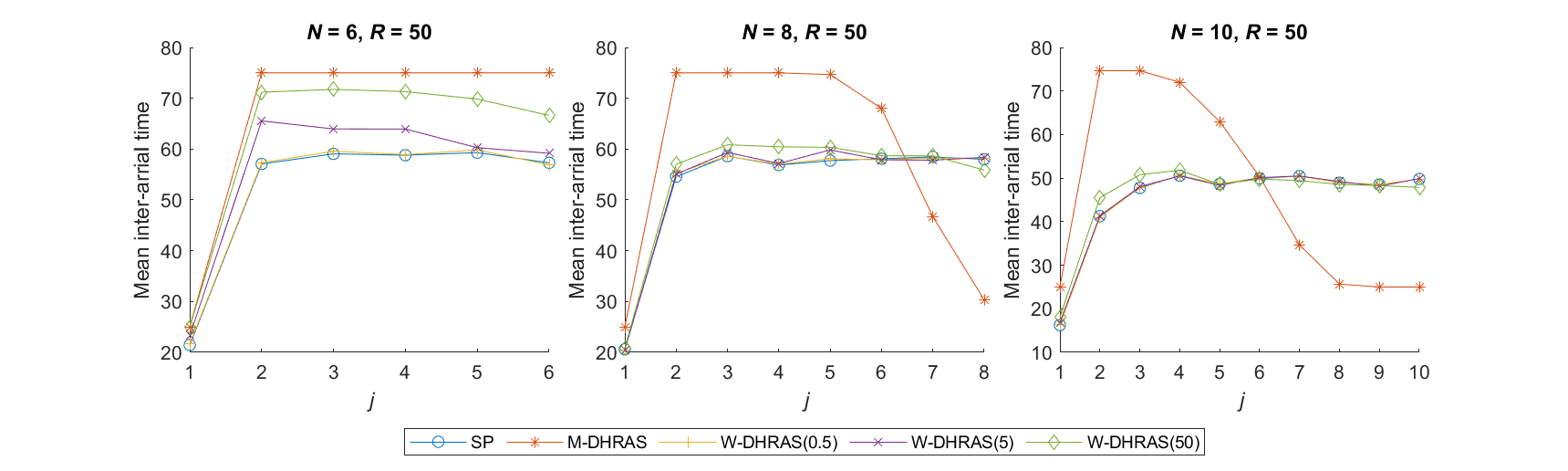}
    \caption{Mean inter-arrival times $a_j-a_{j-1}$ with $R=50$ under $(\cw_j,\cu_j,\co)=(2,1,20)$ and $\lambda=2$}
    \label{fig:AppointmentTime1_lam2}
\end{figure}
\begin{figure}[t!]
    \hspace{-10mm}
    \includegraphics[scale=0.60]{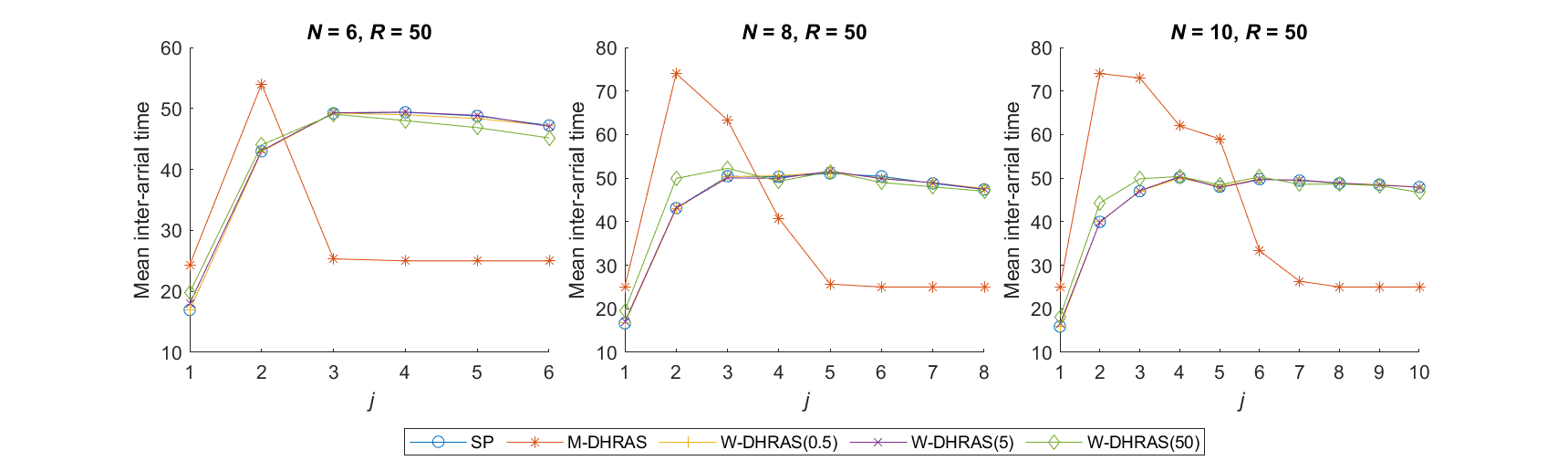}
    \caption{Mean inter-arrival times $a_j-a_{j-1}$ with $R=50$ under $(\cw_j,\cu_j,\co)=(1,5,7.5)$ and $\lambda=2$}
    \label{fig:AppointmentTime2_lam2}
\end{figure}

We first analyze the results under $(\cw_j,\cu_j,\co)=(2,1,20)$ in Figure~\ref{fig:AppointmentTime1_lam2}. First, all models assign less time for the first customer than the subsequent customers and more time between the first and second customers. The models distribute (roughly) equal time between the middle customers (e.g., customers 2-5 when $N=6$ and customers 3-7 when $N=8$), and schedule less time between the last 2 to 3 customers, especially when $N=10$ for M-DHRAS, though SP allocates (roughly) equal time for both middle and last customers. Second, it is evident that M-DHRAS is the most conservative model and tends to schedule more time between customers when $N=6$ than the SP and W-DHRAS models. Also, M-DHRAS model yields a dome-shaped inter-arrival time structure (i.e., the time between customers first increases then decreases) when $N=8$ and $N=10$. Third, the SP model always schedules less time between the customers than the DRO models (except possibly the last few customers). This is not surprising since the SP model seeks risk-neutral decisions. Fourth, the W-DHRAS model yields a similar pattern as the SP model; however, it schedules more time between customers than the SP as $\epsilon$ increases. This makes sense because as $\epsilon$ increases, the W-DHRAS model becomes more conservative and schedules more time between appointments. This demonstrates that the W-DHRAS model is more conservative in hedging ambiguity than the SP model but less conservative than the M-DHRAS model. In the next section, we show that the W-DHRAS appointment pattern yields the best out-of-sample performance under various distributions.
 
Next, we analyze the results under $(\cw_j,\cu_j,\co)=(1,5,7.5)$ in Figure \ref{fig:AppointmentTime2_lam2}. We observe different patterns under this cost structure. All models assign more time for the first few customers but gradually less time for the middle and last few customers. The pattern is different from the cost structure, mainly due to a significantly higher idling cost relative to the waiting cost. However, we can still conclude that M-DHRAS is the most conservative model, but instead of allocating more time, it allocates less time between customers to hedge against scenarios that yields long idling time. Third, we observe that the dome-shaped inter-arrival time structure for all models, which is particularly prominent in $N=6$. Forth, we note that W-DHRAS models with $\epsilon=0.5$ and $\epsilon=5$ roughly allocates the same amount of time between customers as the SP model. However, with $\epsilon=50$, W-DHRAS model tends to allocate more time for the first few customers but fewer time for the last few customers.

\subsection{\textbf{Out-of-sample performance}}\label{subsection:Out-of-sample}
In this section, we compare the optimal solutions of the three models we analyzed in the previous section under limited historical data.  We test the out-of-sample performance (i.e.,  the objective value obtained by simulating the optimal solution of a model under a larger unseen data) as follows. For each $N\in\{6,8,10\}$ and $R \in \{5, 10, 20, 50\}$, we generate $30$ SP, M-DHRAS and W-DHRAS instances using the same parameters settings described in Section \ref{subsection:Expt_Setup}. Note that the sample sizes are not large to mimic the situation where the decision-making possesses limited data. We solve each instance and obtain the optimal first-stage $(\xb,\ab)$ decision. Fixing the first-stage decision to $(\xb,\ab)$, we then re-optimize the second-stage of the SP using the following five sets of  $10,000$ samples.  In the first set, we assume that the data we rely on in the optimization comes from the true distribution. We use data from Sets 2-5 to evaluate the performance of the models under the case when the data we rely on in the optimization may follow a biased or misspecified distribution different from the true distribution. 

\begin{enumerate}\itemsep0em
    \item [Set 1.] We assume perfect information for the distributions. That is, we generate the $10,000$ samples from the same distribution we use in the optimization as discussed in Section \ref{subsection:Expt_Setup}. This simulation assumes that the data comes from the true unknown distribution.

    \item [Set 2.] In this set, we assume that we have misspecified the distribution of the travel time in the optimization. We generate $t_{i,i'}$ from $U[25,35]$ instead of $U[15,25]$. That is, the average travel time takes 10 minutes longer than usual. This situation (shift in the travel time range) might be seen in practice due to unexpected traffic congestion (e.g., caused by traffic accidents, weather conditions etc.). 

    \item [Set 3.]In this set, we assume that we have misspecified both the service and travel times distributions in the optimization. Specifically, we follow a similar out-of-sample simulation testing procedure described in \cite{wang2020distributionally} and perturb the support of the random travel and service times by a parameter $\delta$ as $[(1-\delta)$ lower bound, $(1+\delta)$ upper bound], where $\delta\in\{0.1,0.25,0.5\}$. A higher value of $\delta$ corresponds to a higher variation level.
\end{enumerate}

For brevity, we focus on HRAS instances of $N = 6$ customers (i.e., the average number of customers that a home service operator often visits per day; see Section \ref{subsection:Expt_Setup}) and present results under $(\cw_j,\cu_j,\co)=(2,1,20)$ and $\lambda=2$. We observe similar out-of-sample performance with $N=8$ and $N=10$, and under other values of $\lambda$ as well as two additional sets of misspecified distribution (see these results in  \ref{apdx:OutSample}). Figure \ref{fig:OC_1} shows the results using Set 1, where the shaded region represents the 20\% and 80\% percentiles of the out-of-sample cost. Obviously, the optimal solutions of the M-DHRAS model have poor out-of-sample performance. Recall from Section \ref{subsection:AppointmentTime} that M-DHRAS schedules a long time between customers, which, as shown in Figure \ref{fig:IWO_1}, yields a significant amount of operator idle time (indicating poor utilization of the operator's time). The out-of-sample performance of the W-DHRAS depends on $\epsilon$, which intuitively follows from the observation that the optimal solutions depend on $\epsilon$ (Section \ref{WDHRAS}). When $\epsilon=0.5$, the performance of W-DHRAS is similar to SP, while a more conservative choice of $\epsilon$ such as $50$ yields a poorer performance as compared to SP. When $\epsilon=5$,  W-DHRAS  outperforms SP, especially when the sample size is small. This demonstrates that the W-DHRAS model has a superior performance when the data size is small (i.e., when there is limited data on random parameters).
\begin{figure}[t]
    \hspace{-10mm}
    \includegraphics[scale=0.60]{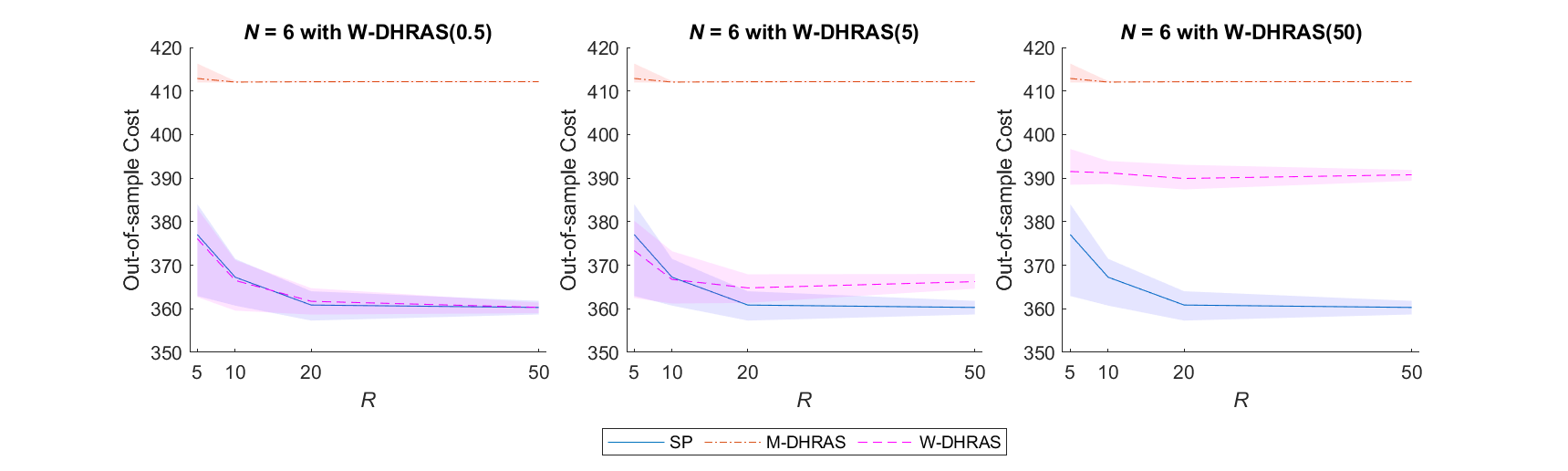}
    \caption{Out-of-sample cost with cost structure (a) and $\lambda=2$ under Set 1}
    \label{fig:OC_1}
\end{figure}
\begin{figure}[t]
    \hspace{-10mm}
    \includegraphics[scale=0.60]{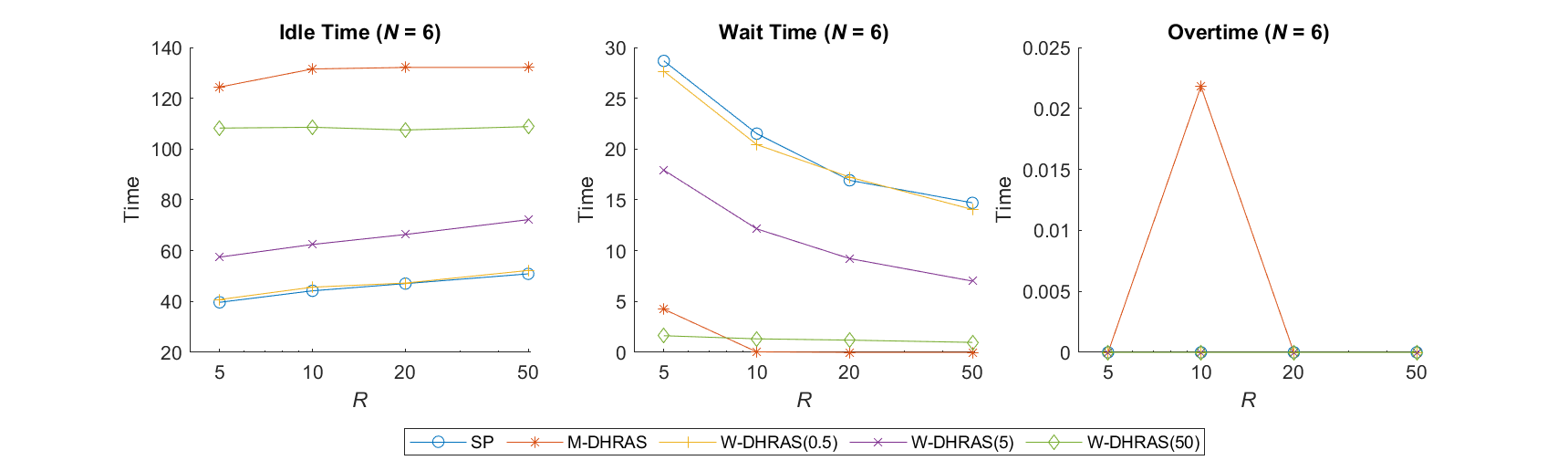}
    \caption{Mean idle time, wait time and overtime with cost structure (a) and $\lambda=2$ under Set 1}
    \label{fig:IWO_1}
\end{figure}

Next, we discuss results for the case when we have misspecified the distributions. Figure \ref{fig:OC_2} shows the performance under Set 2, where a longer travel time is encountered. In contrast to the results in Figure \ref{fig:OC_1}, M-DHRAS and W-DHARS(50) outperforms SP, W-DHRAS(0.5) and W-DHRAS(5). This demonstrates that M-DHRAS and W-DHRAS(50), which are more conservative than the other models, have superior performance in environments where travel time increases by a lot.  This is not suprising since M-DHRAS and W-DHRAS(50) schedule longer time between customers that could hedge against the increase in travel time. The W-DHRAS model outperforms the SP model and the M-DHRAS model when $\epsilon=50 $ (a conservative choice of $\epsilon$).  We can observe from Figure \ref{fig:IWO_2} the trade-off between idle and wait time. Under Set 2, SP model results in a significant amount of waiting time but a small idling time compared with the two DRO models.
\begin{figure}[t]
    \hspace{-10mm}
    \includegraphics[scale=0.60]{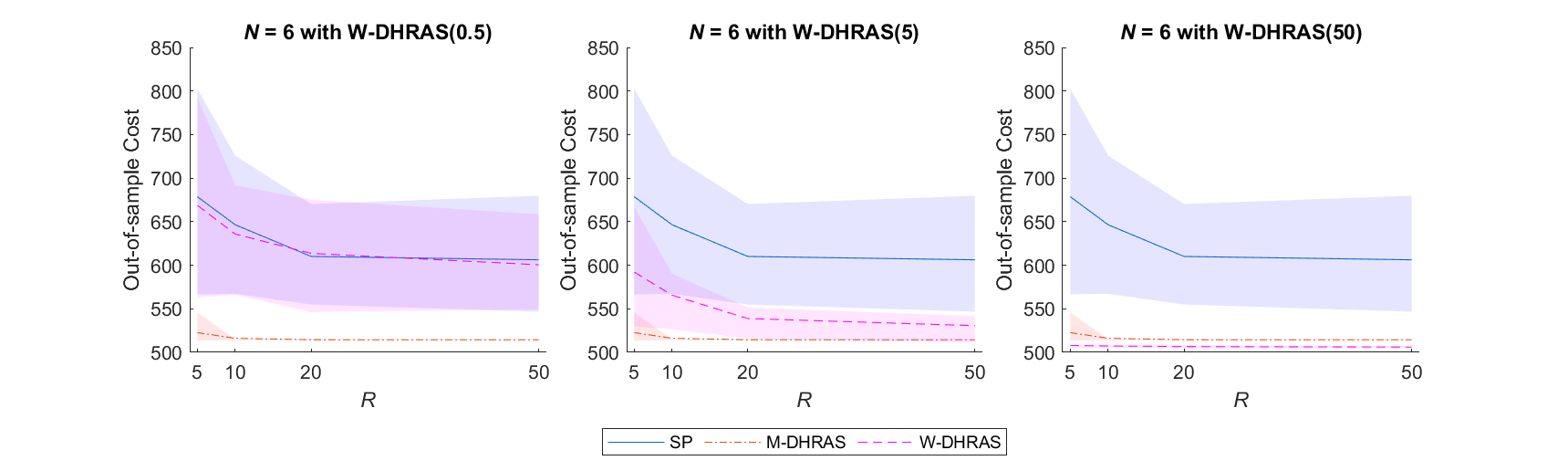}
    \caption{Out-of-sample cost with cost structure (a) and $\lambda=2$ under Set 2}
    \label{fig:OC_2}
\end{figure}
\begin{figure}
    \hspace{-10mm}
    \includegraphics[scale=0.60]{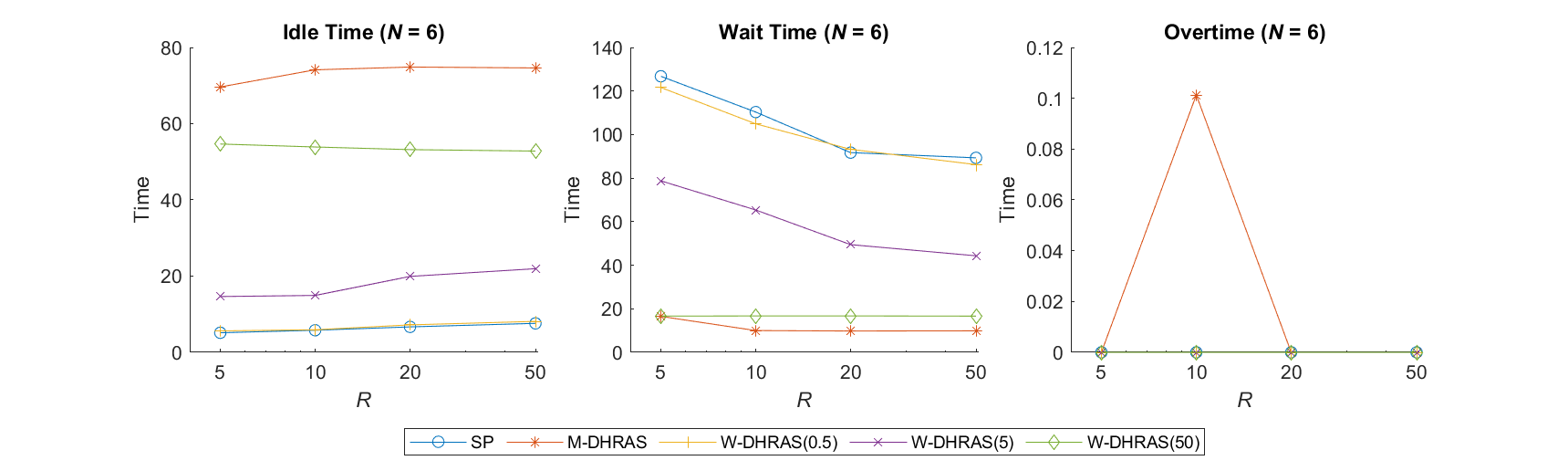}
    \caption{Mean idle time, wait time and overtime with cost structure (a) and $\lambda=2$ under Set 2}
    \label{fig:IWO_2}
\end{figure}

We also demonstrate one of the cases that the supports are perturbed. Figures \ref{fig:OC_3a} to \ref{fig:OC_3c} show the results for Set 3 with different choices of $\delta$. We observe that M-DHRAS performs better (with respect to SP) when $\delta$ increases, which corresponds to a larger deviation. The W-DHRAS(5) model consistently outperforms the SP model. We also observe that the W-DHRAS model has a larger edge over SP when the sample size is small. This is consistent with what we have discussed in previous cases. These results demonstrate that the DRO models perform better than the SP model when there is a large deviation from the sample distribution. 

%
\begin{figure}[t]
    \hspace{-10mm}
    \includegraphics[scale=0.60]{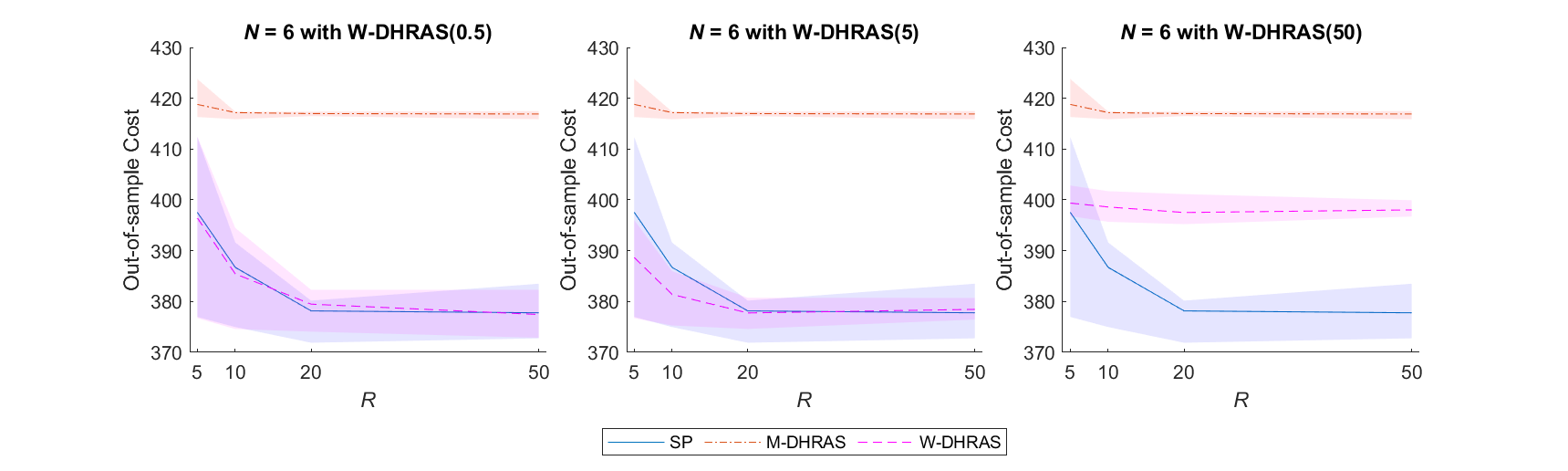}
    \caption{Out-of-sample cost with cost structure (a) and $\lambda=2$ under Set 3 ($\delta=0.10$)}
    \label{fig:OC_3a}
\end{figure}
\begin{figure}[t]
    \hspace{-10mm}
    \includegraphics[scale=0.60]{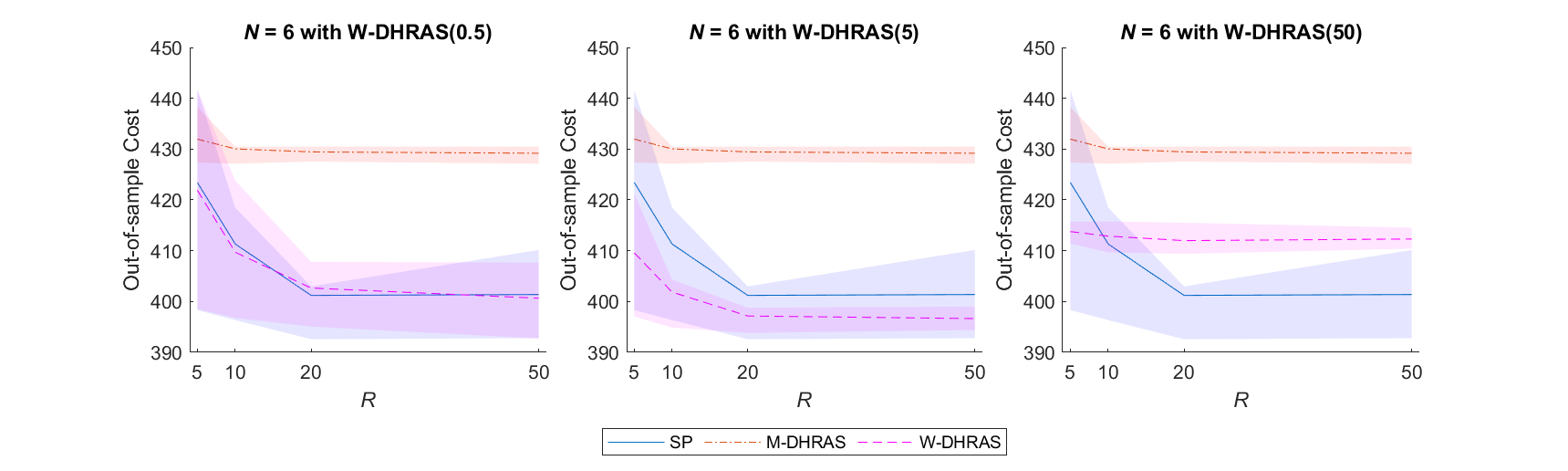}
    \caption{Out-of-sample cost with cost structure (a) and $\lambda=2$ under Set 3 ($\delta=0.25$)}
    \label{fig:OC_3b}
\end{figure}
\begin{figure}[t!]
    \hspace{-10mm}
    \includegraphics[scale=0.60]{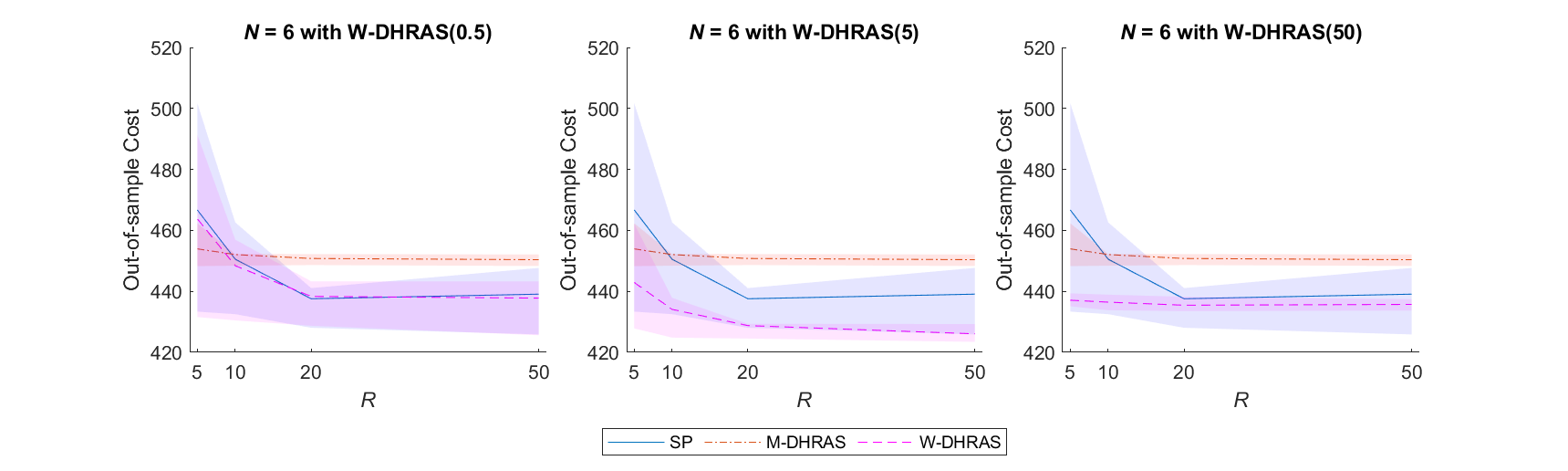}
    \caption{Out-of-sample cost with cost structure (a) and $\lambda=2$ under Set 3 ($\delta=0.50$)}
    \label{fig:OC_3c}
\end{figure}

\subsection{\textbf{Reliability of the models}}\label{subsection:Reliability}
In this section, we analyze the reliability of the three models. Home service operators often need to estimate the budgets on the operational costs ahead of the actual service provision. The optimal value of an HRAS optimization model serves as a relevant estimate of the associated cost if we implement the corresponding optimal decision $(\xb,\ab)$. Intuitively,  when the estimated cost is smaller than the actual cost, the operator may run into financial problems related to a budget deficit and poor planning decisions. A risk-averse operator would then prefer implementing decisions that provide an upper bound on the estimated cost, i.e., seek solutions with higher reliability. Hence, we make use of the reliability measure. For any given decision $(\xb,\ab)$, reliability is the probability that the optimal value from the model is greater than or equal to the actual cost $\E[f(\xb,\ab,\xib)]$. Mathematically, if $v(R)$ is the optimal value and $(\xb(R),\ab(R))$ is the corresponding optimal solution with $R$ samples, $$\text{reliability}=\Prob_{\xib}^R\Big(v(R) \geq \E_{\Prob_{\xib}}[f(\xb(R),\ab(R),\xib)]\Big),$$ where $\Prob_{\xib}^R$ is the product measure of $R$ copies of $\Prob_{\xib}$. The actual cost $\E_{\Prob_{\xib}}[f(\xb(R),\ab(R),\xib)]$ is estimated from $10,000$ out-of-sample scenarios and the reliability is computed from $30$ instances.

Figures \ref{fig:Reliability1} and \ref{fig:Reliability2} show the results for the two cost structures $(\cw_j,\cu_j,\co)=(2,1,20)$ and $(1,5,7.5)$ respectively with $\lambda=2$. We observe that the reliability is, in general, increasing with $R$. This is reasonable since we have a larger amount of historical data. The reliability of SP is the lowest among the three models, which demonstrates that SP with a small number of historical scenarios may not be able to provide a robust cost estimate. On the other hand, the reliability of W-DHRAS(50) consistently gives the best reliability result. This makes sense since $\epsilon=50$ corresponds to a relatively robust model. It is interesting that although M-DHRAS appears to be the most conservative model from the out-of-sample cost perspective, the reliability result is not the best. Overall, this demonstrates that DRO models have a better reliability compared to SP.
\begin{figure}[t]
    \hspace{-10mm}
    \includegraphics[scale=0.60]{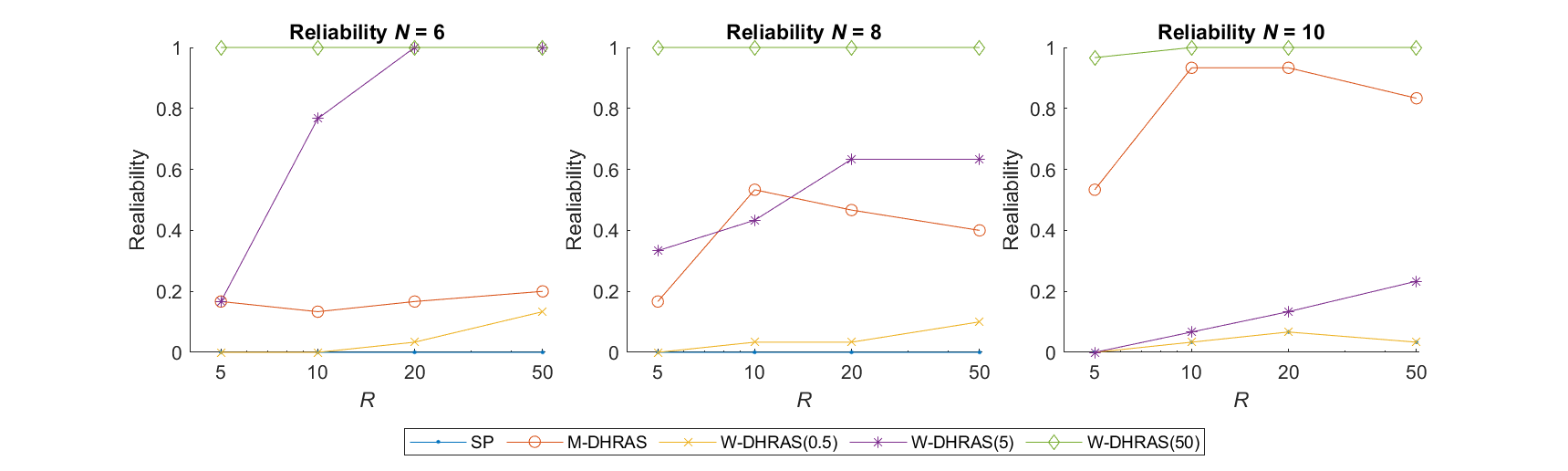}
    \caption{Reliability under $(\cw_j,\cu_j,\co)=(2,1,20)$ and $\lambda=2$}
    \label{fig:Reliability1}
\end{figure}
\begin{figure}[t!]
    \hspace{-10mm}
    \includegraphics[scale=0.60]{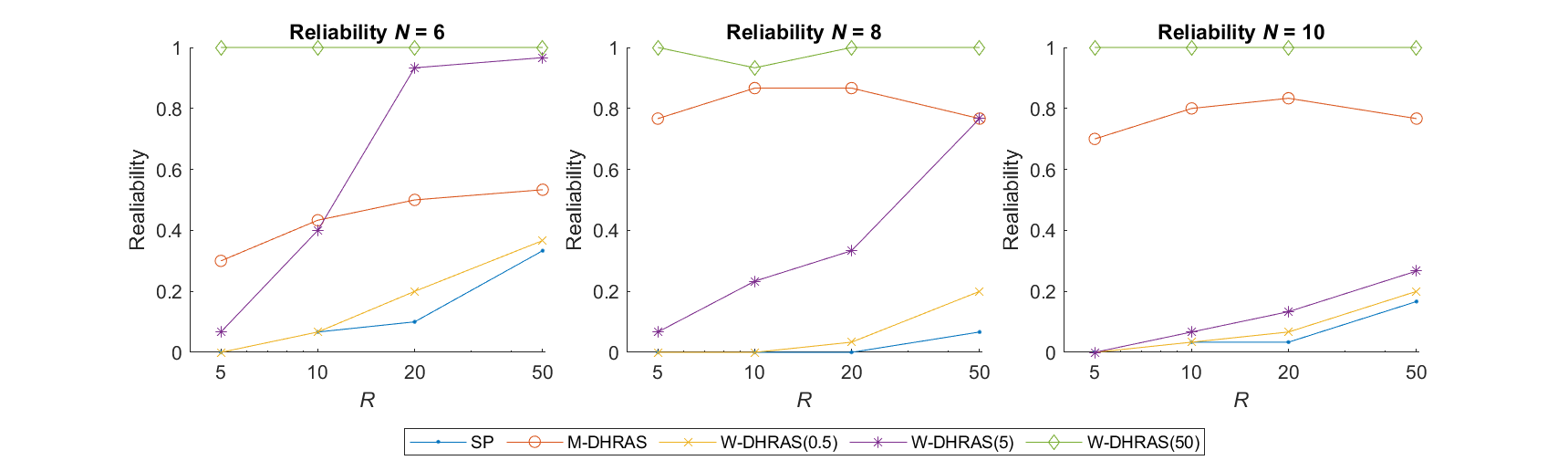}
    \caption{Reliability under $(\cw_j,\cu_j,\co)=(1,5,7.5)$ and $\lambda=2$}
    \label{fig:Reliability2}
\end{figure}

\subsection{\textbf{CPU time}}\label{subsection:CPU_Time}

In this section, we analyze the computational time for solving the constructed practical instances using the three models. Specifically, we examine the effect of $N$ (number of customers) and $R$ (number of scenarios) on the solution time for the three models as follows. Table $\ref{table:CPU_Time}$ shows the mean CPU time for solving $30$ instances using two different cost structures with $\lambda=2$.

\begin{table}[t!]\centering\small
\renewcommand{\arraystretch}{0.4}
\caption{CPU time in seconds for solving the three models are shown and the number in parentheses followed by W-DHRAS model is the choice of $\epsilon$. The reported times are for $\lambda=2$ with two different cost structures. } 
\begin{tabular}{@{}lrrrr|rrrr@{}} \toprule
$N = 6$ & \multicolumn{4}{c}{$(\cw_j,\cu_j,\co)=(2,1,20)$} & \multicolumn{4}{c}{$(\cw_j,\cu_j,\co)=(1,5,7.5)$}  \\
\cmidrule{2-5}  \cmidrule{6-9}
CPU Time (in s) & $R=5$ & $R=10$ & $R=20$ & $R=50$ & $R=5$ & $R=10$ & $R=20$ & $R=50$\\
\cmidrule{1-5} \cmidrule{6-9}
SP           & 0.25 & 0.30 & 0.37  & 0.42   & 0.23 & 0.21 & 0.33  & 0.44   \\
M-DHRAS      & 0.46 & 0.48 & 0.47  & 0.48   & 0.46 & 0.46 & 0.46  & 0.46   \\
W-DHRAS(0.5) & 2.21 & 5.24 & 18.57 & 102.52 & 1.95 & 5.04 & 18.94 & 108.99 \\
W-DHRAS(5)   & 2.33 & 6.02 & 20.93 & 125.19 & 1.99 & 5.22 & 20.26 & 121.32 \\
W-DHRAS(50)  & 2.34 & 5.34 & 22.93 & 134.77 & 2.48 & 7.00 & 25.11 & 170.84 \\
\toprule
$N = 8$ & \multicolumn{4}{c}{$(\cw_j,\cu_j,\co)=(2,1,20)$} & \multicolumn{4}{c}{$(\cw_j,\cu_j,\co)=(1,5,7.5)$}  \\
\cmidrule{2-5}  \cmidrule{6-9}
CPU Time (in s) & $R=5$ & $R=10$ & $R=20$ & $R=50$ & $R=5$ & $R=10$ & $R=20$ & $R=50$\\
\cmidrule{1-5} \cmidrule{6-9}
SP           & 0.28  & 0.33  & 0.39   & 0.76    & 0.28 & 0.32  & 0.39   & 0.79   \\
M-DHRAS      & 1.11  & 1.05  & 1.07   & 1.10    & 0.90 & 0.92  & 0.91   & 0.93   \\
W-DHRAS(0.5) & 6.14  & 16.70 & 90.07  & 583.90  & 5.29 & 15.29 & 81.90  & 514.50 \\
W-DHRAS(5)   & 7.42  & 20.42 & 101.00 & 681.52  & 5.59 & 18.05 & 85.59  & 530.59 \\
W-DHRAS(50)  & 10.05 & 27.10 & 145.41 & 1295.76 & 8.98 & 24.90 & 114.88 & 718.17 \\
\toprule
$N = 10$ & \multicolumn{4}{c}{$(\cw_j,\cu_j,\co)=(2,1,20)$} & \multicolumn{4}{c}{$(\cw_j,\cu_j,\co)=(1,5,7.5)$}  \\
\cmidrule{2-5}  \cmidrule{6-9}
CPU Time (in s) & $R=5$ & $R=10$ & $R=20$ & $R=50$ & $R=5$ & $R=10$ & $R=20$ & $R=50$\\
\cmidrule{1-5} \cmidrule{6-9}
SP           & 0.376 & 0.43  & 0.54   & 1.04    & 0.35  & 0.42  & 0.54   & 1.07    \\
M-DHRAS      & 2.05  & 2.02  & 2.11   & 1.97    & 1.79  & 1.83  & 1.86   & 1.89    \\
W-DHRAS(0.5) & 17.84 & 56.71 & 249.84 & 1801.04 & 15.07 & 51.71 & 220.39 & 1415.10 \\
W-DHRAS(5)   & 18.22 & 61.54 & 260.82 & 1922.29 & 15.54 & 53.34 & 228.30 & 1523.63 \\
W-DHRAS(50)  & 19.98 & 70.72 & 297.87 & 2490.39 & 16.41 & 59.70 & 270.15 & 1919.69 \\
\bottomrule
\end{tabular}
\label{table:CPU_Time}
\end{table}

We first observe that the SP model takes the shortest time to solve for all instances and most of them can be solved within one second. Moreover, SP requires a longer solution time with a larger number of scenarios. The M-DHRAS model has a slightly longer solution time than the SP model but, since it only depends on the mean and range of the sample, it has a consistent performance under all values of $R$. Solution times of the W-DHRAS model are longer than those of the M-DHRAS and SP models. Solution times of the W-DHRAS model increase with $R$ and varies across three choices of $\epsilon$. We attribute the difference in solution times between the SP and W-DHRAS models to their respective sizes (i.e., the number of variables and constraints).  For fixed $R$, the W-DHRAS model has more variables and constraints than the SP model. As pointed out by previous papers (e.g. \citealp{artigues2015mixed,klotz2013practical,shehadeh2020distributionally}), an increase in model size often suggests an increase in solution time for the LP relaxation and, thus, the MILP model’s overall solution time. 

Recall that W-DHRAS is theoretically more useful when there is only a small data on random parameters. When there is a large data set, W-DHRAS may converge to SP with $\epsilon$ close to zero and our SP is efficient with large $R$ (see \ref{apdx:SolTime}). Thus, although W-DHRAS takes the longest time among the three, it can solve realistic instances within a reasonable time, especially when $R$ is small. Specifically, when $N=6$ (common in home healthcare applications), W-DHRAS solution times range from 2 seconds ($R=5$) to 3 minutes ($R=50$). When $N=10$ (the maximum number of customers in HRAS applications; see discussion and references in Section~\ref{subsection:Expt_Setup}), W-DHRAS solution times range from 15 seconds ($R=5$) to around 30 minutes ($R=50$ and $\epsilon=0.5, 5$).

In \ref{apdx:SolTime}, we demonstrate the SP and M-DHRAS abilities to solve (unrealistic) large instances with a large number of customers and scenarios. For example, the solutions times of the SP with $N=15$ customers and $R=500$ scenarios are less than 1.1 minutes. Solution times of the M-DHRAS for the same instance are negligible.


\section{Conclusion} \label{sec5:conclusion}


In this paper, we address the uncertainty and distributional ambiguity of service and travel times in an HRAS problem. We propose and analyze the first SP model and two DRO models for HRAS with random service and travel times. We use two popular ambiguity sets in the DRO models: the mean-support ambiguity set (in the M-DHRAS model) and the $1$-Wasserstein ambiguity set (in the W-DHRAS model).  We derive equivalent mixed-integer linear programming (MILP) reformulations of both DRO models that can be implemented and efficiently solved using off-the-shelf optimization software.  We conduct extensive experiments comparing the proposed models. Our results demonstrate that: (1) W-DHRAS yields robust decisions with both small and large data size, and enjoys both asymptotic consistency and finite-data guarantees; (2) M-DHRAS produces the most conservative schedules and it performs well when the distribution of travel time changes dramatically (e.g., actual travel times are longer); (3) W-DHRAS solutions have better operational performance than the SP solutions both under perfect information and misspecified distributions, even when only a small data set is available; (4) DRO models produce more reliable solutions than the SP model; (5) the proposed SP and DRO models are computationally efficient under realistic HRAS settings. Thus, we conclude that the proposed DRO models are useful if the distributions of the random service and travel time are unknown and only a limited amount of data is available. The resulting optimal solutions are robust to misspecifications of the underlying distribution with high reliability. This addresses the primary goal of our paper to investigate the value of the DRO approach for HRAS and compare it with the SP approach.

Under the circumstance that only limited data is available, decision making process under uncertainty (in both service and travel times) is typically challenging. This is particularly true in the home health care context where the service time data is often not available. Except in the case that the underlying distribution coincides with the sample data distribution, the standard SP approach appears to be vulnerable (e.g., poor performance and reliability). The use of DRO models could remedy the distribution misspecification problem and provide a more robust solution, leading to a reduction in out-of-sample costs and improvement in reliability. This is typically true when we experience changes in the underlying distribution. While M-DHRAS yields the most conservative solution (e.g., large inter-arrival times), the model yields superior performance when there is a large perturbation in distribution (e.g., drastic increase in travel time). For instance, when the operator is planning for a day with possibly severe traffic congestion or weather condition (and hence, travel time is highly random), adopting M-DHRAS solutions will yield better operational performance. On the other hand, under normal operations, adopting W-DHRAS solutions may yield better operational performance  (e.g., with an empirical choice of $\epsilon=5$) and reduce out-of-sample costs due to distribution misspecification with a better reliability than SP.

 Our model can serve as a building block for the following future extensions and areas of research in terms of home service aspects, constraints, and various sources of uncertainties. First, we want to generalize our distributionally robust approach to the case when we have multiple service providers and include decisions such as (1) determining the number of service teams to hire,  (2) assigning the hired service teams to customers, (3) constructing routes for the service teams, and (4) determining the customers' appointment times. Second, we aim to model other sources of uncertainty, such as last-minute customer cancelations and operator's cancelation on the day of service. Third, we also aim to incorporate customers' preferences on appointment times. 

\bibliographystyle{elsarticle-harv}
\linespread{1}

\bibliography{DHARS}

\appendix
\newpage

\begin{center}
    Stochastic Optimization Models for a Home Service Routing and Appointment Scheduling Problem with Random Travel and Service Times (Appendices)

\end{center}

\section{Comparison with \cite{jiang2019data} and \cite{zhan2020home}}\label{Appx:LR_Compare}
Table \ref{table:literature_comparison} summarizes the differences between \cite{jiang2019data}, \cite{zhan2020home} and our work in multiple aspects, which are two recent single-server papers relevant to our work.

\begin{table}[h!]  
\small
\center
\renewcommand{\arraystretch}{0.9}
\caption{Comparison with \cite{jiang2019data} and \cite{zhan2020home}} 
\begin{tabular}{l|ccccccc}
\toprule 
\textbf{Paper} & \multicolumn{2}{c}{\textbf{Model parameters}} & \multicolumn{2}{c}{\textbf{Opt approach}} & \multicolumn{2}{c}{\textbf{Decisions}} \\
               & Service Time & Travel Time & SP & DRO & Apt scheduling & Routing \\ \midrule
\cite{jiang2019data} & Stochastic & Not considered & \checkmark  & \checkmark (W) & \checkmark &  \\
 \cite{zhan2020home} & Stochastic & Deterministic  & \checkmark  & & \checkmark & \checkmark  \\
 \textbf{Our paper}           & Stochastic & Stochastic     & \checkmark  & \checkmark (M,W) & \checkmark & \checkmark \\
\bottomrule 
\end{tabular}\label{table:literature_comparison}
 \begin{flushleft} \textit{Note:} Opt is optimization, SP is stochastic programming, DRO is distributionally robust optimization, Apt is appointment, W is Wasserstein ambiguity, M is mean-support ambiguity \end{flushleft}
\end{table} 

\section{Proof of Proposition~\ref{Prop1M:DualMinMax}}\label{Appx:ProofOfProb1M}
\begin{proof} 
For a fixed $(\xb,\ab)\in\calX\times\calA$, we can formulate problem \eqref{MDHRAS:InnerMax} as the following linear functional optimization problem. 
\begin{subequations}
\begin{align}
& \max_{\Prob \geq 0} \ \int_{\calS} f(\xb, \ab, \xib )  \ d \Prob  \\
& \ \ \text{s.t.} \   \int_{\calS}  d_i \ d\Prob= \mu_i^{\mbox{\tiny d}}, \quad \quad \forall i \in [N], \label{ConInner:d}\\
& \ \ \ \ \  \ \  \int_{\calS}  t_{i,i'} \ d\Prob= \mu_{i,i'}^{\mbox{\tiny t}}, \quad \quad \forall i \in [0,N], \ i' \in [0,N], \label{ConInner:t}\\
& \ \ \ \ \  \ \ \int_{\calS}   d\Prob= 1. \label{ConInner:Distribution}
\end{align} \label{InnerMax}%
\end{subequations}
Letting  $\rho_i$, $\alpha_{i,i'}$, and $\theta$ be the dual variables associated with constraints \eqref{ConInner:d}, \eqref{ConInner:t} and \eqref{ConInner:Distribution} respectively, we present the dual of problem (\ref{InnerMax}):
\begin{subequations}
\begin{align}
& \min_{\rhob,\,\alphab,\,\theta} \ \sum \limits_{i=1}^N \mu_i^{\mbox{\tiny d}} \rho_i+ \sum_{i=0}^N \sum_{i'=0}^N \mu_{i,i'}^{\mbox{\tiny t}}  \alpha_{i,i'}+\theta \label{DualInner:Obj} \\
& \ \  \text{s.t.} \  \sum \limits_{i=1}^N d_i \rho_i+ \sum_{i=0}^N \sum_{i'=0}^N t_{i,i'} \alpha_{i,i'}+ \theta \geq f(\xb, \ab, \xib),\quad\forall \xib\in \calS , \label{DualInner:PrimalVariabl}
\end{align} \label{DualInnerMax}%
\end{subequations}
where $\rhob\in\R^N$, $\alphab\in\R^{(N+1)\times(N+1)}$ and $\theta\in\R$ are unrestricted in sign, and constraint \eqref{DualInner:PrimalVariabl} is associated with the primal variable $\Prob$. Under the standard assumptions that  (1) $\mu_i^{\mbox{\tiny d}}$ lies in the interior of the set $\lbrace \int_{\calS}  d_i \ d\Q: \Q$ is a probability distribution over $\calS \rbrace$, and (2)  $ \mu_{i,i'}^{\mbox{\tiny t}}$ lies in the interior of the set $\lbrace \int_{\calS}  t_{i,i'} \ d\Q: \Q$ is a probability distribution over $\calS \rbrace$,  strong duality holds such that (\ref{InnerMax}) and (\ref{DualInnerMax}) equal in optimal objective value (\citealp{bertsimas2005optimal, shehadeh2020distributionally}). Note that for fixed $(\rhob, \alphab, \theta)$, constraint \eqref{DualInner:PrimalVariabl} is equivalent to 
$$\theta \geq \max \limits_{\xib \in \calS } \left\{ f(x,a, \db, \tb)- \sum \limits_{i=1}^N d_i \rho_i- \sum_{i=0}^N \sum_{i'=0}^N t_{i,i'} \alpha_{i,i'}\right\}.$$
Since we are minimizing $\theta$ in \eqref{DualInnerMax}, the dual formulation of \eqref{InnerMax} is equivalent to:
\begin{align*} 
& \min_{\alphab,\,\rhob}  \ \Bigg \{ \sum \limits_{i=1}^N \mu_i^{\mbox{\tiny d}} \rho_i+ \sum_{i=1}^N \sum_{i'=1}^N \mu_{i,i'}^{\mbox{\tiny t}}  \alpha_{i,i'} +\max \limits_{\xib \in \calS }  \Big \{  f(x,a, \db, \tb)- \sum \limits_{i=1}^N d_i \rho_i- \sum_{i=0}^N \sum_{i'=0}^N t_{i,i'} \alpha_{i,i'} \} \Big \}  \Bigg\} \\
& \ \ \text{s.t.} \ \ \alphab\in\R^{(N+1)\times(N+1)},\, \rhob\in \R^N.
\end{align*}
This completes the proof.
\end{proof}

\section{Proof of Proposition~\ref{Prop2M}}\label{Appx:ProofProp2M}
\begin{proof} 
In view of the objective function \eqref{Dual2nd:Obj}, we consider all the terms involving $y$ and define the function $H$ as
$$H(\yb) =   -a_1 y_1 + \sum_{j=2}^{N+1} \big ( a_{j-1}-a_j) y_j+ \sum_{j=2}^{N+1} \sum_{i=1}^{N}  d_ix_{i,j-1} y_j  + \sum_{i=1}^N t_{0,i}x_{i,1}y_1 + \sum_{j=2}^N \sum_{i=1}^N \sum_{i\neq i'} t_{i,i'}x_{i,j-1}x_{i',j}y_j. $$
We are maximizing $H(\yb)$ over a polyhedral set $\yb\in\calY$. Note that $H$ is a linear (convex) function in $\calY$ and we are maximizing $H$ over a convex compact set $\calY$. From basic convex analysis, we know that there exists an optimal solution at some extreme point of $\calY$. Recall the definition of $\calY$. 
$$\calY= \Big\{\yb \mid 0\leq y_{N+1} \leq \co, \ -\cu_j\leq y_j \leq \cw_j+ y_{j+1}, \ \forall j \in [N]\Big\}$$
We can apply a similar technique in Proposition 2 in \cite{jiang2019data} and Proposition 3 in \cite{shehadeh2020distributionally} to derive an equivalent formulation for the maximization problem. To characterize the extreme points of $\calY$, we introduce dummy variables $y_{N+2}$ with $\cw_{N+1}=\co$, $\cu_{N+1}=0$, $\cw_{N+2}=0$ and $\cu_{N+2}=0$. Then, we rewrite the set $\calY$ as 
\begin{equation*}
\calY=\{\yb\mid y_{N+2}=-\cu_{N+2},\,-\cu_j\leq y_j\leq \cw_j+y_{j+1},\,j\in[N+1]\}.
\end{equation*}
Note that an extreme point of $\calY$ satisfies (i) $y_{N+2}=-\cu_{N+2}=0$ and (ii) for $j\in[N+1]$, the variable constraint on $y_j$ is binding either at the lower bound or the upper bound. If $y_j$ is binding at the upper bound, it does depend on $y_{j+1}$ but if it is binding at the lower bound, it just takes the value of $-c^u_j$. From this, we can construct a one-to-one correspondence between the extreme points in $\calY$ and partitions of the set $\{1,\dots,N+2\}$. That is, given an interval $[k,v]_\Z$ in $[N+2]$, $y_v$ is binding at the lower bound, i.e. $y_v=-c_v^u$, and for $j\in[k,v-1]$, $y_j$ is binding at the upper bound, i.e. $y_j=-\cu_j+\sum_{l=j}^{v-1}\cw_l$. Therefore, for notational simplicity, we define $\pi_{j,v}=-\cu_v+\sum_{l=j}^{v-1}\cw_l$ for $1\leq j\leq v\leq N+2$. We can reformulate our optimization problem over $\calY$ to optimizing over partitions of $\{1,\dots,N+2\}$.

To do so, let $b_{k,v}$ be the binary variable with $1$ indicating that the interval $[k,v]$ belongs to an element of a partition for $1\leq k\leq v\leq N+2$. The condition $$\sum_{k=1}^j \sum_{v=j}^{N+2} b_{k,v}=1,\,\forall j\in[N+2]$$ is equivalent to saying that $\{[k,v]_\Z\mid b_{k,v}=1,\,k\in[N+2],\,v\in[k,N+2]_\Z\}$ is a partition of $[N+2]$. With the use of the equality for any extreme point $y$, we have $$y_j=\sum_{k=1}^j \sum_{v=j}^{N+2} y_jb_{k,v}=\sum_{k=1}^j \sum_{v=j}^{N+2} \pi_{j,v}b_{k,v}\,,$$ and the problem maximizing $H(\yb)$ over $\yb\in\calY$ is equivalent to the following integer program, 
\begin{subequations}\label{IPFormulation}
\begin{align}
\max_{\bb} &   \sum_{v=1}^{N+2} -a_1 \pi_{1,v} b_{1,v}+\sum_{j=2}^{N+1} \sum_{k=1}^j \sum_{v=j}^{N+2} ( a_{j-1}-a_j) \pi_{j,v} b_{k,v} \nonumber +\sum_{j=2}^{N+1} \sum_{k=1}^j \sum_{v=j}^{N+2}  \sum_{i=1}^{N}  d_ix_{i, j-1}\pi_{j,v} b_{k,v}\nonumber \\
& \ \ +\sum_{v=1}^{N+2}\sum_{i=1}^N  t_{0,i} x_{i,1}\pi_{1,v} b_{1,v} +\sum_{j=2}^{N+1} \sum_{k=1}^j \sum_{v=j}^{N+2} \sum_{i=1}^N \sum_{i'\neq i} t_{i,i'}x_{i,j-1} x_{i', j}\pi_{j,v} b_{k,v} \label{IPObj} \\
\text{s.t.} &  \ \sum_{k=1}^j \sum_{v=j}^{N+2} b_{k,v}=1, \qquad \forall j \in [N+2], \label{IPConstraint}\\ 
& \  b_{k,v} \in \{0, 1\}, \quad \forall k \in [N+2], \forall v \in [k, N+2]_\Z. \label{IPConst}
\end{align}
\end{subequations}
We remark that the last summation in \eqref{IPObj} is summed from $j=2$ up to $N+1$ instead of $N$ since $x_{i,N+1}=0$ for all $i$. This makes the last summation term can behave similarly to the previous terms, which simplifies our discussion.

Now, we consider the problem 
$$\max \limits_{\xib \in \calS  }  \Bigg \{  f(\xb,\ab, \xib)- \sum \limits_{i=1}^N d_i \rho_i- \sum_{j=2}^N \sum_{i=1}^N \sum_{i'\ne1} t_{i,i'} \alpha_{i,i'} x_{i,j-1} x_{i',j}-\sum_{i=1}^N t_{0,i}\alpha_{0,i}x_{i,1} - \sum_{i=1}^N t_{i,0}\alpha_{i,0}x_{i,N} \Bigg \}.$$
With the use of \eqref{IPFormulation}, we can reformulate it as
\begin{subequations}\label{IPFull}
\begin{align}
\max_{\bb,\,\xib} &  \sum_{v=1}^{N+2} -a_1 \pi_{1,v} b_{1,v}+\sum_{j=2}^{N+1} \sum_{k=1}^j \sum_{v=j}^{N+2} ( a_{j-1}-a_j) \pi_{j,v} b_{k,v} \nonumber +\sum_{j=2}^{N+1} \sum_{k=1}^j \sum_{v=j}^{N+2}  \sum_{i=1}^{N}  d_ix_{i, j-1}\pi_{j,v} b_{k,v} \nonumber \\
& \ \ +\sum_{v=1}^{N+2}\sum_{i=1}^N  t_{0,i}x_{i,1}\pi_{1,v}  b_{1,v} +\sum_{j=2}^{N+1} \sum_{k=1}^j \sum_{v=j}^{N+2} \sum_{i=1}^N \sum_{i'\neq i} t_{i,i'}x_{i,j-1} x_{i', j}\pi_{j,v} b_{k,v} + \lambda\sum_{i=1}^N t_{0,i}x_{i,1} \nonumber \\
& \ \ +\lambda\sum_{i=1}^Nt_{i,0}x_{i,N}+\lambda\sum_{j=2}^{N+1}\sum_{i=1}^N\sum_{i'\ne i} t_{i,i'}x_{i,j-1}x_{i',j} - \sum_{i=1}^N \rho_i d_i - \sum_{j=2}^N \sum_{i=1}^N \sum_{i'\ne i} t_{i,i'} \alpha_{i,i'} x_{i,j-1} x_{i',j} \nonumber \\
& \ \ - \sum_{i=1}^N t_{0,i}\alpha_{0,i}x_{i,1} - \sum_{i=1}^N t_{i,0}\alpha_{i,0}x_{i,N} \label{IPObjFull} \\
\text{s.t.} &  \ \sum_{k=1}^j \sum_{v=j}^{N+2} b_{k,v}=1, \qquad \forall j \in [N+2], \label{IPConstraintFull}\\ 
& \  b_{k,v} \in \{0, 1\}, \quad \forall k \in [N+2], \forall v \in [k, N+2]_\Z, \label{IPConstFull} \\
& \ \xib \in \calS.
\end{align}
\end{subequations}
By the construction of $\calS$, we can take maximum over each $d_i$ and $t_{i,i'}$. To achieve this, we first gather all the terms involving $d_i$ and $t_{i,i'}$ separately and take supremum accordingly. We can reformulate the objective function \eqref{IPObjFull} as
\begin{align}\label{IPFull2}
&  \sum_{v=1}^{N+2} -a_1 \pi_{1,v} b_{1,v}+\sum_{j=2}^{N+1} \sum_{k=1}^j \sum_{v=j}^{N+2} ( a_{j-1}-a_j) \pi_{j,v} b_{k,v} \nonumber +\sum_{i=1}^{N} \left[\sum_{j=2}^{N+1} \sum_{k=1}^j \sum_{v=j}^{N+2} \left(\pi_{j,v}-\rho_i\right)\right]x_{i, j-1}b_{k,v}d_i  \nonumber \\
& \ \ +\sum_{i=1}^N \sum_{v=1}^{N+2}  [(\pi_{1,v}+\lambda)-\alpha_{0,i}]x_{i,1} b_{1,v}t_{0,i}  + \sum_{i=1}^N (\lambda-\alpha_{i,0})x_{i,N}t_{i,0}  \nonumber \\
& \ \ +\sum_{i=1}^N \sum_{i'\neq i} \left\{\sum_{j=2}^{N+1} \sum_{k=1}^j \sum_{v=j}^{N+2}  \left[(\pi_{j,v}+\lambda)-\alpha_{i,i'}\right]x_{i,j-1} x_{i', j}b_{k,v} \right\} t_{i,i'} ,
\end{align}
where we focus on the third to the last term. We first define $\Delta d_i=\overline{d}_i-\underline{d}_i$ and $\Delta t_{i,i'}=\overline{t}_{i,i'}-\underline{t}_{i,i'}$. Note that 
\begin{align} \label{IP:5term}
\sup_{t_{i,0}\in[\underline{t}_{i,0},\overline{t}_{i,0}],\, i\in[N]} \sum_{i=1}^N (\lambda-\alpha_{i,0})x_{i,N}t_{i,0}=\sum_{i=1}^N\left[(\lambda-\alpha_{i,0})\underline{t}_{i,0}+\Delta t_{i,0}(\lambda-\alpha_{i,0})^+\right]x_{i,N}.
\end{align}
For the forth term, for any given $\xb\in\calX$ and feasible $b$, since $x_{i,1}=1$ and $b_{1,v}=1$ for exactly one $i$ and $v$, say $i_1$ and $v_1$, we have
\begin{align}
&\quad\  \sup_{t_{0,i}\in[\underline{t}_{0,i},\overline{t}_{0,i}],\,i\in[N]} \sum_{i=1}^N \sum_{v=1}^{N+2}  x_{i,1}(\pi_{1,v}+\lambda-\alpha_{0,i}) b_{1,v}t_{0,i}  \nonumber \\
&= \sup_{t_{0,i_1}\in[\underline{t}_{0,i_1},\overline{t}_{0,i_1}]} (\pi_{1,v_1}+\lambda-\alpha_{0,i_1}) t_{0,i_1} \nonumber\\
&= (\pi_{1,v_1}+\lambda-\alpha_{0,i_1})\underline{t}_{0,i_1} + \Delta t_{0,i}(\pi_{1,v_1}+\lambda-\alpha_{0,i_1})^+ \nonumber\\
&= \sum_{i=1}^N \sum_{v=1}^{N+2} \left[(\pi_{1,v}+\lambda-\alpha_{0,i})\underline{t}_{0,i} + \Delta t_{0,i}(\pi_{1,v}+\lambda-\alpha_{0,i})^+\right] x_{i,1}b_{1,v}. \label{IP:4term}
\end{align}
For the third term, notice that for any given $\xb\in\calX$, $i$ and $j$ have a one-to-one correspondence. That is, for any $i$, we can identify exactly one $j=j_i$ such that $x_{i,j-1}=1$. Then,
$$\left[\sum_{j=2}^{N+1} \sum_{k=1}^j \sum_{v=j}^{N+2} \left(\pi_{j,v}-\rho_i\right)x_{i, j-1}b_{k,v} \right]d_i = \sum_{k=1}^{j_i}\sum_{v=j_i}^{N+2}\left(\pi_{j_i,v}-\rho_i\right)b_{k,v} d_i=(\pi_{j_i,v_i}-\rho_i)d_i.$$
since $b_{k,v}=1$ only for one pair of $(k,v)$, say $(k_i,v_i)$. Hence, we have 
 \allowdisplaybreaks
\begin{align}
&\quad\, \sup_{d_i\in[\underline{d}_i,\overline{d}_i],\,i\in[N]} \sum_{i=1}^N \left[\sum_{j=2}^{N+1} \sum_{k=1}^j \sum_{v=j}^{N+2} \left(\pi_{j,v}-\rho_i\right)x_{i, j-1}b_{k,v} \right]d_i  \nonumber\\
&=\sum_{i=1}^N \sup_{d_i\in[\underline{d}_i,\overline{d}_i],\,i\in[N]} \left[\sum_{j=2}^{N+1} \sum_{k=1}^j \sum_{v=j}^{N+2} \left(\pi_{j,v}-\rho_i\right)x_{i, j-1}b_{k,v} \right]d_i  \nonumber\\
&=\sum_{i=1}^N \sup_{d_i\in[\underline{d}_i,\overline{d}_i],\,i\in[N]} (\pi_{j_i,v_i}-\rho_i) d_i   \nonumber\\
&= \sum_{i=1}^N \left[(\pi_{j_i,v_i}-\rho_i)\underline{d}_i + \Delta d_i (\pi_{j_i,v_i}-\rho_i)^+\right] \nonumber\\
&=\sum_{i=1}^N \sum_{j=2}^{N+1} \sum_{k=1}^j \sum_{v=j}^{N+2}\left[(\pi_{j,v}-\rho_i)\underline{d}_i + \Delta d_i (\pi_{j,v}-\rho_i)^+\right]x_{i,j-1}b_{k,v}. \label{IP:3term}
\end{align}
The last term involves the quadratic term $x_{i,j-1}x_{i',j}$. For any given pair $(i,i')$, there exists at most one $j$, say $j_i$, such that $x_{i,j-1}x_{i',j}=1$. If there does not exist such an $j_i$, the triple summation term is just zero. Otherwise, there exists such an $j_i$ and the last term reads 
$$\left\{\sum_{j=2}^{N+1} \sum_{k=1}^j \sum_{v=j}^{N+2}  \left[(\pi_{j,v}+\lambda)-\alpha_{i,i'}\right]x_{i,j-1} x_{i', j}b_{k,v} \right\} t_{i,i'}=(\pi_{j_i,v_i}+\lambda)-\alpha_{i,i'}.$$
With the use of these two observations, we have
 \allowdisplaybreaks
\begin{align}
&\quad\, \sup_{t_{i,i'}\in[\underline{t}_{i,i'},\overline{t}_{i,i'}],\, i\in[N], i'\ne i} \sum_{i=1}^N \sum_{i'\neq i} \left\{\sum_{j=2}^{N+1} \sum_{k=1}^j \sum_{v=j}^{N+2}  \left[(\pi_{j,v}+\lambda)-\alpha_{i,i'}\right]x_{i,j-1} x_{i', j}b_{k,v} \right\} t_{i,i'} \nonumber\\
&=  \sum_{i=1}^N \sum_{i'\neq i} \sum_{j=2}^{N+1} \sum_{k=1}^j \sum_{v=j}^{N+2}  \left[(\pi_{j,v}+\lambda-\alpha_{i,i'})\underline{t}_{i,i'} + \Delta t_{i,i'} (\pi_{j,v}+\lambda-\alpha_{i,i'})^+\right]x_{i,j-1}x_{i',j}b_{k,v} \label{IP:6term}
\end{align}
Eventually, we have reformulated the objective function using \eqref{IP:5term}, \eqref{IP:4term}, \eqref{IP:3term} and \eqref{IP:6term}. 

Note that from \eqref{IPConstraintFull}, the coefficient matrix is totally unimodular and hence, we can relax the integer constraint \eqref{IPConstFull} to $b_{k,v}\geq0$. This reformulates \eqref{IPFull} into an LP in $\bb$. One can observe that we put $b_{k,v}$ to the outermost position in all the terms so that we are able to collect the terms with the same $b_{k,v}$ by interchanging the summations. Indeed, for some arbitrary terms $z_{jkv}$, the triple summation can be written as 
$$\sum_{j=2}^{N+1} \sum_{k=1}^j \sum_{v=j}^{N+2}z_{jkv}=\sum_{v=2}^{N+2}\sum_{j=2}^{\min(v,N+1)} z_{j1v} + \sum_{k=2}^{N+1}\sum_{v=k}^{N+2}\sum_{j=k}^{\min(v,N+1)}  z_{jkv}.$$ 
Then, we are able to take the dual of the resulting LP. Let $\beta_j$ be the dual variable associated to the constraint \eqref{IPConstraintFull} for $j=1,\dots,N+2$. Thus, we arrive at the following program.
\begin{subequations}
\begin{align}
 \min_{\betab}  &\  \sum \limits_{j=1}^{N+2} \beta_j  +\sum_{i=1}^N\left[(\lambda-\alpha_{i,0})\underline{t}_{i,0}+\Delta t_{i,0}(\lambda-\alpha_{i,0})^+\right]x_{i,N} \\
\text{s.t.} & \ \beta_1 \geq -a_1  \pi_{1,1} +\sum_{i=1}^N \big [\underline{t}_{0,i} (\pi_{1,1}+\lambda-\alpha_{0,i})+\Delta t_{0,i} (\pi_{1,1}+\lambda-\alpha_{0,i})^+\big] x_{i,1},\\
& \ \sum_{j=1}^v \beta_j \geq -a_1  \pi_{1,v} + \sum_{j=2}^{\min(v,N+1)} \big ( a_{j-1}-a_j \big) \pi_{j,v}\nonumber \\
& \ \qquad \   + \sum_{i=1}^N \big [\underline{t}_{0,i} (\pi_{1,v}+\lambda-\alpha_{0,i})+\Delta t_{0,i} (\pi_{1,v}+\lambda-\alpha_{0,i})^+\big] x_{i,1} \nonumber \\
& \ \qquad \   + \sum_{j=2}^{\min(v,N+1)} \sum_{i=1}^N \sum_{i' \neq i} \big[ \underline{t}_{i,i'} (\pi_{j,v}+\lambda-\alpha_{i,i'})+ \Delta t_{i,i'} (\pi_{j,v}+\lambda-\alpha_{i,i'})^+  \big]x_{i,j-1}x_{i', j} \nonumber\\ 
& \ \qquad \   + \sum_{j=2}^{\min(v,N+1)} \sum_{i=1}^N \big[ \underline{d}_i (\pi_{j,v}-\rho_i)+ \Delta d_i (\pi_{j,v}-\rho_i)^+ \big ] x_{i,j-1}, \ \
 \forall v \in [2, N+2]_{\Z}, \\
& \ \sum_{j=k}^v\beta_j \geq \sum_{j=k}^{\min(v,N+1)} \big ( a_{j-1}-a_j\big)  \pi_{j,v}+\sum_{j=k}^{\min(v,N+1)} \sum_{i=1}^N \big[ \underline{d}_i (\pi_{j,v}-\rho_i)+ \Delta d_i (\pi_{j,v}-\rho_i)^+ \big ] x_{i,j-1} \nonumber\\
& \ \qquad \    \ + \sum_{j=k}^{\min(v,N+1)} \sum_{i=1}^N \sum_{i' \neq i}  \bigg\{ (\pi_{j,v}+\lambda-\alpha_{i,i'})\underline{t}_{i,i'}+\Delta t_{i,i'}(\pi_{j,v}+\lambda-\alpha_{i,i'})^+  \bigg\}x_{i,j-1}x_{i',j} \nonumber \\
& \ \qquad \    \ \forall j\in[2,N+1]_\Z,\, v\in[k,N+2]_\Z, \\
& \ \beta_{N+2} \geq 0.
\end{align}
\end{subequations}

Finally, we can introduce auxiliary variables to replace the terms with $(\cdot)^+$. In particular, we let $\gamma_{0,i,1,v}=(\pi_{1,v}+\lambda-\alpha_{0,i})^+$, $\gamma_{i,i',j,v}=(\pi_{j,v}+\lambda-\alpha_{i,i'})^+$, $\delta_{i,j,v} = (\pi_{j,v}-\rho_i)^+$ and introduce the following constraints.
\begin{subequations}
\begin{align}
& \ \gamma_{0,i,1,v}\geq 0,\quad \gamma_{0,i,1,v}\geq \pi_{1,v}+\lambda-\alpha_{0,i},\quad \forall i\in[N],\,\forall v\in[N+2] \label{M_DHRAS_L:Con5} \\
& \ \gamma_{i,i',j,v}\geq 0,\quad \gamma_{i,i',j,v}\geq \pi_{j,v}+\lambda-\alpha_{i,i'},\quad \forall i\in[N],\,\forall i'\in[N]\setminus\{i\}, \nonumber \\
& \hspace{69mm}  \forall j\in[2,N+1]_\Z,\,\forall v\in[j,N+2]_\Z  \label{M_DHRAS_L:Con6}  \\
& \ \delta_{i,j,v}\geq 0,\quad \delta_{i,j,v}\geq \pi_{j,v} -\rho_i,\quad\forall i\in[N],\, \forall j\in[2,N+1]_\Z,\,\forall v\in[j,N+2]_\Z  \label{M_DHRAS_L:Con7}
\end{align}
\end{subequations}
This completes the reformulation.
\end{proof}

\section{Proof of Corollary~\ref{Coro3}} \label{apdx: Coro_proof}
\begin{proof}
In view of Propositions \ref{Prop1M:DualMinMax} and \ref{Prop2M}, problem \eqref{eq:FinalDualInnerMax-1} (i.e., the inner problem  $\sup \limits_{\Prob \in \calF(\calS,\mu) }\E_{\Prob} [f(x,a, \db) ]$ in \eqref{M-DHRAS}) is equivalent to 
\begin{subequations}\label{InnerMINLP}
\begin{align} 
 \min_{\alphab,\,\rhob,\,\betab} & \  \sum \limits_{i=1}^N \mu_i^{\mbox{\tiny d}} \rho_i+ \sum_{j=2}^N\sum_{i=1}^N \sum_{i'\ne i} \mu_{i,i'}^{\mbox{\tiny t}}  \alpha_{i,i'}x_{i,j-1}x_{i',j}  + \sum \limits_{i=1}^N \mu_{i,0}^{\mbox{\tiny t}} \alpha_{i,0} x_{i,N} + \sum \limits_{i=1}^N \mu_{0,i}^{\mbox{\tiny t}} \alpha_{0,i} x_{i,1} \nonumber \\
 & \ \ \ + \sum \limits_{j=1}^{N+2} \beta_j  +\sum_{i=1}^N\left[(\lambda-\alpha_{i,0})\underline{t}_{i,0}+\Delta t_{i,0}(\lambda-\alpha_{i,0})^+\right]x_{i,N} \label{InnerMINLP:Obj}\\
\text{s.t.} & \ \alphab\in\R^{(N+1)\times(N+1)},\, \rhob\in\R^N,\, \betab\in\R^{N+2},\, \eqref{DHRASMINLP:Con1} - \eqref{DHRASMINLP:Con7}.
\end{align}
\end{subequations}
Noting that the variables $\alpha_{i,0}$ for $i\in[N]$ appear in the objective function only and it is easy to see that the minimizer is given by $\lambda$. Indeed, consider the terms involving $\alpha_{i,0}$, i.e., 
$$ \mut_{i,0}\alpha_{i,0}x_{i,N}+\left[(\lambda-\alpha_{i,0})\underline{t}_{i,0}+\Delta t_{i,0}(\lambda-\alpha_{i,0})^+\right]x_{i,N}.$$
If $\alpha_{i,0} \geq \lambda$, the term reads $[ (\mut_{i,0}-\tlb_{i,0})\alpha_{i,0}+\lambda  \tlb_{i,0}]x_{i,N}$, which is non-decreasing in $\alpha_{i,0}$ and if $\alpha_{i,0} \leq \lambda$, the terms read $[ (\mut_{i,0}-\tub_{i,0})\alpha_{i,0}+\lambda  \tlb_{i,0}]x_{i,N}$, which is non-increasing in $\alpha_{i,0}$. Plugging in the optimal $\alpha_{i,0}=\lambda$ to the objective function, we arrive at the desired reformulation.
\end{proof}

\section{Details of the final MINLP and MILP reformulation of M-DHRAS} \label{Appx:Mac_MDHRAS}

From Corollary \ref{Coro3}, we can reformulate M-DHRAS as the following MINLP.
\allowdisplaybreaks
\begin{subequations}
\begin{align} 
 \min & \  \sum \limits_{i=1}^N \mu_i^{\mbox{\tiny d}} \rho_i+ \sum_{j=2}^N\sum_{i=1}^N \sum_{i'\ne i} \mut_{i,i'}  \alpha_{i,i'}x_{i,j-1}x_{i',j}  + \lambda \sum \limits_{i=1}^N \mut_{i,0}  x_{i,N} + \sum \limits_{i=1}^N \mut_{0,i} \alpha_{0,i} x_{i,1}  + \sum \limits_{j=1}^{N+2} \beta_j  \\
\textup{s.t.} & \ \xb\in\calX,\, \ab\in\calA,\, \alphab\in\R^{(N+1)\times(N+1)},\, \rhob\in\R^N,\, \betab\in\R^{N+2}, \\
 & \ \beta_1 \geq -a_1  \pi_{1,1} +\sum_{i=1}^N \big [\underline{t}_{0,i} (\pi_{1,1}+\lambda-\alpha_{0,i})+\Delta t_{0,i} \gamma_{0,i,1,1}\big] x_{i,1} , \\
& \ \sum_{j=1}^v \beta_j \geq -a_1  \pi_{1,v} + \sum_{j=2}^{\min(v,N+1)} \big ( a_{j-1}-a_j \big) \pi_{j,v}   + \sum_{i=1}^N \big [\underline{t}_{0,i} (\pi_{1,v}+\lambda-\alpha_{0,i})+\Delta t_{0,i} \gamma_{0,i,1,v}\big] x_{i,1}  \nonumber \\
& \ \qquad \   + \sum_{j=2}^{\min(v,N+1)} \sum_{i=1}^N \sum_{i' \neq i} \big[ \underline{t}_{i,i'} (\pi_{j,v}+\lambda-\alpha_{i,i'})+ \Delta t_{i,i'} \gamma_{i,i',j,v}  \big]x_{i,j-1}x_{i', j} \nonumber\\ 
& \ \qquad \   + \sum_{j=2}^{\min(v,N+1)} \sum_{i=1}^N \big[ \underline{d}_i (\pi_{j,v}-\rho_i)+ \Delta d_i \delta_{i,j,v} \big ] x_{i,j-1}, \ \
 \forall v \in [2, N+2]_{\Z} ,  \\
& \ \sum_{j=k}^v\beta_j \geq \sum_{j=k}^{\min(v,N+1)} \big ( a_{j-1}-a_j\big)  \pi_{j,v}+\sum_{j=k}^{\min(v,N+1)} \sum_{i=1}^N \big[ \underline{d}_i (\pi_{j,v}-\rho_i)+ \Delta d_i \delta_{i,j,v} \big ] x_{i,j-1} \nonumber\\
& \ \qquad \    \ + \sum_{j=k}^{\min(v,N+1)} \sum_{i=1}^N \sum_{i' \neq i}  \bigg\{ (\pi_{j,v}+\lambda-\alpha_{i,i'})\underline{t}_{i,i'}+\Delta t_{i,i'}\gamma_{i,i',j,v}  \bigg\}x_{i,j-1}x_{i',j},   \nonumber \\
& \ \qquad \    \  \forall k\in[2,N+1]_\Z,\,\forall v\in[k,N+2]_\Z , \\
& \ \beta_{N+2} \geq 0 , \\
& \ \gamma_{0,i,1,v}\geq 0,\quad \gamma_{0,i,1,v}\geq \pi_{1,v}+\lambda-\alpha_{0,i},\quad \forall i\in[N],\,\forall v\in[N+2] , \\
& \ \gamma_{i,i',j,v}\geq 0,\quad \gamma_{i,i',j,v}\geq \pi_{j,v}+\lambda-\alpha_{i,i'},\quad \forall i\in[N],\,\forall i'\in[N]\setminus\{i\}, \nonumber \\
& \hspace{70mm} \forall j\in[2,N+1]_\Z,\,\forall v\in[j,N+2]_\Z ,   \\
& \ \delta_{i,j,v}\geq 0,\quad \delta_{i,j,v}\geq \pi_{j,v} -\rho_i,\quad\forall i\in[N],\,\, \forall j\in[2,N+1]_\Z,\,\forall v\in[j,N+2]_\Z .
\end{align}
\end{subequations}
Note that this is non-linear due to the interaction terms (such as $x_{i,j-1}x_{i',j}$ in the objective, $\alpha_{i,i'} x_{i,j-1}x_{i',j}$, $\gamma_{0,i,1,v}x_{i,1}$ and $\rho_i x_{i,j-1}$ in the constraints). To linearize this MINLP formulation, we define variables $\tau_{i,i',j-1, j}= x_{i,j-1}x_{i',j}$, $\eta_{i,i',j-1, j}= \alpha_{i,i'} \tau_{i,i',j-1, j}$, $\psi_{0,i}=\alpha_{0,i} x_{i,1}$, $\sigma_{0,i,1,v}=\gamma_{0,i,1,v}x_{i,1}$, $\phi_{i,i',j,v}=\gamma_{i,i',j,v}\tau_{i,i',j-1,j}$, $\xi_{i,j,v}=\delta_{i,j,v}x_{i,j-1}$ and $\zeta_{i,j}=\rho_i x_{i,j-1}$. We also introduce McCormick inequalities \eqref{MILP-MAC1M} to \eqref{MILP-MAC11M} for these variables.
\begin{subequations}
\begin{align}
&  \tau_{i,i',j-1,j} \geq x_{i,j-1}+x_{i',j}-1, \ \tau_{i,i',j-1,j} \geq 0, \label{MILP-MAC1M}\\
&  \tau_{i,i',j-1,j} \leq x_{i,j-1}, \ \tau_{i,i',j-1,j} \leq x_{i',j} \label{MILP-MAC2M}\\
& \eta_{i,i',j-1,j} \geq \alphalb_{i,i'} \tau_{i,i',j-1,j}, \ \eta_{i,i',j-1,j} \geq \alpha_{i,i'}+\alphaub_{i,i'}(\tau_{i,i',j-1,j}-1), \label{MILP-MAC3M}\\
&  \eta_{i,i',j-1,j}  \leq \alphaub_{i,i'} \tau_{i,i',j-1,j}, \ \eta_{i,i',j-1,j} \leq \alpha_{i,i'}+\alphalb_{i,i'}(\tau_{i,i',j-1,j}-1) \label{MILP-MAC4M}\\
& \psi_{0,i} \geq \alphalb_{0,i}x_{i,1}, \ \psi_{0,i} \geq \alpha_{0,i}+\alphaub_{0,i} (x_{i,1}-1), \ \psi_{0,i} \leq \alphaub_{0,i} x_{i,1}, \ \psi_{0,i} \leq \alpha_{0,i} + \alphalb_{0,i} (x_{i,1}-1) \label{MILP-MAC5M}\\
&  \phi_{i,i',j,v} \geq \gamma_{i,i',j,v}+\gammaub_{i,i',j,v}(\tau_{i,i',j-1,j}-1), \ \phi_{i,i',j,v}\geq 0 \label{MILP-MAC6M}\\
&  \phi_{i,i',j,v} \leq \gammaub_{i,i',j,v}\tau_{i,i',j-1,j}, \ \phi_{i,i',j,v}\leq \gamma_{i,i',j,v} \label{MILP-MAC7M} \\
&  \sigma_{0,i,1,v} \geq \gamma_{0,i,1,v}+\gammaub_{0,i,1,v}(x_{i,1}-1), \ \sigma_{0,i,1,v}\geq 0 \label{MILP-MAC8M} \\
&  \sigma_{0,i,1,v} \leq \gammaub_{0,i,1,v}x_{i,1}, \ \sigma_{0,i,1,v}\leq \gamma_{0,i,1,v} \label{MILP-MAC9M} \\
&  \xi_{i,j,v} \geq \delta_{i,j,v}+\deltaub_{i,j,v}(x_{i,j-1}-1), \ \xi_{i,j,v} \geq 0, \ \xi_{i,j,v} \leq \delta_{i,j,v}, \ \xi_{i,j,v}\leq\deltaub_{i,j,v}x_{i,j-1} \label{MILP-MAC10M}\\
& \zeta_{i,j} \geq \rholb_{i}x_{i,j-1}, \ \zeta_{i,j} \geq \rho_{i}+\rhoub_{i} (x_{i,j-1}-1), \ \zeta_{i,j} \leq \rhoub_{i} x_{i,j-1}, \ \zeta_{i,j} \leq \rho_{i} + \rholb_{i} (x_{i,j-1}-1) \label{MILP-MAC11M}
\end{align}
\end{subequations}
The notations $\overline{\cdot}$ and $\underline{\cdot}$ are the upper and lower bounds for the variable respectively. Therefore, we obtain an MILP reformulation of the M-DHRAS model.
\allowdisplaybreaks
\begin{subequations}
\begin{align} 
 \min & \  \sum \limits_{i=1}^N \mu_i^{\mbox{\tiny d}} \rho_i+ \sum_{j=2}^N\sum_{i=1}^N \sum_{i'\ne i} \mut_{i,i'}  \eta_{i,i',j-1,j}  + \lambda \sum \limits_{i=1}^N \mut_{i,0}  x_{i,N} + \sum \limits_{i=1}^N \mut_{0,i} \psi_{0,i}  + \sum \limits_{j=1}^{N+2} \beta_j  \\
\textup{s.t.} & \ \xb\in\calX,\, \ab\in\calA,\, \alphab\in\R^{(N+1)\times(N+1)},\, \rhob\in\R^N,\, \betab\in\R^{N+2}, \\
& \text{constraints } \eqref{MILP-MAC1M}-\eqref{MILP-MAC11M}, \\
 & \ \beta_1 \geq -a_1  \pi_{1,1} +\sum_{i=1}^N \big [\underline{t}_{0,i} (\pi_{1,1}x_{i,1}+\lambda x_{i,1}-\psi_{0,i})+\Delta t_{0,i} \sigma_{0,i,1,1}\big], \label{M-DHRAS_MILP_Con2}\\
& \ \sum_{j=1}^v \beta_j \geq -a_1  \pi_{1,v} + \sum_{j=2}^{\min(v,N+1)} \big ( a_{j-1}-a_j \big) \pi_{j,v}   + \sum_{i=1}^N \big [\underline{t}_{0,i} (\pi_{1,v}x_{i,1}+\lambda x_{i,1}-\psi_{0,i})+\Delta t_{0,i} \sigma_{0,i,1,v}\big]   \nonumber \\
& \ \qquad \   + \sum_{j=2}^{\min(v,N+1)} \sum_{i=1}^N \sum_{i' \neq i} \big[ \underline{t}_{i,i'} (\pi_{j,v}\tau_{i,i',j-1,j}+\lambda\tau_{i,i',j-1,j}-\eta_{i,i',j-1,j})+ \Delta t_{i,i'} \phi_{i,i',j,v}  \big] \nonumber\\ 
& \ \qquad \   + \sum_{j=2}^{\min(v,N+1)} \sum_{i=1}^N \big[ \underline{d}_i (\pi_{j,v} x_{i,j-1}-\zeta_{i,j})+ \Delta d_i \xi_{i,j,v} \big ], \
 \forall v \in [2, N+2]_{\Z} , \label{M-DHRAS_MILP_Con3} \\
& \ \sum_{j=k}^v\beta_j \geq \sum_{j=k}^{\min(v,N+1)} \big ( a_{j-1}-a_j\big)  \pi_{j,v}+\sum_{j=k}^{\min(v,N+1)} \sum_{i=1}^N \big[ \underline{d}_i (\pi_{j,v} x_{i,j-1} -\zeta_{i,j})+ \Delta d_i \xi_{i,j,v} \big ]  \nonumber\\
& \ \qquad \    \ + \sum_{j=k}^{\min(v,N+1)} \sum_{i=1}^N \sum_{i' \neq i}  \bigg\{ (\pi_{j,v}\tau_{i,i',j-1,j}+\lambda\tau_{i,i',j-1,j}-\eta_{i,i',j-1,j})\underline{t}_{i,i'}+\Delta t_{i,i'}\phi_{i,i',j,v}  \bigg\},   \nonumber \\
& \ \qquad \    \  \forall k\in[2,N+1]_\Z,\,\forall v\in[k,N+2]_\Z , \label{M-DHRAS_MILP_Con4}\\
& \ \beta_{N+2} \geq 0 , \label{M-DHRAS_MILP_Con5}\\
& \ \gamma_{0,i,1,v}\geq 0,\quad \gamma_{0,i,1,v}\geq \pi_{1,v}+\lambda-\alpha_{0,i},\quad \forall i\in[N],\,\forall v\in[N+2] , \\
& \ \gamma_{i,i',j,v}\geq 0,\quad \gamma_{i,i',j,v}\geq \pi_{j,v}+\lambda-\alpha_{i,i'},\quad \forall i\in[N],\,\forall i'\in[N]\setminus\{i\}, \nonumber \\
& \hspace{70mm} \forall j\in[2,N+1]_\Z,\,\forall v\in[j,N+2]_\Z ,   \\
& \ \delta_{i,j,v}\geq 0,\quad \delta_{i,j,v}\geq \pi_{j,v} -\rho_i,\quad\forall i\in[N],\,\, \forall j\in[2,N+1]_\Z,\,\forall v\in[j,N+2]_\Z .
\end{align}
\end{subequations}

Note that the McCormick inequalities often rely on big-M coefficients (i.e., variable lower and upper bounds) which take large values and can undermine the computational efficiency. Next, we derive tight bounds of these  big-M coefficients to strengthen formulation.
\begin{prop} \label{Prop:McCormick}
Let $P^u_1 = \max_{j\in[2,N+1]_\Z,\, v\in[j,N+2]_\Z} \pi_{j,v}$ and $P^l_1 = \min_{j\in[2,N+1]_\Z,\, v\in[j,N+2]_\Z} \pi_{j,v}$. Also, let $P^u_2 = \max_{v\in[N+2]} \pi_{1,v}$ and $P^l_2 = \min_{v\in[N+2]} \pi_{1,v}$. Then, the lower bounds are $\rholb_i=P_1^l$, $\alphalb_{i,i'}=P_1^l-\lambda$ and $\alphalb_{0,i}=P_2^l-\lambda$ while the upper bounds are $\rhoub_i = P_1^u$, $\alphaub_{i,i'}=P_1^u+\lambda$ and $\alphaub_{0,i}=P_2^u+\lambda$.
\end{prop}

\begin{proof}
We only present the case for $\rho_i$ and the remaining two cases are similar. From the model in Proposition \ref{Prop1M:DualMinMax} and its reformulation \eqref{IPFull} with \eqref{IP:5term} to \eqref{IP:6term}, the terms involving $\rho_i$ can be summarized as 
$$\mud_i \rho_i + \sum_{j=2}^{N+1} \sum_{k=1}^j \sum_{v=j}^{N+2} \left[(\pi_{j,v}-\rho_i) \dlb_i + \Delta d_i (\pi_{j,v}-\rho_i)^+ \right] x_{i,j-1} b_{k,v}.$$
If $\rho_i\geq P_1^u$,  we can simplify the terms as
$$\mud_i\rho_i+ \sum_{j=2}^{N+1} \sum_{k=1}^j \sum_{v=j}^{N+2} (\pi_{j,v}-\rho_i)\dlb_i x_{i,j-1} b_{k,v} = (\mud_i-\dlb_i)\rho_i + \sum_{j=2}^{N+1} \sum_{k=1}^j \sum_{v=j}^{N+2}  x_{i,j-1}\pi_{j,v}b_{k,v},$$
which is non-decreasing in $\rho_i$ since $\mud_i - \dlb_i \geq 0$. Similarly, if $\rho_i \leq P^l_1$, then the terms read 
$$\mud_i\rho_i+ \sum_{j=2}^{N+1} \sum_{k=1}^j \sum_{v=j}^{N+2} (\pi_{j,v}-\rho_i)\dub_i x_{i,j-1} b_{k,v} = (\mud_i-\dub_i)\rho_i + \sum_{j=2}^{N+1} \sum_{k=1}^j \sum_{v=j}^{N+2}  x_{i,j-1}\pi_{j,v}b_{k,v},$$
which is non-increasing in $\rho_i$ since $\mud_i - \dub_i \leq 0$. This shows that there exists a minimizer over $\rho_i$ lying in $[P^l_1, P^u_1]$ and we can set $\rholb_i=P_1^l$ and  $\rhoub_i = P_1^u$.
\end{proof}

\begin{cor}
We have lower bounds $\deltalb_{i,j,v}=0$, $\gammalb_{i,i',j,v}=0$ and $\gammalb_{0,i,1,v}=0$ with upper bounds $\deltaub_{i,j,v}=\pi_{j,v}-P^l_1$, $\gammaub_{i,i',j,v}=\pi_{j,v}+2\lambda-P^l_1$ and $\gammaub_{0,i,1,v}=\pi_{1,v}+2\lambda-P^l_2$.
\end{cor}

\begin{proof}
We only present the case for $\delta_{i,j,v}$ and the remaining two cases are similar. Recall $\delta_{i,j,v}=(\pi_{j,v}-\rho_i)^+$ and from Proposition \ref{Prop:McCormick}, we have $$\pi_{j,v}-P^u_1 \leq \pi_{j,v}-\rho_i \leq \pi_{j,v} - P^l_1\,,$$ where the lower bound is no greater than $0$ and the upper bound is no less than $0$.
\end{proof}

\section{Proof of Lemma 1 Adapted from \cite{jiang2019data} }\label{Proof_Lemma1}
\noindent \textbf{Lemma 1.} (of  \cite{jiang2019data}). \textit{Suppose that $\calS$ is non-empty, convex and compact. Then there exist nonegative constants $c_1$ and $c_2$ such that, for all $R\geq 1$ and $\beta \in (0, \min \{ 1, c_1\} )$,}
\begin{equation*}
\Prob_{\xib}^R \Big \{ W_p (\Prob_{\xib}, \hat{\Prob}_{\xib}^R) \leq \epsilon_R(\beta) \Big\} \geq 1-\beta
\end{equation*}
\textit{where $\Prob_{\xib}^R$ represents the product measure of $R$ copies of $\Prob_{\xib}$ and $\epsilon_R(\beta)= \Big [ \frac{\log(c_1\beta^{-1})}{c_2R}\Big]^{\frac{1}{\max \{ 3p, n\}}}$.}

\vspace{2mm}
For completeness, we provide the proof of Lemma 1 of \cite{jiang2019data} as detailed in their paper, which is adapted from Theorem 2 in \cite{fournier2015rate}. 

\begin{proof} 
Notice that by the compactness assumption of the support, there exist $\alpha>p$ and $\gamma>0$ such that $\E_{\Prob_{\xib}}[\exp\{\gamma||\xi||_p^\alpha\}]<\infty$. By Theorem 2 in \cite{fournier2015rate}, for any $R\geq 1$ and $\epsilon\in(0,\infty)$, there exist positive constants $c$ and $C$ depending only on $p$, $n$, $\alpha$ and $\gamma$ such that 
$$\Prob_{\xib}^R\left(W_p(\Prob_{\xib},\hat{\Prob}_{\xib}^R)\geq \epsilon^{1/p}\right)\leq a(R,\epsilon)\mathbbm{1}(\epsilon\leq1)+b(R,\epsilon),$$ 
where 
\begin{equation}
a(R, \epsilon)=C\left\{\begin{array}{ll}
\exp \left\{-c R \epsilon^{2}\right\} & \text { if } p>n / 2 \\
\exp \left\{-c R(\epsilon / \log (2+1 / \epsilon))^{2}\right\} & \text { if } p=n / 2 \\
\exp \left\{-c R \epsilon^{n / p}\right\} & \text { if } p \in[1, n / 2)
\end{array}\right.
\end{equation}
and $b(R,\epsilon)=C\exp\{-CR\epsilon^{\alpha/p}\}\mathbbm{1}(\epsilon>1)$ with $n$ being the dimension of the random vector $\xib$. We first bound the term $a(R,\epsilon)$ for $\epsilon\in(0,1]$. Notice that $$\left[\frac{\epsilon}{\log(2+1/\epsilon)}\right]^2 \geq \frac{\epsilon^3}{[\log(3)]^2}\,,$$ which immediately gives $\epsilon[\log(2+1/\epsilon)]^2 \leq [\log(3)]^2$. Then, we have $$a(R,\epsilon)\leq C\exp\left\{ -\frac{c}{[\log(3)]^2}R\epsilon^{\max\{3,n/p\}} \right\}.$$ Next, to bound the term $b(R,\epsilon)$, let $\alpha=\max\{3p,n\} > p$. Then, we have $$b(R,\epsilon)\leq C\exp\left\{ -cR\epsilon^{\max\{3,n/p\}}\right\}.$$ To sum up, we arrive at $$\Prob_{\xib}^R\left(W_p(\Prob_{\xib},\hat{\Prob}_{\xib}^R)\geq \epsilon^{1/p}\right)\leq  c_1\exp\left\{ -c_2R\epsilon^{\max\{3,n/p\}}\right\},$$ where $c_1=C$ and $c_2=c/[\log(3)]^2$. By equating the right hand side of the last inequality to $\beta$, we obtain $\epsilon=[(c_2 R)^{-1}\log(c_1\beta^{-1})]^{-\max\{3,n/p\}^{-1}}$. Plugging in this expression to the last inequality gives the desired result.
\end{proof}

\section{Proof of Theorem 1 from \cite{jiang2019data}}\label{Proof_Thrm1}

\noindent \textbf{Theorem 1.} (Asymptotic consistency, adapted from \citealp{jiang2019data} and Theorem 3.6 of \citealp{esfahani2018data}). \textit{Suppose that the support $\calS$ is non-empty, convex and compact. Consider a sequence of confidence levels $\{ \beta_R\}_{R \in \mathbb{R}}$ such that $\sum_{R=1}^\infty \beta_R <\infty $ and $\lim_{R\rightarrow \infty} \epsilon_R(\beta_R)=0$, and let $(\hat{\xb} (R, \epsilon_R(\beta_R)),\hat{\ab}(R, \epsilon_R(\beta_R)))$ represent an optimal solution to W-DHRAS with the ambiguity set $\calF_p (\hat{\Prob}_{\xib}^R, \epsilon_R(\beta_R))$. Then, $\Prob_{\xib}^\infty$- almost surely we have $\hat{Z}(R, \epsilon_R(\beta_R))\rightarrow Z^*$ as $R\rightarrow \infty$. In addition, any accumulation points of  $\{(\hat{\xb} (R, \epsilon_R(\beta_R)),\hat{\ab}(R, \epsilon_R(\beta_R))) \}_{R \in \mathbb{N}}$ is an optimal solution of \eqref{SP} $\Prob_{\xib}^\infty$- almost surely. }

\begin{proof} 
Recall the dual of $f(\xb,\ab,\xib)$. 
\begin{subequations}
\begin{align}
  f(\xb, \ab,\xib )=  & \max_{\yb} \Bigg(\sum_{i=1}^N t_{0,i}x_{i,1}-a_1\Bigg)y_1+  \sum_{j=2}^{N+1} \Bigg ( a_{j-1}-a_j+\sum_{i=1}^N d_ix_{i,j-1}+ \sum_{i=1}^N \sum_{i'\neq i} t_{i,i'} x_{i,j-1}x_{i,j}\Bigg)y_j  \nonumber\\
  & \text{s.t.}  \ \  \yb\in\calY= \{\yb:\, y_{N+1} \leq c^{\mbox{\tiny o}}, \ -c_j^{\mbox{\tiny u}}\leq y_j \leq c_j^{\mbox{\tiny w}}+ y_{j+1}, \ \forall j \in [N]\}. \nonumber
\end{align} 
\end{subequations}
For a fixed pair of $(\xb,\ab)\in\calX\times\calA$, this is an LP in $\yb$. Note that $\calY$ is bounded, which implies that $f(\ab,\xb,\xib)$ is finite. Also, since $\calX\times\calA$ and $\calS$ are bounded sets, the function $f(\xb, \ab,\xib)$ is bounded on $\calX\times\calA\times\calS$. Immediately, we have $|f(\xb,\ab,\xib)|\leq M(1+\norms{\xib})$, where $M$ is just the upper bound on $f$. Next, we claim that $f(\xb, \ab,\xib)$ is continuous on $\calX\times\calA\times\calS$. For simplicity, write the objective function of the dual problem as $b(\xb,\ab,\xib)^\top y$, where the $i$-th entry of $b(\xb,\ab,\xib)$ is the coefficient associated to $y_i$. By fundamental theorem of LP, instead of maximizing over the entire $\calY$, we can maximize over the set of finite extreme points of $\calY$. Hence, $f(\xb, \ab,\xib)$ is the maximum of finitely many linear functions and the continuity follows.
Finally, the result follows directly from Theorem 3.6 of \cite{esfahani2018data}, where the required conditions are verified.
\end{proof}

\section{Proof of Theorem 2 in \cite{jiang2019data} }\label{Proof_Thrm2}

\noindent \textbf{Theorem 2.} (Finite-data guarantee, adapted from \citealp{jiang2019data} and Theorem 3.5 in \citealp{esfahani2018data}).  \textit{For any $\beta \in (0, 1)$, let $(\hat{\xb}(R, \epsilon_R(\beta_R)),  \hat{\ab}(R, \epsilon_R(\beta_R)))$ represent an optimal solution of W-DHRAS with ambiguity set $\calF_p (\hat{\Prob}_{\xib}^R, \epsilon_R(\beta_R))$. Then,}
$$\Prob_{\xib}^R \Big \{  \E_{\Prob_{\xib}} [f (  \hat{\xb}(R, \epsilon_R(\beta_R)),  \hat{\ab}(R, \epsilon_R(\beta_R)), \xib) ] ) \leq \hat{Z} (R, \epsilon_R(\beta_R))\Big\} \geq 1- \beta. $$ 

\begin{proof}
(adapted from \citealp{jiang2019data}).   By the assumption on support $\calS$ and Lemma 1, all conditions of Theorem 3.5 in \cite{esfahani2018data} are satisfied. Therefore, the conclusions of Theorem 2 hold valid.
\end{proof}

\section{Proof of Proposition~\ref{Prop1}}\label{Proof_Prop1}

\begin{proof} 
Recall that $\hat{\Prob}^R_{\xib}=\frac{1}{R}\sum_{r=1}^R \delta_{\hat{\xib}^r}$. For any $\Prob_{\xib}\in\mathcal{P}(\calS)$, we can rewrite the joint distribution $\Pi\in\mathcal{P}(\Prob_{\xib},\hat{\Prob}^R_{\xib})$ by the conditional distribution of $\xib$ given $\hat{\xib}=\hat{\xib}^r$ for $r=1,\dots,R$, denoted as $\Q_{\xib}^r$. That is, $\Pi=\frac{1}{R}\sum_{r=1}^R (\delta_{\hat{\xib}^r} \times \Q_{\xib}^r)$. Notice that if we find one joint distribution $\Pi\in\mathcal{P}(\Prob_{\xib},\hat{\Prob}^R_{\xib})$ such that $\int ||\xib-\hat{\xib}||_1 d\Pi \leq \epsilon$, then $W_1(\Prob_{\xib},\hat{\Prob}_{\xib}^R)\leq\epsilon$. Hence, we can drop the infimum operator in Wasserstein distance and arrive at the following equivalent problem
\begin{subequations}
\begin{align}
& \sup_{\Q_{\xib}^r\in\mathcal{P}(\calS), r\in[R]} \frac{1}{R}\sum_{r=1}^R\int_\calS f(\xb,\ab,\xib) d\Q_{\xib}^r\\
& \ \ \ \ \ \ \ \text{s.t.}  \ \ \ \ \ \  \frac{1}{R} \sum_{r=1}^R  \int_\calS \norms{\xib-\hat{\xib}^r}_1 d\Q_{\xib}^r \leq \epsilon .
\end{align}
\end{subequations}
Using a standard strong duality argument, we can reformulate the problem by its dual, i.e.,
\begin{align}
&\quad\,\inf_{\rho\geq 0} \sup_{\Q_{\xib}^r\in\mathcal{P}(\calS),r\in[R]}\left\{\frac{1}{R}\sum_{r=1}^R\int_\calS f(\xb,\ab,\xib) d\Q_{\xib}^r+\rho\left[\epsilon-\frac{1}{R} \sum_{r=1}^R  \int_\calS \norms{\xib-\hat{\xib}^r}_1 d\Q_{\xib}^r \right] \right\}\nonumber\\
&=\inf_{\rho\geq 0}  \Bigg\{\epsilon\rho + \frac{1}{R}\sum_{r=1}^R \sup_{\Q_{\xib}^r\in\mathcal{P}(\calS)}\int_\calS \left[f(\xb,\ab,\xib)-\rho\norms{\xib-\hat{\xib}^r}_1\right] d\Q_{\xib}^r\Bigg\} \nonumber\\
&=\inf_{\rho\geq 0} \Bigg\{ \epsilon \rho + \frac{1}{R} \sum_{r=1}^R \sup_{\xib \in \calS}  \Big\{ f(\xb, \ab, \xib)-\rho \norms{ \xib -\hat{\xib}^r }_1  \Big\} \Bigg\}.
\end{align}
This completes the proof.
\end{proof}

\section{Proof of Proposition~\ref{Prop4M}}\label{Proof_Prop4M}
\begin{proof}
With reference to \eqref{IPFormulation}, we can reformulate $\max_{\xib\in\calS}\{f(\xb,\ab,\xib)-\rho\norms{\xib-\xibhat^r}_1\}$ as
\begin{subequations}\label{Prop4_pf}
\begin{align}
\max_{\bb,\,\xib} &   \sum_{v=1}^{N+2} -a_1 \pi_{1,v} b_{1,v}+\sum_{j=2}^{N+1} \sum_{k=1}^j \sum_{v=j}^{N+2} ( a_{j-1}-a_j) \pi_{j,v} b_{k,v} \nonumber +\sum_{j=2}^{N+1} \sum_{k=1}^j \sum_{v=j}^{N+2}  \sum_{i=1}^{N}  d_ix_{i, j-1}\pi_{j,v} b_{k,v}\nonumber \\
& \ \ +\sum_{v=1}^{N+2}\sum_{i=1}^N  t_{0,i}\pi_{1,v} x_{i,1} b_{1,v} +\sum_{j=2}^{N+1} \sum_{k=1}^j \sum_{v=j}^{N+2} \sum_{i=1}^N \sum_{i'\neq i} t_{i,i'}\pi_{j,v}x_{i,j-1} x_{i', j} b_{k,v} \nonumber \\
& \ \ + \lambda \sum_{i=1}^N t_{0,i}x_{i,1} + \lambda \sum_{i=1}^N t_{i,0} x_{i,N} + \lambda \sum_{j=2}^N \sum_{i=1}^N \sum_{i'\ne i} t_{i,i'}x_{i,j-1}x_{i',j}-\rho  \sum_{j=2}^{N+1}\sum_{i=1}^N |d_i-\hat{d}_i^r|x_{i,j-1}  \nonumber \\
& \ \ -\rho  \sum_{j=2}^{N+1}\sum_{i=1}^N\sum_{i'\neq i}  |t_{i,i'}-\hat{t}_{i,i'}^r|x_{i,j-1}x_{i',j} - \rho  \sum_{i=1}^N |t_{0,i}-\that_{0,i}^r|x_{i,1}  - \rho  \sum_{i=1}^N |t_{i,0}-\that_{i,0}^r|x_{i,N}
\label{Prop4_pf:Obj} \\
\text{s.t.} &  \ \sum_{k=1}^j \sum_{v=j}^{N+2} b_{k,v}=1, \qquad \forall j \in [N+2], \label{Prop4_pf:Con}\\ 
& \  b_{k,v} \in \{0, 1\}, \quad \forall k \in [N], \forall v \in [k, N+2]_\Z, \label{Prop4_pf:Const} \\
& \ \xib\in\calS.
\end{align}
\end{subequations}
As in the proof of Proposition \ref{Prop2M}, we group terms with the same $t_{i,i'}$ and $d_i$ and find an explicit formula for their supremum. 

First, we define the following quantities for $r\in[R]$.
\begin{subequations} \label{WDHRAS_Var}
\begin{align}
u_{i,0}^r &=\max\{\lambda\that^r_{i,0},\lambda\tub_{i,0}-\rho(\tub_{i,0}-\that^r_{i,0})\} \\
u_{0,i,1,v}^r &=\max\{(\pi_{1,v}+\lambda)\tlb_{0,i}-\rho(\that_{0,i}^r-\tlb_{0,i}),(\pi_{1,v}+\lambda)\that_{0,i}^r,(\pi_{1,v}+\lambda)\tub_{0,i}-\rho(\tub_{0,i}-\that^r_{0,i})\}\\
u_{i,i',j,v}^r &=\max\{(\pi_{j,v}+\lambda)\tlb_{i,i'}-\rho(\that_{i,i'}^r-\tlb_{i,i'}),(\pi_{j,v}+\lambda)\that_{i,i'}^r,(\pi_{j,v}+\lambda)\tub_{i,i'}-\rho(\tub_{i,i'}-\that^r_{i,i'})\}\\
\nu_{i,j,v}^r &=\max\{\pi_{j,v}\dlb_{i}-\rho(\dhat_{i}^r-\dlb_{i}),\pi_{j,v}\dhat_{i}^r,\pi_{j,v}\dub_{i}-\rho(\dub_{i}-\dhat^r_{i})\}
\end{align}
\end{subequations}
Note that for the terms with $t_{i,0}$, 
\begin{equation}\label{Prop4_pf:1}
\sup_{t_{i,0}\in[\tlb_{i,0},\tub_{i,0}],\, i\in[N]} \sum_{i=1}^N \left[\lambda t_{i,0}-\rho |t_{i,0}-\that^r_{i,0}| \right] x_{i,N} =\sum_{i=1}^N \sup_{t_{i,0}\in[\tlb_{i,0},\tub_{i,0}]}\left[\lambda t_{i,0}-\rho |t_{i,0}-\that^r_{i,0}| \right] x_{i,N}.
\end{equation}
Since $t_{i,0}-\rho |t_{i,0}-\that^r_{i,0}|$ is a piecewise linear function on the intervals $[\tlb_{i,0},\that^r_{i,0}]$ and $[\that^r_{i,0},\tub_{i,0}]$, and it is strictly increasing on $[\tlb_{i,0},\that^r_{i,0}]$, the maximum is attained either at $\that^r_{i,0}$ or $\tub_{i,0}$. That is, 
\begin{equation}\label{Prop4_pf:2}
\sup_{t_{i,0}\in[\tlb_{i,0},\tub_{i,0}]}\left[\lambda t_{i,0}-\rho |t_{i,0}-\that^r_{i,0}| \right] =u_{i,0}^r.  
\end{equation}
For the terms with $t_{0,i}$, we have 
\begin{align}
&\quad\ \sup_{t_{0,i}\in[\tlb_{0,i},\tub_{0,i}],\, i\in[N]} \sum_{i=1}^N \sum_{v=1}^{N+2}\left[(\pi_{1,v}+\lambda) t_{0,i}-\rho |t_{0,i}-\that^r_{0,i}| \right] b_{i,v} x_{i,1} \nonumber\\
&= \sum_{i=1}^N \sum_{v=1}^{N+2} \sup_{t_{0,i}\in[\tlb_{0,i},\tub_{0,i}]}\left[(\pi_{1,v}+\lambda) t_{0,i}-\rho |t_{0,i}-\that^r_{0,i}| \right] b_{i,v} x_{i,1} \nonumber\\
&= \sum_{i=1}^N \sum_{v=1}^{N+2} u_{0,i,1,v}^r b_{i,v} x_{i,1}, \label{Prop4_pf:3}
\end{align}
where the last equality follows from the fact that $(\pi_{1,v}+\lambda) t_{0,i}-\rho |t_{0,i}-\that^r_{0,i}|$ is a piecewise linear function on the interval $[\tlb_{0,i},\that^r_{0,i}]$ and $[\that^r_{0,i},\tub_{0,i}]$ and the maximum is attained when $t_{0,i}$ takes any one of the values from $\tlb_{0,i}$, $\that^r_{0,i}$ and $\tub_{0,i}$. Using the same argument, we obtain the expressions for the terms with $t_{i,i'}$ and $d_i$ as follows.
\begin{align}
&\quad\ \sup_{t_{i,i'}\in[\tlb_{i,i'},\tub_{i,i'}],i\in[N],i'\ne i}\sum_{i=1}^N \sum_{i'\ne i} \sum_{j=2}^{N+1}\sum_{k=1}^j \sum_{v=j}^{N+2} \left[(\pi_{j,v}+\lambda) t_{i,i'}-\rho|t_{i,i'}-\that^r_{i,i'}|\right] x_{i,j-1} x_{i',j} b_{k,v} \nonumber\\ 
&=\sum_{i=1}^N \sum_{i'\ne i} \sum_{j=2}^{N+1}\sum_{k=1}^j \sum_{v=j}^{N+2} u_{i,i',j,v}^r x_{i,j-1} x_{i',j} b_{k,v}  \label{Prop4_pf:4}\\
&\quad\ \sup_{d_{i}\in[\dlb_{i},\dub_{i}],i\in[N]}\sum_{i=1}^N  \sum_{j=2}^{N+1}\sum_{k=1}^j \sum_{v=j}^{N+2} \left[\pi_{j,v} d_i-\rho|d_{i}-\dhat^r_{i}|\right] x_{i,j-1} b_{k,v} \nonumber\\
&= \sum_{i=1}^N  \sum_{j=2}^{N+1}\sum_{k=1}^j \sum_{v=j}^{N+2} \nu_{i,j,v}^r  x_{i,j-1} b_{k,v}\label{Prop4_pf:5}
\end{align}
With the use of \eqref{Prop4_pf:1} to \eqref{Prop4_pf:5}, we arrive at the following integer program.
\begin{subequations}\label{Prop4_pf_LP}
\begin{align}
\max_{\bb} &   \sum_{v=1}^{N+2} -a_1 \pi_{1,v} b_{1,v}+\sum_{j=2}^{N+1} \sum_{k=1}^j \sum_{v=j}^{N+2} ( a_{j-1}-a_j) \pi_{j,v} b_{k,v} \nonumber +\sum_{j=2}^{N+1} \sum_{k=1}^j \sum_{v=j}^{N+2}  \sum_{i=1}^{N}  \nu_{i,j,v}^r x_{i, j-1} b_{k,v}\nonumber \\
& \ \ +\sum_{v=1}^{N+2}\sum_{i=1}^N  u_{0,i,1,v}^r x_{i,1} b_{i,v} +\sum_{j=2}^{N+1} \sum_{k=1}^j \sum_{v=j}^{N+2} \sum_{i=1}^N \sum_{i'\neq i} u_{i,i',j,v}^r x_{i,j-1} x_{i', j} b_{k,v} + \sum_{i=1}^N u_{i,0}^r x_{i,N} \label{Prop4_pf_LP:Obj} \\
\text{s.t.} &  \ \sum_{k=1}^j \sum_{v=j}^{N+2} b_{k,v}=1, \qquad \forall j \in [N+2], \label{Prop4_pf_LP:Con}\\ 
& \  b_{k,v} \in \{0, 1\}, \quad \forall k \in [N], \forall v \in [k, N+2]_\Z. \label{Prop4_pf_LP:Const}
\end{align}
\end{subequations}

Since the coefficient matrix associated to the constraints \eqref{Prop4_pf_LP:Con} is totally unimodular, we can relax the integral constraints on $b_{k,b}$ by $b_{k,v}\geq 0$. Taking the dual of the resulting LP in \eqref{Prop4_pf_LP} with $\beta^r_j$ being the dual variable associated to the constraint \eqref{Prop4_pf_LP:Con}, we arrive at the following LP.
\allowdisplaybreaks
\begin{subequations}
\begin{align}
 \min_{\betab}  &\  \sum \limits_{j=1}^{N+2} \beta^r_j + \sum_{i=1}^N u_{i,0}^r x_{i,N}  \\
\text{s.t.} & \ \beta^r_1 \geq -a_1  \pi_{1,1} +\sum_{i=1}^N u_{0,i,1,1}^r x_{i,1},  \\
& \ \sum_{j=1}^v \beta^r_j \geq -a_1  \pi_{1,v} + \sum_{j=2}^{\min(v,N+1)} \big ( a_{j-1}-a_j \big) \pi_{j,v} + \sum_{i=1}^N u_{0,i,1,v}^r x_{i,1} \nonumber \\
& \ \qquad \   + \sum_{j=2}^{\min(v,N+1)} \sum_{i=1}^N \sum_{i' \neq i} u_{i,i',j,v}^r x_{i,j-1}x_{i', j}  + \sum_{j=2}^{\min(v,N+1)} \sum_{i=1}^N \nu_{i,j,v}^r x_{i,j-1}, \ \
 \forall v \in [2, N+2]_{\Z},   \\
& \ \sum_{j=k}^v\beta^r_j \geq \sum_{j=k}^{\min(v,N+1)} \big ( a_{j-1}-a_j\big)  \pi_{j,v}+\sum_{j=k}^{\min(v,N+1)} \sum_{i=1}^N \nu_{i,j,v}^r x_{i,j-1}  \nonumber \\
& \ \qquad \    + \sum_{j=k}^{\min(v,N+1)} \sum_{i=1}^N \sum_{i' \neq i} u_{i,i',j,v}x_{i,j-1}x_{i',j},\quad  \forall k\in[2,N+1]_\Z,\,\forall v\in[k,N+2]_\Z,  \\
& \ \beta^r_{N+2} \geq 0  .
\end{align}
\end{subequations}

Finally, we can treat $\ub^r$ and $\nub^r$ as variables and introduce the following constraints.
\allowdisplaybreaks
\begin{subequations}
\begin{align}
& \ u^r_{i,0} \geq \lambda\that^r_{i,0},\quad  u^r_{i,0} \geq \lambda\tub_{i,0}-\rho(\tub_{i,0}-\that^r_{i,0}),\quad\forall i\in[N] \\
& \ u_{0,i,1,v}^r  \geq (\pi_{1,v}+\lambda)\tlb_{0,i}-\rho(\that_{0,i}^r-\tlb_{0,i}),\quad  u_{0,i,1,v}^r  \geq (\pi_{1,v}+\lambda)\that_{0,i}^r,  \nonumber \\
& \ u_{0,i,1,v}^r  \geq (\pi_{1,v}+\lambda)\tub_{0,i}-\rho(\tub_{0,i}-\that^r_{0,i}),\quad \forall i\in[N],\, v\in[N+2] \\
& \ u_{i,i',j,v}^r   \geq (\pi_{j,v}+\lambda)\tlb_{i,i'}-\rho(\that_{i,i'}^r-\tlb_{i,i'}),\quad u_{i,i',j,v}^r   \geq (\pi_{j,v}+\lambda)\tub_{i,i'}-\rho(\tub_{i,i'}-\that^r_{i,i'}), \nonumber  \\
& \ u_{i,i',j,v}^r   \geq (\pi_{j,v}+\lambda)\that_{i,i'}^r,\quad \forall i\in[N],\, i'\in[N]\setminus\{i\},\, j\in[2,N+1]_\Z,\, v\in[j,N+2]_\Z  \\
& \ \nu_{i,j,v}^r    \geq \pi_{j,v}\dlb_{i}-\rho(\dhat_{i}^r-\dlb_{i}),\quad 
 \nu_{i,j,v}^r    \geq \pi_{j,v}\dub_{i}-\rho(\dub_{i}-\dhat^r_{i}), \nonumber  \\
& \ \nu_{i,j,v}^r    \geq \pi_{j,v}\dhat_{i}^r,\quad \forall i\in[N],\, j\in[2,N+1]_\Z,\, v\in[j,N+2]_\Z 
\end{align}
\end{subequations}
This completes the proof.
\end{proof}

\section{Details of the final MINLP and MILP reformulation of W-DHRAS} \label{Appx:Mac_WDHRAS}
Using Proposition \ref{Prop4M}, we can reformulate W-DHRAS as the following MINLP.
\allowdisplaybreaks
\begin{subequations}
\begin{align}
 \min &\ \ \epsilon \rho + \frac{1}{R} \sum_{r=1}^R \left(\sum_{i=1}^N u^r_{i,0} x_{i,N} + \sum_{j=1}^{N+2}\beta_j^r \right)  \\
\text{s.t.} &\ \  \xb\in\calX,\,\ab\in\calA,\, \rho\geq 0,\, \betab^r\in\R^{N+2},\quad\forall r\in[R], \\
& \ \beta^r_1 \geq -a_1  \pi_{1,1} +\sum_{i=1}^N u_{0,i,1,1}^r x_{i,1} , \\
& \ \sum_{j=1}^v \beta^r_j \geq -a_1  \pi_{1,v} + \sum_{j=2}^{\min(v,N+1)} \big ( a_{j-1}-a_j \big) \pi_{j,v} + \sum_{i=1}^N u_{0,i,1,v}^r x_{i,1} \nonumber \\
& \ \qquad \   + \sum_{j=2}^{\min(v,N+1)} \sum_{i=1}^N \sum_{i' \neq i} u_{i,i',j,v}^r x_{i,j-1}x_{i', j}  + \sum_{j=2}^{\min(v,N+1)} \sum_{i=1}^N \nu_{i,j,v}^r x_{i,j-1}, \ \
 \forall v \in [2, N+2]_{\Z},  \\
& \ \sum_{j=k}^v\beta^r_j \geq \sum_{j=k}^{\min(v,N+1)} \big ( a_{j-1}-a_j\big)  \pi_{j,v}+\sum_{j=k}^{\min(v,N+1)} \sum_{i=1}^N \nu_{i,j,v}^r x_{i,j-1}  \nonumber \\
& \ \qquad \   + \sum_{j=k}^{\min(v,N+1)} \sum_{i=1}^N \sum_{i' \neq i} u_{i,i',j,v}x_{i,j-1}x_{i',j},\quad  \forall k\in[2,N+1]_\Z,\,\forall v\in[k,N+2]_\Z ,\\
& \ \beta^r_{N+2} \geq 0 ,\\
& \ u^r_{i,0} \geq \lambda\that^r_{i,0},\quad  u^r_{i,0} \geq \lambda\tub_{i,0}-\rho(\tub_{i,0}-\that^r_{i,0}),\quad\forall i\in[N], \\
& \ u_{0,i,1,v}^r  \geq (\pi_{1,v}+\lambda)\tlb_{0,i}-\rho(\that_{0,i}^r-\tlb_{0,i}),\quad  u_{0,i,1,v}^r  \geq (\pi_{1,v}+\lambda)\that_{0,i}^r,\nonumber\\
& \ u_{0,i,1,v}^r  \geq (\pi_{1,v}+\lambda)\tub_{0,i}-\rho(\tub_{0,i}-\that^r_{0,i}),\quad \forall i\in[N],\, v\in[N+2], \\
& \ u_{i,i',j,v}^r   \geq (\pi_{j,v}+\lambda)\tlb_{i,i'}-\rho(\that_{i,i'}^r-\tlb_{i,i'}),\quad u_{i,i',j,v}^r   \geq (\pi_{j,v}+\lambda)\tub_{i,i'}-\rho(\tub_{i,i'}-\that^r_{i,i'}), \nonumber\\
& \ u_{i,i',j,v}^r   \geq (\pi_{j,v}+\lambda)\that_{i,i'}^r,\quad \forall i\in[N],\, i'\in[N]\setminus\{i\},\, j\in[2,N+1]_\Z,\, v\in[j,N+2]_\Z ,\\
& \ \nu_{i,j,v}^r    \geq \pi_{j,v}\dlb_{i}-\rho(\dhat_{i}^r-\dlb_{i}),\quad 
 \nu_{i,j,v}^r    \geq \pi_{j,v}\dub_{i}-\rho(\dub_{i}-\dhat^r_{i}), \nonumber \\
& \ \nu_{i,j,v}^r    \geq \pi_{j,v}\dhat_{i}^r,\quad \forall i\in[N],\, j\in[2,N+1]_\Z,\, v\in[j,N+2]_\Z. 
\end{align}
\end{subequations}
We can apply McCormick inequalities to linearize the non-linear (such as $u_{i,0}^r x_{i,N}$ in the objective and $u_{i,i',j,v}^r x_{i,j-1} x_{i',j}$ in the constraints). Here, we reuse the Greek letters in M-DHRAS model for the McCormick inequalities. Let $\tau_{i,i',j-1,j}=x_{i,j-1}x_{i',j}$, $\psi^r_{i,0}=u^r_{i,0}x_{i,N}$, $\sigma^r_{0,i,1,v}=u^r_{0,i,1,v} x_{i,1}$, $\phi^r_{i,i',j,v}=u^r_{i,i',j,v}\tau_{i,i',j-1,j}$ and $\zeta^r_{i,j,v}=\nu^r_{i,j,v}x_{i,j-1}$. In addition to \eqref{MILP-MAC1M} and \eqref{MILP-MAC2M}, we introduce the following McCormick inequalities.
\begin{subequations}
\begin{align}\
&  \tau_{i,i',j-1,j} \geq x_{i,j-1}+x_{i',j}-1, \ \tau_{i,i',j-1,j} \geq 0, \label{WMILP-MAC1M}\\
&  \tau_{i,i',j-1,j} \leq x_{i,j-1}, \ \tau_{i,i',j-1,j} \leq x_{i',j} \label{WMILP-MAC2M}\\
& \psi^r_{i,0} \geq \ulb^r_{i,0}x_{i,N}, \ \psi^r_{i,0} \geq u^r_{i,0}+\uub^r_{i,0} (x_{i,N}-1), \ \psi^r_{i,0} \leq \uub^r_{i,0} x_{i,N}, \ \psi^r_{i,0} \leq u^r_{i,0} + \ulb^r_{0,i} (x_{i,N}-1) \label{WMILP-MAC3M}\\
&  \phi^r_{i,i',j,v}\geq \ulb^r_{i,i',j,v} \tau_{i,i',j-1,j}, \ \phi^r_{i,i',j,v} \geq u^r_{i,i',j,v}+\uub^r_{i,i',j,v}(\tau_{i,i',j-1,j}-1)\label{WMILP-MAC4M}\\
&  \phi^r_{i,i',j,v} \leq \uub^r_{i,i',j,v}\tau_{i,i',j-1,j}, \ \phi^r_{i,i',j,v}\leq u^r_{i,i',j,v} + \ulb^r_{i,i',j,v}(\tau_{i,i',j-1,j}-1) \label{WMILP-MAC5M} \\
&  \sigma^r_{0,i,1,v} \geq \ulb^r_{0,i,1,v}x_{i,1} ,\ \sigma^r_{0,i,1,v}  \geq u^r_{0,i,1,v}+\uub^r_{0,i,1,v}(x_{i,1}-1) \label{WMILP-MAC6M} \\
&  \sigma^r_{0,i,1,v} \leq \uub^r_{0,i,1,v}x_{i,1}, \ \sigma^r_{0,i,1,v}\leq u^r_{0,i,1,v} + \ulb^r_{0,i,1,v}(x_{i,1}-1) \label{WMILP-MAC7M} \\
&  \zeta^r_{i,j,v} \geq \nulb^r_{i,j,v}x_{i,j-1} ,\ \zeta^r_{i,j,v}  \geq \nu^r_{i,j,v} + \nuub^r_{i,j,v}(x_{i,j-1}-1) \label{WMILP-MAC8M} \\
&  \zeta^r_{i,j,v} \leq \nuub^r_{i,j,v}x_{i,j-1}, \ \zeta^r_{i,j,v}  \leq \nu^r_{i,j,v} + \nulb^r_{i,j,v}(x_{i,j-1}-1) \label{WMILP-MAC9M} 
\end{align}
\end{subequations}
Therefore, we obtain an MILP reformulation of the W-DHRAS model.
\allowdisplaybreaks
\begin{subequations}
\begin{align}
 \min &\ \ \epsilon \rho + \frac{1}{R} \sum_{r=1}^R \left(\sum_{i=1}^N \psi^r_{i,0} + \sum_{j=1}^{N+2}\beta_j^r \right)  \\
\text{s.t.} &\ \  \xb\in\calX,\,\ab\in\calA,\, \rho\geq 0,\, \betab^r\in\R^{N+2},\quad\forall r\in[R], \\
& \ \text{constraints } \eqref{WMILP-MAC1M}-\eqref{WMILP-MAC9M},\quad\forall r\in[R],\\
& \ \beta^r_1 \geq -a_1  \pi_{1,1} +\sum_{i=1}^N \sigma_{0,i,1,1}^r, \label{W-DHRAS_MILP_Con2} \\
& \ \sum_{j=1}^v \beta^r_j \geq -a_1  \pi_{1,v} + \sum_{j=2}^{\min(v,N+1)} \big ( a_{j-1}-a_j \big) \pi_{j,v} + \sum_{i=1}^N \sigma_{0,i,1,v}^r \nonumber \\
& \ \qquad \   + \sum_{j=2}^{\min(v,N+1)} \sum_{i=1}^N \sum_{i' \neq i} \phi_{i,i',j,v}^r  + \sum_{j=2}^{\min(v,N+1)} \sum_{i=1}^N \zeta_{i,j,v}^r, \ \
 \forall v \in [2, N+2]_{\Z}, \label{W-DHRAS_MILP_Con3}  \\
& \ \sum_{j=k}^v\beta^r_j \geq \sum_{j=k}^{\min(v,N+1)} \big ( a_{j-1}-a_j\big)  \pi_{j,v}+\sum_{j=k}^{\min(v,N+1)} \sum_{i=1}^N \zeta_{i,j,v}^r  \nonumber \\
& \ \qquad \   + \sum_{j=k}^{\min(v,N+1)} \sum_{i=1}^N \sum_{i' \neq i} \phi_{i,i',j,v},\quad  \forall k\in[2,N+1]_\Z,\,\forall v\in[k,N+2]_\Z , \label{W-DHRAS_MILP_Con4} \\
& \ \beta^r_{N+2} \geq 0 , \label{W-DHRAS_MILP_Con5} \\
& \ u^r_{i,0} \geq \lambda\that^r_{i,0},\quad  u^r_{i,0} \geq \lambda\tub_{i,0}-\rho(\tub_{i,0}-\that^r_{i,0}),\quad\forall i\in[N], \\
& \ u_{0,i,1,v}^r  \geq (\pi_{1,v}+\lambda)\tlb_{0,i}-\rho(\that_{0,i}^r-\tlb_{0,i}),\quad  u_{0,i,1,v}^r  \geq (\pi_{1,v}+\lambda)\that_{0,i}^r,\nonumber\\
& \ u_{0,i,1,v}^r  \geq (\pi_{1,v}+\lambda)\tub_{0,i}-\rho(\tub_{0,i}-\that^r_{0,i}),\quad \forall i\in[N],\, v\in[N+2], \\
& \ u_{i,i',j,v}^r   \geq (\pi_{j,v}+\lambda)\tlb_{i,i'}-\rho(\that_{i,i'}^r-\tlb_{i,i'}),\quad u_{i,i',j,v}^r   \geq (\pi_{j,v}+\lambda)\tub_{i,i'}-\rho(\tub_{i,i'}-\that^r_{i,i'}), \nonumber\\
& \ u_{i,i',j,v}^r   \geq (\pi_{j,v}+\lambda)\that_{i,i'}^r,\quad \forall i\in[N],\, i'\in[N]\setminus\{i\},\, j\in[2,N+1]_\Z,\, v\in[j,N+2]_\Z ,\\
& \ \nu_{i,j,v}^r    \geq \pi_{j,v}\dlb_{i}-\rho(\dhat_{i}^r-\dlb_{i}),\quad 
 \nu_{i,j,v}^r    \geq \pi_{j,v}\dub_{i}-\rho(\dub_{i}-\dhat^r_{i}), \nonumber \\
& \ \nu_{i,j,v}^r    \geq \pi_{j,v}\dhat_{i}^r,\quad \forall i\in[N],\, j\in[2,N+1]_\Z,\, v\in[j,N+2]_\Z.
\end{align}
\end{subequations}

Finally, we give tight big-M coefficients in the McCormick inequalities. From the proof of Proposition \ref{Prop4M} in \ref{Proof_Prop4M}, we have the explicit formula for the quantities $\ub^r$ and $\nub^r$ in \eqref{WDHRAS_Var}. Immediately, the lower bounds are $\ulb^r_{i,0}=\lambda\that^r_{i,0}$, $\ulb^r_{0,i,1,v}=(\pi_{1,v}+\lambda)\that^r_{0,i}$, $\ulb^r_{i,i',j,v}=(\pi_{j,v}+\lambda)\that^r_{i,i'}$ and $\nulb^r_{i,j,v}=\pi_{j,v}\dhat^r_{i}$. Since the expressions in \eqref{WDHRAS_Var} are non-increasing in $\rho$ and we have $\rho\geq 0$, the upper bounds are simply $\uub^r_{i,0}=\lambda\tub_{i,0}$, $\uub^r_{0,i,1,v}=\max\{(\pi_{1,v}+\lambda)\tlb_{0,i},(\pi_{1,v}+\lambda)\tub_{0,i}\}$, $\uub^r_{i,i',j,v}=\max\{(\pi_{j,v}+\lambda)\tlb_{i,i'},(\pi_{j,v}+\lambda)\tub_{i,i'}\}$ and $\nuub^r_{i,j,v}=\max\{\pi_{j,v}\dlb_{i},\pi_{j,v}\dub_{i}\}$.

\section{Sample averaging approximation approach} \label{apdx: SAA}
In the numerical experiments, we compare DRO models with the sample averaging approximation (SAA) approach. Suppose we have a set of $R$ scenarios $\{\xibhat^1,\dots,\xibhat^R\}$. The SAA approach is to solve the original problem by replacing the true distribution with the empirical distribution $\widehat{\Prob}_{\xib}^{R}$. That is, 
\begin{subequations}\label{model:SAA}
\begin{align}
& \min_{\xb,\ab,\ub,\wb} \  \frac{1}{R} \sum_{r=1}^R \sum \limits_{j=1}^N \left[(\cw_j w^r_j+\cu_j u^r_j \big) +\co w^r_{N+1} + \lambda A^r\right] \nonumber\\
& \ \ \ \text{s.t.}  \ \ \xb\in\calX,\quad \ab\in\calA, \nonumber \\  
&  \ \ \ \ \ \ \ \ \ w^r_1-u^r_1= \sum \limits_{i=1}^N t^r_{0,i}x_{i,1}-a_1, \nonumber\\
&   \ \ \ \ \ \ \ \ \ w^r_j-w^r_{j-1}-u^r_j=a_{j-1}-a_j+\sum_{i=1}^N d^r_ix_{i,j-1}+ \sum_{i=1}^N \sum_{i'\neq i} t^r_{i,i'} x_{i,j-1}x_{i,j}, \ \forall j\in [2, N], \nonumber\\
&  \ \ \ \ \ \ \ \ \   w^r_{N+1}-w^r_N-u^r_{N+1}= a_N-a_{N+1}+\sum_{i=1}^N d^r_ix_{i,N}, \nonumber\\
&  \ \ \ \ \ \ \ \ \ A^r= \sum \limits_{j=2}^N   \sum_{i=1}^N \sum_{i'\neq i} t^r_{i,i'} x_{i,j-1}x_{i',j}+\sum_{i=1}^N(t^r_{0,i}x_{i,1}+t^r_{i,0}x_{i,N}), \nonumber \\
&  \ \ \ \ \ \ \ \ \ (w^r_j,u^r_j)\geq 0, \ \ \forall j \in [N+1]. \nonumber
\end{align}
\end{subequations}

\section{Symmetry-breaking constraints} \label{apdx: SymBreak}
To enhance the tractability of the models, we introduce the symmetry-breaking constraints. The idea behind these constraints follows from prior appointment scheduling observations that customers of the same type (i.e., requesting the same service) have the same service time distribution and the observation from \cite{nikzad2021matheuristic} that customers within a service region form a basic unit or cluster that shares the same travel time distribution. In this case, the route within the same group does not matter since the service and travel time distributions are the same. This means, with the presence of homogeneous groups, we could be able to eliminate some of the equialent routes and hopefully, improve the model's tractability. In particular, we focus on the presence of one homogeneous group, which is common when services are provided within a service region.

Suppose that except the depot (node $0$), the $N$ customers form a homogeneous group. That is, $d_i$ are distributionally the same over $i\in[N]$ and $t_{i,i'}$ are distributionally the same over $i\in[N]$ and $i'\in[N]\setminus\{i\}$. In this case, we argue that the service provider only needs to make decision on the nodes departing from and entering to the depot. We summarize this result formally in the following lemma.

\begin{lem} \label{lem:symBreak1}
Assume that there is only one homogeneous group. For a fixed route $(i_1,\dots,i_N)$ denoted as $\xb$, i.e. $x_{i_j,j}=1$ for $j\in[N]$, define $\xb^\Pi$ by another route with $x_{\Pi(i_j),j}$, where $(\Pi(i_j))_{j\in[N]}$ is a permutation of $[N]$ with $\Pi(i_1)=i_1$ and $\Pi(i_N)=i_N$. Then, $f(\xb,\ab,\xib)$ is distributionally the same as $f(\xb^\Pi,\ab,\xib)$ for any $\ab\in\calA$.
\end{lem}

\begin{proof}
Recall that $f(\xb,\ab,\xib)$ is the second-stage cost, which is a function of the idling time $u_i$ for $i\in[N]$, the waiting time and overtime $w_i$ for $i\in[N+1]$. Define the same quantities $u^\Pi_i$ and $w^\Pi_i$ for the decision $\xb^\Pi$. We will show, by induction, that $w_j$ and $u_j$ are distributionally the same as $w^\Pi_j$ and $u^\Pi_j$ respectively. Note that $w_1$ and $u_1$ are distributionally the same as $w^\Pi_1$ and $u^\Pi_1$ respectively since by assumption, the first visiting customer is the same. Next, assume that $w_{j-1}$ and $u_{j-1}$ are distributionally the same as $w^\Pi_{j-1}$ and $u^\Pi_{j-1}$ respectively. Note that 
$$w_j=\max\left\{w_{j-1}+a_{j-1}-a_{j}+\sum_{i=1}^N d_ix_{i,j-1}+\sum_{i=1}^N\sum_{i'\ne i} t_{i,i'}x_{i,j-1}x_{i',j}, 0\right\}.$$
By induction assumption and the assumption that the distributions of travel and service times are the same, $w_j$ equals $w^\Pi_j$ in distribution. It is easy to see that, with the same argument, $u_j$ equals $u^\Pi_j$ in distribution. This completes the proof.
\end{proof}

From Lemma \ref{lem:symBreak1}, we have $\E[f(\xb,\ab,\xib)]=\E[f(\xb^\Pi,\ab,\xib)]$ for any permutation with fixed entering and departing node. As an illustration of the symmetry, consider that we have a homogeneous group of $6$ customers. Then, the routes $(2,1,3,4,6,5)$ and $(2,4,3,6,1,5)$ give the same expected cost. We can impose the lexicographic order of the route for the middle customers (not the first and the last customer), i.e., the service provider will first visit customers with a smaller index $i$. In this example, we only consider the route $(2,1,3,4,6,5)$ and eliminate any other possible permutations of the middle customers. Hence, we can introduce the following symmetry-breaking constraints.
\begin{subequations}
\begin{align}
& x_{1,j} \geq x_{1,j+1},\quad\forall j\in[2,N-2]_\Z \label{eqn:SBC1-1}\\
& x_{i,j} \leq \sum_{l=1}^{i-1} x_{l,j-1},\quad\forall i\in[2,N]_\Z,\, j\in[3,N-1]_\Z\label{eqn:SBC1-2}
\end{align}
\end{subequations}
If customer $1$ is neither the first nor the last customer, constraints \eqref{eqn:SBC1-1} enforces customer $1$ to be the second customer (see example matrix $X^1$, where each entry represents $x_{i,j}$). Otherwise, the constraint is satisfied since $x_{1,j}=0$ for all $j\in[2,N-1]_\Z$ (see $X^2$ and $X^3$). Constraints \eqref{eqn:SBC1-2} enforces the ordering of the third to the $(N-1)$-th customer. If customer $i$ is the $j$-th customer, i.e. $x_{i,j}=1$ for some $j\in[3,N-1]_\Z$, we require that a customer with index less than $i$ must be served at position $j-1$ (see $X^1$, $X^2$ and $X^3$). By imposing these constraints,  the solver solely determines the entering and departing node without optimizing any permutations for the middle customers.

$$X^1 =\begin{pmatrix} 0 & 1 & 0 & 0 & 0 & 0\\
                       1 & 0 & 0 & 0 & 0 & 0\\
                       0 & 0 & 1 & 0 & 0 & 0\\
                       0 & 0 & 0 & 0 & 0 & 1\\
                       0 & 0 & 0 & 1 & 0 & 0\\
                       0 & 0 & 0 & 0 & 1 & 0
                       \end{pmatrix} \quad 
X^2 =\begin{pmatrix}   1 & 0 & 0 & 0 & 0 & 0\\
                       0 & 1 & 0 & 0 & 0 & 0\\
                       0 & 0 & 1 & 0 & 0 & 0\\
                       0 & 0 & 0 & 1 & 0 & 0\\
                       0 & 0 & 0 & 0 & 0 & 1\\
                       0 & 0 & 0 & 0 & 1 & 0
                       \end{pmatrix} \quad 
X^3 =\begin{pmatrix}   1 & 0 & 0 & 0 & 0 & 0\\
                       0 & 0 & 0 & 0 & 0 & 1\\
                       0 & 1 & 0 & 0 & 0 & 0\\
                       0 & 0 & 1 & 0 & 0 & 0\\
                       0 & 0 & 0 & 1 & 0 & 0\\
                       0 & 0 & 0 & 0 & 1 & 0
                       \end{pmatrix}                        
                       $$

\section{Additional results on appointment time structure} \label{apdx:AT_Structure}
In this section, we provide additional results for the appointment time structure with the two other choices of $\lambda$, namely $0.5$ and $1$. Figures \ref{fig:AppointmentTime1_lam0.5} and \ref{fig:AppointmentTime1_lam1} show the inter-arrival times for the cost structure $(\cw_j,\cu_j,\co)=(2,1,20)$ while Figures \ref{fig:AppointmentTime2_lam0.5} and \ref{fig:AppointmentTime2_lam1} show the results for the cost structure $(\cw_j,\cu_j,\co)=(5,1,7.5)$. We observe similar patterns for the three choices of $\lambda$.

\begin{figure}[t]
    \hspace{-15mm}
    \includegraphics[scale=0.67]{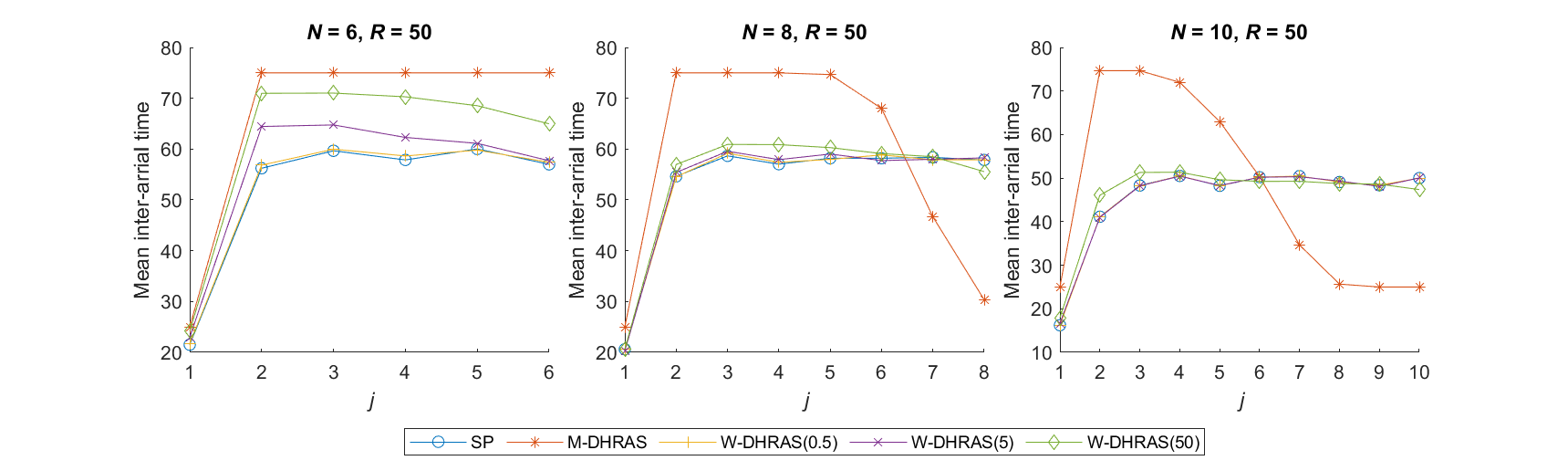}
    \caption{Mean inter-arrival times $a_j-a_{j-1}$ with $R=50$ under $(\cw_j,\cu_j,\co)=(2,1,20)$ and $\lambda=0.5$}
    \label{fig:AppointmentTime1_lam0.5}
\end{figure}
\begin{figure}[p]
    \hspace{-15mm}
    \includegraphics[scale=0.67]{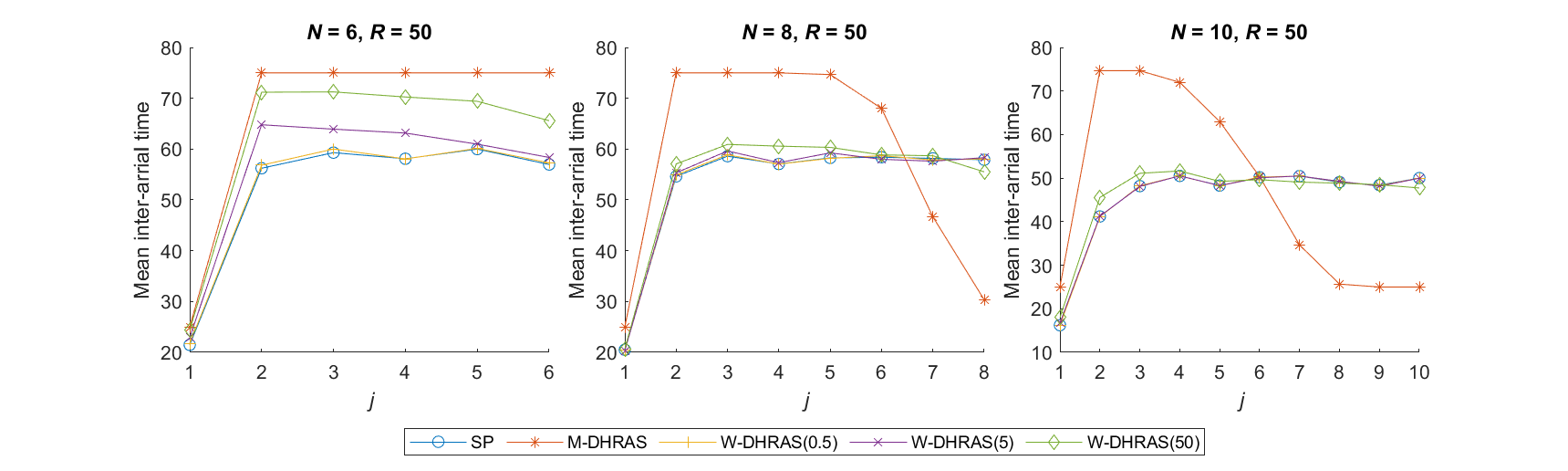}
    \caption{Mean inter-arrival times $a_j-a_{j-1}$ with $R=50$ under $(\cw_j,\cu_j,\co)=(2,1,20)$ and $\lambda=1$}
    \label{fig:AppointmentTime1_lam1}
\end{figure}
\begin{figure}[p]
    \hspace{-15mm}
    \includegraphics[scale=0.67]{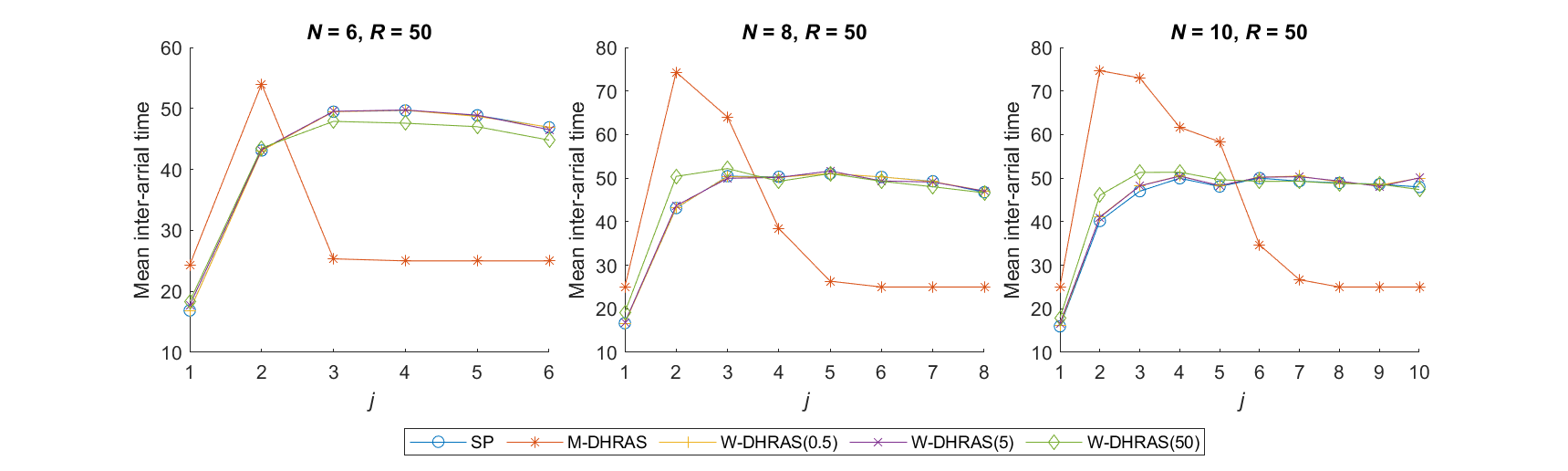}
    \caption{Mean inter-arrival  $a_j-a_{j-1}$ with $R=50$ under $(\cw_j,\cu_j,\co)=(1,5,7.5)$ and $\lambda=0.5$}
    \label{fig:AppointmentTime2_lam0.5}
\end{figure}
\begin{figure}[p]
    \hspace{-15mm}
    \includegraphics[scale=0.67]{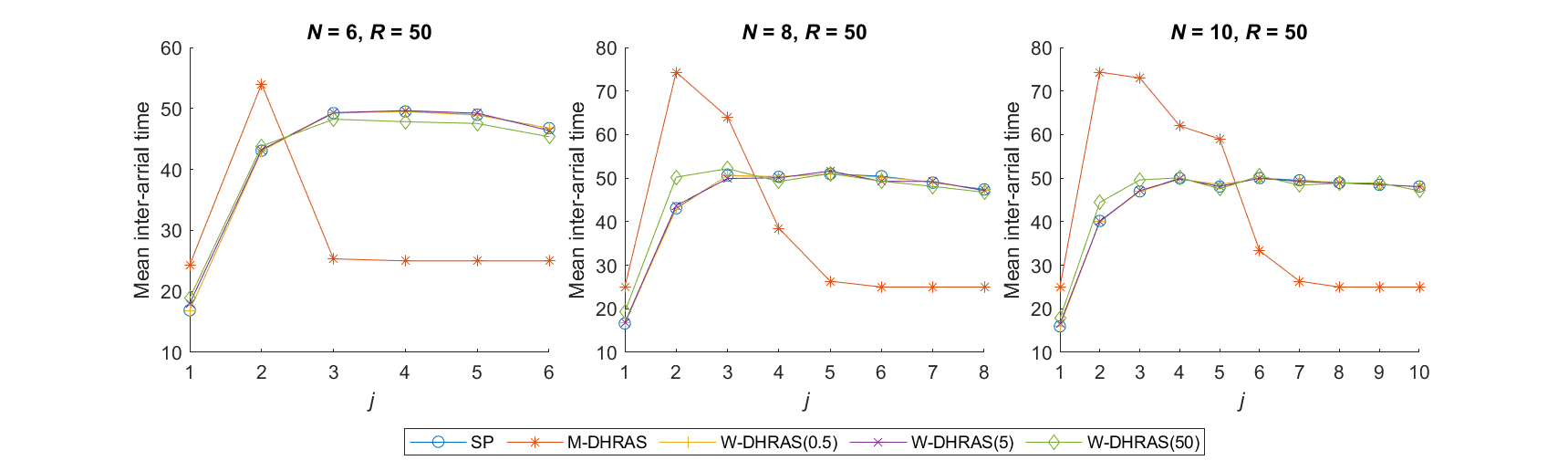}
    \caption{Mean inter-arrival times $a_j-a_{j-1}$ with $R=50$ under $(\cw_j,\cu_j,\co)=(1,5,7.5)$ and $\lambda=1$}
    \label{fig:AppointmentTime2_lam1}
\end{figure}

\section{Additional results on out-of-sample performance} \label{apdx:OutSample}

We provide additional out-of-sample performance results in this appendix. All the settings are the same as described in Section \ref{sec4:Exp} and we only present the results for $\lambda=2$ since similar patterns could be observed for other two choices of $\lambda$. Besides the distribution we used in Section \ref{sec4:Exp}, we consider two additional distributions. The following sets summarize all our testing instances.

\begin{enumerate}
    \item [Set 1.] We assume perfect information for the distributions. That is, we generate the $10,000$ samples from the same distribution we use in the optimization as discussed in Section \ref{subsection:Expt_Setup}. This simulation assumes that the data comes from the true unknown distribution.

    \item [Set 2.] In this set, we assume that we have misspecified the distribution of the travel time in the optimization. We generate $t_{i,i'}$ from $U[25,35]$ instead of $U[15,25]$. That is, the average travel time takes 10 minutes longer than usual. This situation (shift in the travel time range) might be seen in practice due to unexpected traffic congestion (e.g., caused by traffic accidents, weather conditions etc.). 

    \item [Set 3.]In this set, we assume that we have misspecified both the service and travel times distributions in the optimization. Specifically, we follow a similar out-of-sample simulation testing procedure described in \cite{wang2020distributionally} and perturb the support of the random travel and service times by a parameter $\delta$ as $[(1-\delta)$ lower bound, $(1+\delta)$ upper bound], where $\delta\in\{0.1,0.25,0.5\}$. A higher value of $\delta$ corresponds to a higher variation level.
    
    \item [Set 4.] In this set, we assume that we have misspecified the service time distribution in the optimization (i.e., lognormal is not the true distribution).  We simulate $d_i$ from the U-shaped beta distribution $B(0.5,0.5)$ on $[10,50]$ (\citealp{jiang2019data}), which has the same mean and range as the in-sample distribution.

    \item [Set 5.] We generate data similar to Set 3 but only changing the service time. 
\end{enumerate}

We show the results of the $9$ out-of-sample testing sets in each figure in the following order: (from row 1 to row 3, left to right) Set 1, Set 2, Set 3 with $\delta\in\{0.1,0.25,0.5\}$, Set 4 and Set 5 with $\delta\in\{0.1,0.25,0.5\}$. Figures \ref{fig:OC_8} and \ref{fig:OC_10} show the corresponding performance for $N=8$ and $N=10$ under cost structure (a). We only present the choice of $\epsilon\in\{0.05,0.5,5,50\}$ such that W-DHRAS gives the best overall performance. In this case, we choose $\epsilon=0.05$ and $\epsilon=50$ for $N=8$ and $N=10$ respectively.

Similar patterns could be observed as in the case with $N=6$. First, M-DHRAS model is the most conservative model, which yields the highest out-of-sample cost for almost all the cases. Second, we observe that W-DHRAS shares a similar out-of-sample performance as SP (mainly due to a small choice of $\epsilon$) for $N=8$ while for $N=10$, the W-DHRAS model consistently produces a lower out-of-sample cost than the SP model. Indeed, for $N=10$, the operator has a very tight schedule. A small deviation from the in-sample distribution (from the empirical data) results in a significant amount of overtime. Therefore, W-DHRAS could perform better with its ability to hedge against unfavorable scenarios. 

Figures \ref{fig:OC2_6} to \ref{fig:OC2_10} show the results under cost structure (b) with $\lambda=2$ for the three choices of $N$. We only present the W-DHRAS model with the best performing $\epsilon$, which are $0.05$, $0.5$ and $50$ respectively. Similar patterns are observed as in cost structure (a). First, M-DHRAS gives the largest out-of-sample costs for almost all the cases. Second, we observe that W-DHRAS and SP have similar performance with $N=6$ and $N=8$, though the performance of W-DHRAS is slightly better than SP when the sample size is small. Finally, for $N=10$, W-DHRAS consistently performs the best among the three models. 

\begin{figure}[p]
    \hspace{-15mm}
    \includegraphics[scale=0.67]{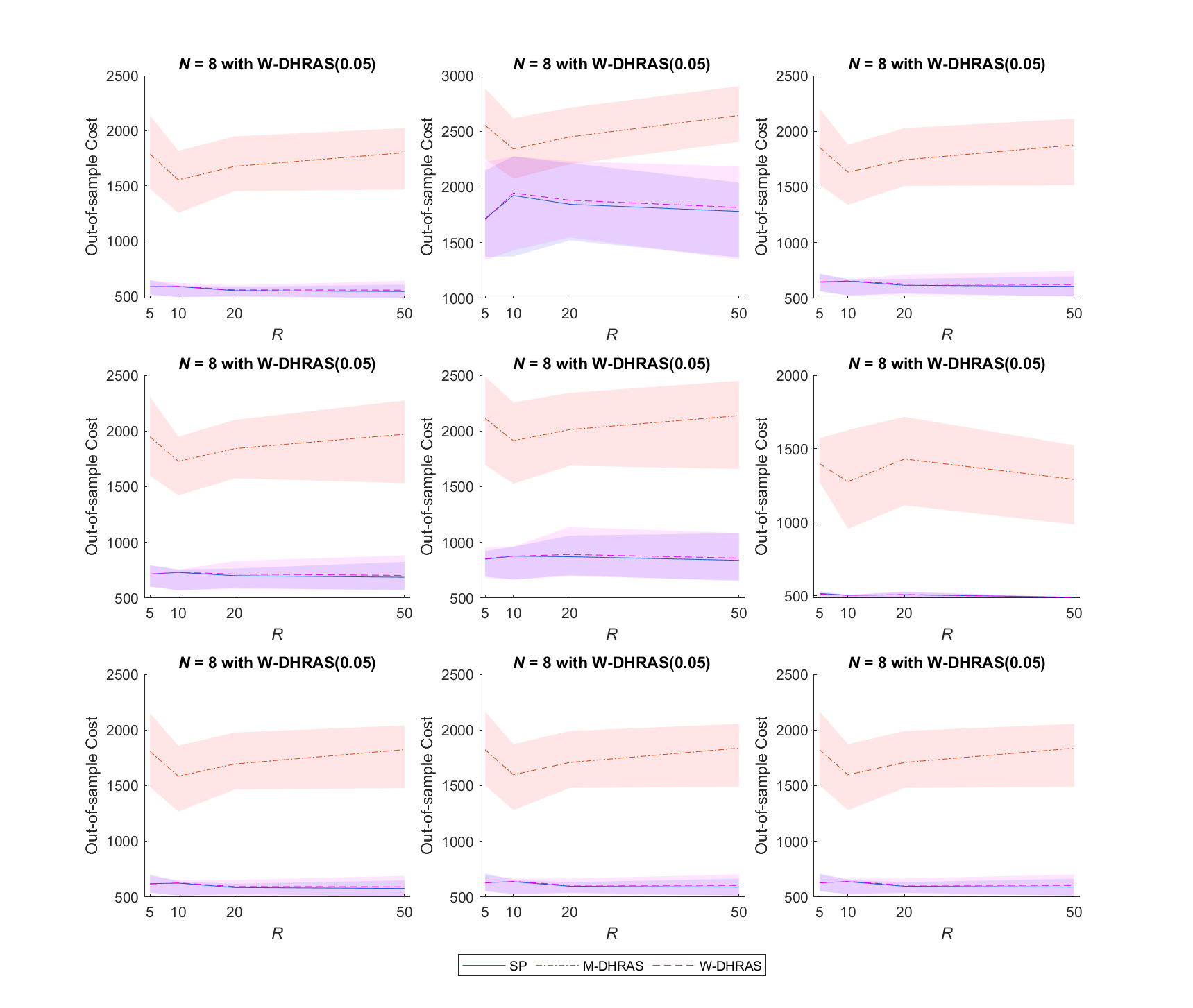}
    \caption{Out-of-sample performance for $N=8$ with cost structure (a) and $\lambda=2$}
    \label{fig:OC_8}
\end{figure}
\begin{figure}[p]
    \hspace{-15mm}
    \includegraphics[scale=0.67]{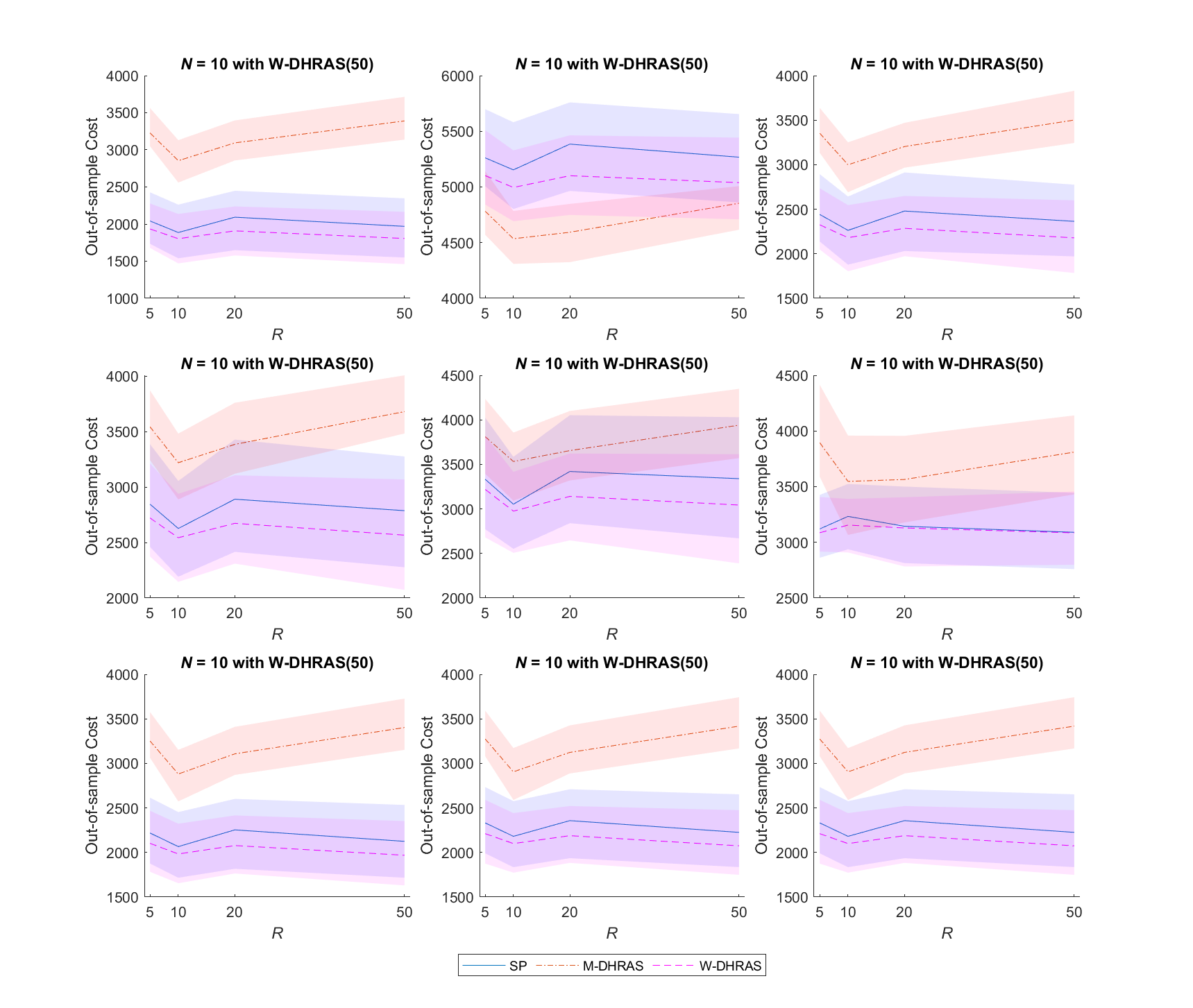}
    \caption{Out-of-sample performance for $N=10$ with cost structure (a) and $\lambda=2$}
    \label{fig:OC_10}
\end{figure}
\begin{figure}
    \hspace{-15mm}
    \includegraphics[scale=0.67]{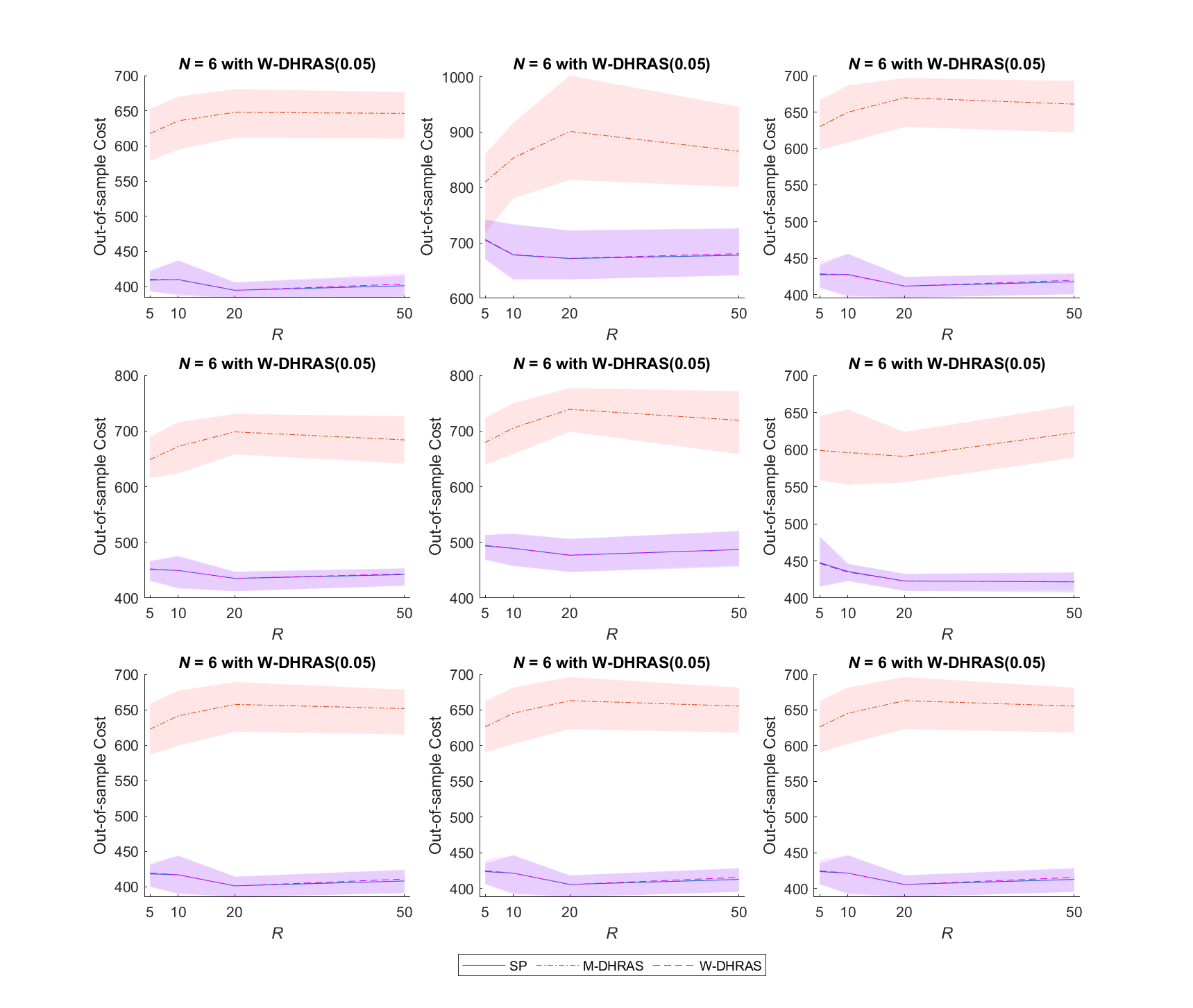}
    \caption{Out-of-sample performance for $N=6$ with cost structure (b) and $\lambda=2$}
    \label{fig:OC2_6}
\end{figure}
\begin{figure}
    \hspace{-15mm}
    \includegraphics[scale=0.67]{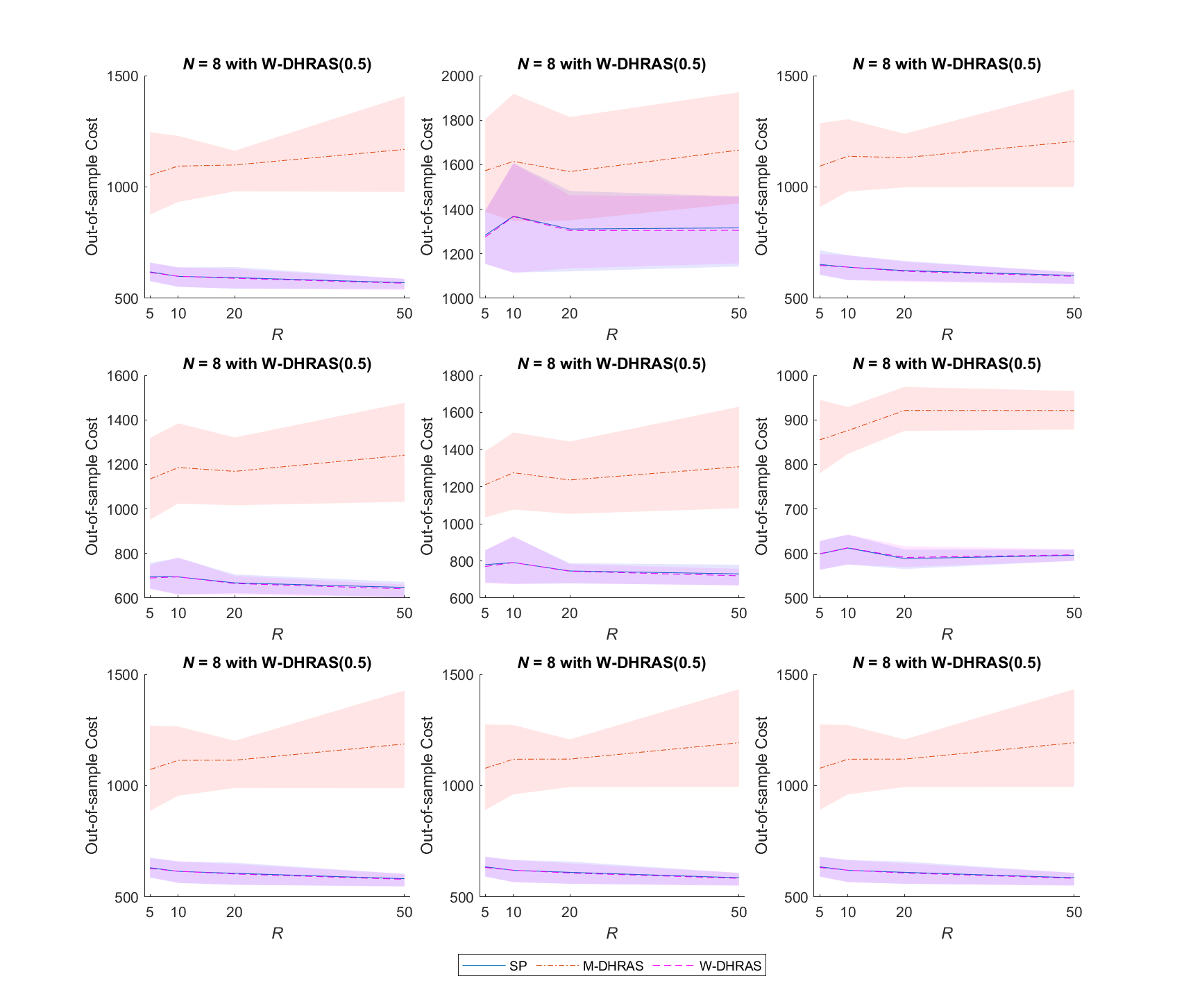}
    \caption{Out-of-sample performance for $N=8$ with cost structure (b) and $\lambda=2$}
    \label{fig:OC2_8}
\end{figure}
\begin{figure}
    \hspace{-15mm}
    \includegraphics[scale=0.67]{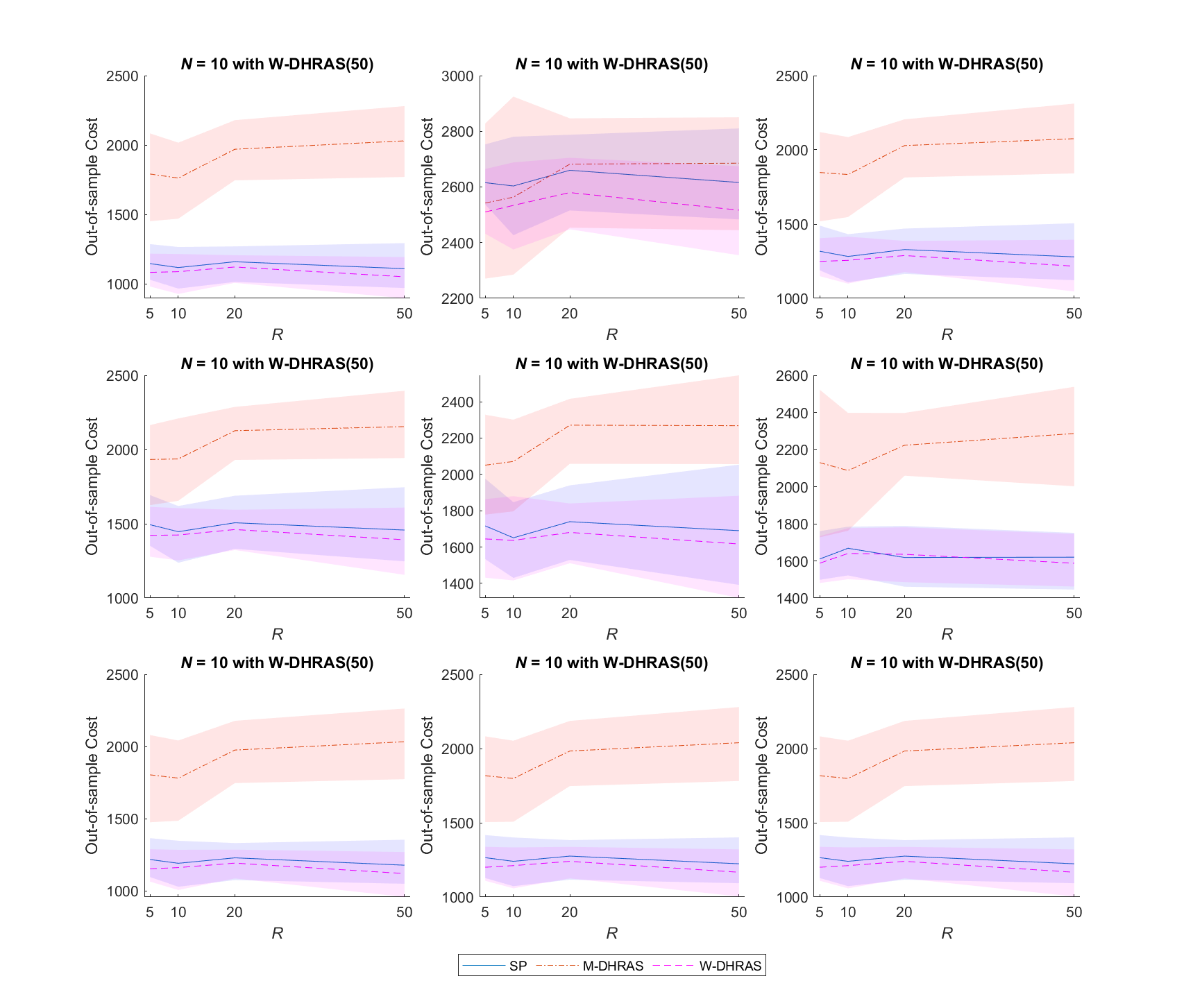}
    \caption{Out-of-sample performance for $N=10$ with cost structure (b) and $\lambda=2$}
    \label{fig:OC2_10}
\end{figure}

\section{Additional results on computational time} \label{apdx:SolTime}
We provide additional computational time results for the SP and M-DHRAS models. We solve $30$ instances of the models with different number of scenarios $R\in\{5,10,20,50,100,200,500\}$ and different number of customers $N\in\{6,8,10,15\}$ and obtain the average solution time. All the remaining experiment settings follow Section \ref{subsection:Expt_Setup}. We only show the solution time when the transportation cost $\lambda=2$ since they are similar for the other two choices of $\lambda$. Table \ref{table:CPU_Time_SP} presents the results for the SP model. We observe that all the instances could be solved efficiently (typically less than one minute).  We observe that the SP can solve all instances quickly in less than a minute.  This demonstrates that our SP model can solve large instances with an unrealistic number of customers even with many scenarios. Table~\ref{table:CPU_Time_Mom} shows the computational time for M-DHRAS for different numbers of customers. Recall that this model is not a sample-based model, i.e., it does not depend on the scenarios directly, and hence, we only report the time with $R=5$.  We observe that the M-DHRAS model can solve all instances efficiently in a few seconds. In Table \ref{table:CPU_Time_Wass}, we also present the computational time for W-DHRAS with $N=15$ under limited data setting $R\in\{5,10\}$ (where W-DHRAS could yield a better performance). We observe that the model could be solved in a reasonable time (within 4 minutes). In Table \ref{table:CPU_Time_Wass}, we also present the computational time for W-DHRAS with $N=15$ under limited data setting $R\in\{5,10\}$ (where W-DHRAS could yield a better performance). We observe that the model can solve these instances in a reasonable time within 4 minutes.


\begin{table}[t!]\centering\small
\ra{0.9}  
\caption{CPU time in seconds for solving the SP model} 
\begin{tabular}{@{}lrrrrrrr@{}} \toprule
SP  & \multicolumn{7}{c}{$(\cw_j,\cu_j,\co)=(2,1,20)$}  \\
\cmidrule{2-8} 
CPU Time (in s) & $R=5$ & $R=10$ & $R=20$ & $R=50$ & $R=100$ & $R=200$ & $R=500$\\
\cmidrule{1-8} 
$N = 6$     & 0.25  & 0.30  & 0.37   & 0.42    & 0.72  & 1.16  & 5.86   \\
$N = 8$     & 0.28  & 0.33  & 0.39   & 0.76    & 1.04  & 2.65  & 13.21  \\
$N = 10$    & 0.38  & 0.43  & 0.54   & 1.04    & 2.70  & 8.58  & 43.19  \\
$N = 15$    & 0.49  & 0.75  & 1.06   & 1.82    & 4.13  & 12.03 & 65.91  \\
\toprule
 SP  & \multicolumn{7}{c}{$(\cw_j,\cu_j,\co)=(1,5,7.5)$}\\
\cmidrule{2-8}  
CPU Time (in s) & $R=5$ & $R=10$ & $R=20$ & $R=50$ & $R=100$ & $R=200$ & $R=500$\\
\cmidrule{1-8} 
$N = 6$     & 0.23  & 0.21  & 0.33   & 0.44    & 0.75  & 1.37  & 7.22   \\
$N = 8$     & 0.28  & 0.32  & 0.39   & 0.79    & 1.07  & 3.01  & 16.24  \\
$N = 10$    & 0.35  & 0.42  & 0.54   & 1.07    & 2.13  & 6.84  & 36.11  \\
$N = 15$    & 0.51  & 0.77  & 1.06   & 1.74    & 4.00  & 11.13 & 63.83  \\
\bottomrule
\end{tabular}
\label{table:CPU_Time_SP}
\end{table}

\begin{table}[t!]\centering\small
\ra{0.9}  
\caption{CPU time in seconds for solving the M-DHRAS model.} 
\begin{tabular}{@{}lrr@{}} \toprule
M-DHRAS  & \multicolumn{2}{c}{$(\cw_j,\cu_j,\co)$}  \\
\cmidrule{2-3}
CPU Time (in s) & $(2,1,20)$ & $(1,5,7.5)$\\
\cmidrule{1-3} 
$N = 6$  & 0.46 & 0.46 \\
$N = 8$  & 1.11 & 0.90 \\
$N = 10$ & 2.05 & 1.79 \\
$N = 15$ & 8.85 & 7.97 \\
\bottomrule
\end{tabular}
\label{table:CPU_Time_Mom}
\end{table}

\begin{table}[t!]\centering\small
\ra{0.9}  
\caption{CPU time in seconds for solving the W-DHRAS model with $N=$15 customers.} 
\begin{tabular}{@{}lrr|rr@{}} \toprule
W-DHRAS ($N=15$)  & \multicolumn{2}{c}{$(\cw_j,\cu_j,\co)=(2,1,20)$} & \multicolumn{2}{c}{$(\cw_j,\cu_j,\co)=(1,5,7.5)$}  \\
\cmidrule{2-5}
CPU Time (in s) & $R=5$ & $R=10$ & $R=5$ & $R=10$\\
\cmidrule{1-5} 
W-DHRAS(0.5) & 52.16 & 203.57 & 47.71 & 210.13 \\
W-DHRAS(5)   & 54.86 & 208.56 & 48.78 & 210.00 \\
W-DHRAS(50)  & 60.57 & 233.40 & 54.60 & 239.81 \\
\bottomrule
\end{tabular}
\label{table:CPU_Time_Wass}
\end{table}

\newpage

\end{document}